\documentclass[11pt, reqno]{amsart}
\usepackage{amsmath, amsfonts, amssymb, amsbsy, bigstrut, graphicx, enumerate,  upref, color, comment, xcolor}

\usepackage{pstricks}

\usepackage[breaklinks]{hyperref}

\newcommand{\ms}{\mathcal{S}}

\newcommand{\tf}{\tilde{f}}

\newtheorem{thm}{Theorem}[section]
\newtheorem{lmm}[thm]{Lemma}
\newtheorem{cor}[thm]{Corollary}
\newtheorem{prop}[thm]{Proposition}

\newcommand{\lx}{\ell(X)}

\newcommand{\cp}{\mathcal{P}}
\newcommand{\cs}{\mathcal{S}}

\newcommand{\mx}{\mathcal{X}}

\newcommand{\ra}{\rightarrow}
\newcommand{\rr}{\mathbb{R}}
\newcommand{\smallavg}[1]{\langle #1 \rangle}

\newcommand{\tr}{\operatorname{Tr}}

\newcommand{\zz}{\mathbb{Z}}

\newcommand{\fs}{\mathbb{S}}
\newcommand{\fd}{\mathbb{D}}
\newcommand{\ftw}{\mathbb{T}}
\newcommand{\fst}{\mathbb{M}}

\numberwithin{equation}{section}

\begin{document}
\title[The $1/N$ expansion at strong coupling]{The $1/N$ expansion for $SO(N)$ lattice gauge theory at strong coupling}
\author{Sourav Chatterjee}
\address{\newline Department of Statistics \newline Stanford University\newline Sequoia Hall, 390 Serra Mall \newline Stanford, CA 94305\newline \textup{\tt souravc@stanford.edu}}
\author{Jafar Jafarov}
\address{\newline Department of Mathematics \newline Stanford University\newline 450 Serra Mall, Building 380 \newline Stanford, CA 94305\newline \textup{\tt jafarov@stanford.edu}}
\thanks{Research partially supported by NSF grant DMS-1441513}

\begin{abstract}
The $1/N$ expansion is an asymptotic series expansion for certain quantities in large-$N$ lattice gauge theories. This article gives a rigorous formulation and proof of the $1/N$ expansion for Wilson loop expectations in $SO(N)$ lattice gauge theory in the strong coupling regime in any dimension. The terms in the expansion are expressed as sums over trajectories of strings in a lattice string theory, establishing an explicit gauge-string duality. The trajectories trace out surfaces of genus zero for the first term in the expansion, and surfaces of higher genus for the higher terms. %The original $1/N$ expansion was a perturbative weak coupling expansion based on Feynman diagram calculations. The result of this paper, on the other hand, is for the strong coupling regime, is non-perturbative and does not make use of Feynman diagrams.
\end{abstract}
\maketitle

%\tableofcontents
%\vskip.5in
\section{Introduction}\label{intro}
%\noindent
%\textbf{- Definition of the model and object of interest (Wilson loops)\\
%- Brief discussion of Sourav's result without full details.\\
%- What's in this paper.\\}
Quantum Yang--Mills theories are the basic building blocks of the Standard Model of quantum mechanics. In an effort to gain a better theoretical understanding of quantum Yang--Mills theories, Wilson~\cite{wilson74} introduced a discretized version of these theories that are now known as lattice gauge theories. Lattice gauge theories are mathematically well-defined objects, which are supposed to converge to quantum Yang--Mills theories in the scaling limit as the lattice spacing tends to zero. The existence of the scaling limit, however, is a famous open problem in mathematical physics that was posed as one of the seven millennium prize problems by the Clay Institute in the year 2000.

A lattice gauge theory requires two ingredients: a compact Lie group,  called the gauge group of the theory, and a lattice, usually taken to be $\zz^d$ for some $d\ge 2$. For the convenience of the reader, we will now give a quick definition of lattice gauge theory for the gauge group $SO(N)$. 

Let $d\geq 2$ be a positive integer and let $E$ be the directed nearest-neighbor edges of $d$-dimensional integer lattice $\zz^d$. For any edge $e \in E$ let $u(e)$ and $v(e)$ respectively be the starting and ending points of $e$. We say an edge $e$ is positively oriented if $u(e)$ is lexicographically smaller than $v(e)$, and negatively oriented otherwise. Let $E^+$ and $E^-$ be the sets of positively and negatively oriented edges in $E$. If $e=(u,v)$, we will denote the edge $(v,u)$ by $e^{-1}$. 

A sequence of edges $l=e_1e_2\cdots e_n$ is called a closed loop if $v(e_i)=u(e_{i+1})$ for each $i<n$ and $v(e_n)=u(e_1)$.  A closed loop of length four  is called a plaquette if $e_i\neq e_j^{-1}$ for all $i,j$. Let $\cp$ be the set of all plaquettes. A plaquette $p=e_1e_2e_3e_4$ is called positively oriented if $u(e_1)$ is lexicographically the smallest among all vertices in $p$ and $v(e_1)$ is the second smallest. Let $\cp^+$ be the set of all positively oriented plaquettes. %It is easy to see that for any plaquette $p$ wither $p$ or $p^{-1}$ is in $\cp^+$. %Also let $\cp^+(e)$ be the set of all positively oriented plaquettes that contain either $e$ or $e^{-1}$.

Let $SO(N)$ be the group of $N\times N$ orthogonal matrices with determinant 1,  and let $\sigma_N$ be the Haar measure on $SO(N)$. Let $\Lambda$ be a finite subset of $\zz^d$ and $\beta$ be a real number. Let $E^+_{\Lambda}$ be the set of all positively oriented edges with both endpoints in $\Lambda$. Take a configuration $Q$ of $SO(N)$ matrices indexed by the elements of $E_\Lambda^+$. For each $e\in E_\Lambda^+$, let $Q_{e^{-1}} := Q_e^{-1}$.  For any plaquette $p=e_1e_2e_3e_4$ let $Q_p:=Q_{e_1}Q_{e_2}Q_{e_3}Q_{e_4}$. Define a probability measure $\mu_{\Lambda, N,\beta}$ on the set of all such configurations as 
\[
d\mu_{\Lambda, N, \beta}(Q):=Z^{-1}_{\Lambda, N, \beta}\exp\biggl(N\beta\sum_{p\in \cp^{+}_{\Lambda}}\tr(Q_p)\biggr)\prod\limits_{e\in E^{+}_{\Lambda}}d\sigma_N(Q_e)\,,
\]
where $\cp^{+}_{\Lambda}$ is the set of all positively oriented plaquettes lying completely inside $\Lambda$ and $Z_{\Lambda, N, \beta}$ is the normalizing constant (usually called the partition function). This probability measure defines a lattice gauge theory on $\Lambda$ for the group $SO(N)$. The number $\beta$ is called the inverse coupling constant of the model. In the lingo of lattice gauge theories, ``strong coupling'' means small $\beta$, and ``weak coupling'' means large $\beta$.

For any function $f$ defined on the set of configurations, denote the expectation of $f$ with respect to above probability measure by $\smallavg{f}_{\Lambda, N, \beta}$. That is,
\begin{equation}\label{expdef}
\smallavg{f}_{\Lambda, N, \beta} := \int f(Q) d\mu_{\Lambda, N, \beta}(Q)\, .
\end{equation}
We will omit $\Lambda$, $N$, and $\beta$ from the subscript and write $\smallavg{f}$ whenever they are clear from the context.

The primary objects of interest in lattice gauge theories are Wilson loop variables and their expectations. For any closed loop $l=e_1e_2\ldots e_n$  lying completely inside $\Lambda$, the Wilson loop variable $W_l$ is defined as 
\[
W_{l}:=\tr(Q_{e_1}Q_{e_2}\cdots Q_{e_n})
\]
and its expectation $\smallavg{W_l}$ is defined using \eqref{expdef}. %If $l$ is the null loop, denoted by $\emptyset$, then let $W_{\emptyset}=N$. 

The goal of this paper is to understand some aspects of lattice gauge theories {\it before} taking the continuum limit. This has been a challenging field in itself for more than forty years. A famous open problem is the mass gap conjecture, which claims that if $p$ and $q$ are two plaquettes, then the correlation between the random variables $W_p$ and $W_q$ decays exponentially in the distance between $p$ and $q$, for any value of the inverse coupling strength $\beta$. Another important conjecture is quark confinement, which claims that as the loop $l$ gets larger, $-\log\smallavg{W_l}$ increases at least as fast as the minimum surface area enclosed by $l$.  

Some of the earliest rigorous results about lattice gauge theories were obtained by L\"uscher~\cite{luscher77}, who constructed strictly positive self-adjoint transfer matrices for a large class of lattice gauge theories, and by Osterwalder and Seiler~\cite{osterwalderseiler78}, who proved the quark confinement and mass gap properties at sufficiently strong coupling. Guth~\cite{guth80} proved that in four-dimensional $U(1)$ lattice gauge theory, quark confinement breaks down at sufficiently weak coupling. A simpler proof of Guth's theorem was given by Fr\"ohlich and Spencer~\cite{frohlichspencer82}. G\"opfert and Mack~\cite{gopfertmack82} proved the deep result that three-dimensional $U(1)$ lattice gauge theory is quark confining at all coupling strengths. For some recent developments and further pointers to the mathematical literature on lattice gauge theories --- especially the substantial literature on continuum limits, which is not reviewed here --- see \cite{chatterjee15, chatterjee16}.

A physics method for performing theoretical computations in lattice gauge theories that is particularly relevant for this paper was proposed by 't Hooft~\cite{thooft74}, who considered $U(N)$ lattice gauge theory for large $N$, and calculated an asymptotic series expansion for the partition function of the theory in powers of $1/N$. This is known as 't Hooft's $1/N$ expansion. 't Hooft's work revealed a connection between large matrix integrals and the enumeration of planar maps. This connection has been subsequently deeply explored by mathematicians working in a variety of areas. One of the earliest attempts at making rigorous mathematics out of 
't Hooft's idea was due to Ercolani and McLaughlin~\cite{ercolanimclaughlin03, ercolanimclaughlin08}. Alice Guionnet and her collaborators have developed the map enumeration idea and its implications in random matrix theory and free probability in a long sequence of papers~\cite{collinsetal09, guionnet04, guionnet06, guionnetetal12, guionnetmaida05,  guionnetnovak14, guionnetzeitouni02}. In another direction, the 't Hooft expansion has been used in developing the theory of topological recursion and its applications in enumerative geometry by Eynard~\cite{eynard03, eynard05} and Eynard and Orantin~\cite{eynardorantin07, eynardorantin09}.

The goal of this paper is to give a complete rigorous statement and proof of the $1/N$ expansion for Wilson loop expectations in $SO(N)$ lattice gauge theory at strong coupling, in any dimension. More precisely, we will present a sequence of functions $f_0, f_1,\ldots$ on the space of all closed loops in $\zz^d$ such that for any closed loop $l$ and any $k\ge 0$,
\begin{equation}\label{1nexp}
\frac{\smallavg{W_l}}{N} = f_0(l) + \frac{1}{N}f_1(l) + \frac{1}{N^2}f_2(l) + \cdots + \frac{1}{N^{k}}f_k(l) + o\biggl(\frac{1}{N^{k}}\biggr)
\end{equation}
as $N\ra\infty$, provided that $|\beta|$ is small enough (depending on $k$). Here $\smallavg{W_l}$ is the expectation of the Wilson loop variable associated with the loop $l$ in $SO(N)$ lattice gauge theory on $\zz^d$ at inverse coupling strength $\beta$, as defined previously. Partial progress on this question was made in \cite{chatterjee15}, where the function $f_0$ was identified as a sum over trajectories in a certain kind of string theory on the lattice $\zz^d$. This result gave a precise mathematical meaning to a simple instance of what is broadly known as ``gauge-string duality'' in the physics literature. Gauge-string duality refers to a conjectured duality between string theories and quantum gauge theories that had been floating around in the physics literature ever since the work of 't Hooft~\cite{thooft74}, before being transformed into an enormously popular area of research after the seminal work of \cite{maldacena97}. The topic, however, is not yet a part of rigorous mathematics except for the paper \cite{chatterjee15} cited above, as far as we know. 

This paper develops the program initiated in \cite{chatterjee15}, by identifying each $f_j$ in the expansion \eqref{1nexp} as a sum over trajectories in a lattice string theory. There are two major differences with the earlier work. First, the proofs are based on significantly more complicated three-fold induction arguments than the proofs in \cite{chatterjee15}, which used two-fold induction. The second and more important point is that while the string theory used to define $f_0$ in \cite{chatterjee15} allowed strings to only split and deform, thereby tracing out surfaces of genus zero, the string theory used to define $f_j$ for a general $j$ allows strings that have previously split to merge together at a later time. This has the consequence that the surfaces traced out by these strings can have genus greater than zero, which are substantially more difficult to handle. %This introduces significant new complications.

The proofs in this paper, as in the previous one, start with a rigorous version of a set of equations derived by Makeenko and Migdal~\cite{makeenkomigdal79}. These are special cases of the class of equations that are broadly known as Schwinger--Dyson equations in the random matrix literature. The validity of the  Makeenko--Migdal equations in the continuum limit of two-dimensional lattice gauge theories was rigorously established by L\'evy~\cite{levy11}, with subsequent simpler proofs and extensions by Dahlqvist~\cite{dahlqvist} and Driver et al.~\cite{driveretal2, driveretal}. 

In dimensions higher than two, the Makeenko--Migdal equations are notoriously difficult to analyze. In the physics literature, a famous attempt at simplifying the Makeenko--Migdal equations for higher dimensional lattice gauge theories was made by Eguchi and Kawai~\cite{eguchikawai82}. There were some problems with the Eguch--Kawai reduction, which were addressed by Bhanot, Heller and Neuberger~\cite{bhanotetal82}, and later by Gonzalez-Arroyo and Okaway~\cite{gonzalez83}, who proposed the twisted Eguchi--Kawai model. In recent years, there has been a spate of new breakthroughs in this area, initiated by the works of Kovtun, \"Unsal and Yaffe~\cite{kovtunetal} and \"Unsal and Yaffe~\cite{uy08}. For a survey of these recent developments, see Dunne and \"Unsal~\cite{du16}. %On the mathematical side, the only rigorously proven version of the Makeenko--Migdal equations for lattice gauge theories in dimensions higher than two are the ones established in \cite{chatterjee15}, as far as we are aware of. The main insight in \cite{chatterjee15} was that to get a useful set of equations, one needs to work with expectations of products of Wilson loop variables rather than single loops.

On the rigorous mathematical side, the only formulation and proof of the Makeenko--Migdal equations in dimensions higher than two is the one given in  \cite{chatterjee15}, where the key new idea was that instead of considering the expectation of a single Wilson loop variable, one should consider expectations of products of Wilson loop variables and develop recursive equations for these expectations. The second step was to write down the solutions of  these equations as sums over trajectories in a string theory on the lattice. The strong coupling condition was required to ensure the convergence of this infinite sum. The technical aspects of carrying out this program, however, are quite complex. The basic steps, required for deriving the formula for the first term in the $1/N$ expansion, were executed in \cite{chatterjee15}. The present paper develops this further and gets the formulas for all the terms in the $1/N$ expansion. It may be worth emphasizing here that this paper is not merely a simple extension of \cite{chatterjee15}; substantial complications arise in going beyond the first term of the expansion.

The next section introduces the lattice string theory that describes the terms in the $1/N$ expansion. The main result is presented in Section \ref{result}.

%\vskip.5in

\section{A string theory on the lattice}\label{stringsec}
For the reader's convenience, we will now restate the definition of the lattice string theory that was defined in \cite{chatterjee15}, together with some additional definitions that are needed for this manuscript. Recall that $E$ denotes the set of edges of $\zz^d$. A path $\rho$ in $\zz^d$ is defined to be a sequence of edges $\rho = e_1e_2\cdots e_n\in E$ such that $v(e_i)=u(e_{i+1})$ for $i=1,2,\ldots, n-1$. We will say that $\rho$ has length $|\rho| = n$. The path $\rho$ will be called closed if $v(e_n)= u(e_1)$. We also define a null path, $\emptyset$, that has no edges and therefore has length zero. By definition, the null path is closed. The ``inverse'' of the path $\rho$ is defined as 
\[
\rho^{-1}:= e_n^{-1}e_{n-1}^{-1}\cdots e_1^{-1}\, .
\]
If $\rho' = e_1'\cdots e_m'$ is another path such that $v(e_n)=u(e_1')$, then the concatenated path $\rho\rho'$ is defined~as 
\[
\rho\rho' := e_1\cdots e_n e_1'\cdots e_m'\,.
\]
If $\rho=e_1\cdots e_n$ is a closed path, we will say that another closed path $\rho'$ is cyclically equivalent to $\rho$ if 
\[
\rho' = e_i e_{i+1}\cdots e_n e_1e_2\cdots e_{i-1}
\]
for some $2\le i\le n$. The equivalence classes will be called cycles. The length of a cycle $l$, denoted by $|l|$, is defined to be equal to the length of any of the closed paths in $l$. The null cycle is defined to be the equivalence class containing only the null path, and its length is defined to be zero.

If $\rho = e_1\cdots e_n$ is a path, then $e_k$ will be called the edge at the $k^{\mathrm{th}}$ location of $\rho$. For cycles, the first edge is defined by some arbitrary rule. Once the first edge is defined, the definition of the $k^{\mathrm{th}}$  edge is automatic. %If $e_1\cdots e_n$ is a closed path and $l$ is the cycle containing this path, we will often forget about the distinction between the two and simply write $l = e_1\cdots e_n$. %If $e_k$ is the $k$th edge of a cycle $l$ of length $n$, the `representative path' of the $l$ is the path $e_1e_2\cdots e_n$. We will often write $l=e_1e_2\cdots e_n$.

We will say that a closed path $\rho = e_1e_2\cdots e_n$ has a backtrack at location $i$ if $e_{i+1}=e_i^{-1}$, where $e_{n+1}$ means $e_1$. %In a closed path of length $n$, a backtrack that occurs at some location $i\le n-1$ will be called an interior backtrack, and a backtrack at location $n$ will be called a terminal backtrack. 
Given a path $\rho=e_1e_2\cdots e_n$ that has a backtrack at location $i$, the path obtained by erasing this backtrack is defined as
\[
\rho' := e_1e_2\cdots e_{i-1}e_{i+2}\cdots e_n\, .
\]
It is easy to check that $\rho'$ is indeed a path, and it is closed if $\rho$ is closed. If  $\rho$ is closed and has a backtrack at location $n$, then backtrack erasure at location $n$ results in
\[
\rho' := e_2e_3\cdots e_{n-1}\, .
\]
Again, it is easy to check that $\rho'$ is a closed path. If we start with a closed path $\rho$ and keep erasing backtracks successively, we must at some point end up with a closed path with no backtracks. Such a path will be called a nonbacktracking closed path, and its cyclical equivalence class will be called a nonbacktracking cycle, since each member of such an equivalence class is nonbacktracking. Nonbacktracking cycles will be called loops. This is a more precise definition of a loop than the one given in Section \ref{intro}. The null cycle is by definition a loop, called the null loop. 

It was shown in \cite{chatterjee15} that  the nonbacktracking closed path obtained by successively erasing backtracks from a closed path is unique up to cyclical equivalence, irrespective of the order in which backtracks are erased. Using this result, the nonbacktracking core of a cycle $l$ is defined to be the unique loop obtained by successive backtrack erasures until there are no more backtracks. The nonbacktracking core of a cycle $l$ will be denoted by $[l]$. Figure \ref{corefig} illustrates a nonbacktracking core obtained by succesively erasing all backtracks.

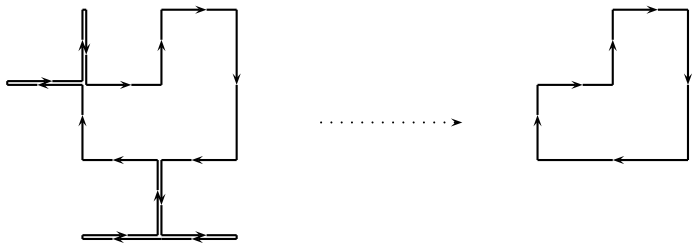
\begin{figure}[t]
%\begin{pdfpic}
\begin{pspicture}(1,.5)(11,4.5)
\psset{xunit=1cm,yunit=1cm}
\psline{->}(2,3)(2.6,3)
\psline{-}(2.6,3)(3,3)
\psline{->}(3,3)(3,3.6)
\psline{-}(3,3.6)(3,4)
\psline{->}(3,4)(3.6,4)
\psline{-}(3.6,4)(4,4)
\psline{->}(4,4)(4,3)
\psline{-}(4,3)(4, 2)
\psline{->}(4,2)(3.4,2)
\psline{-}(3.4,2)(3,2)
\psline{->}(3,2)(3,1.4)
\psline{-}(3,1.4)(3,1)
\psline{->}(3,1)(3.6,1)
\psline{-}(3.6,1)(4,1)
\psline{-}(4,1)(4,0.95)
\psline{->}(4,0.95)(3.4,0.95)
\psline{-}(3.4,0.95)(3,0.95)
\psline{->}(3,0.95)(2.35, 0.95)
\psline{-}(2.35,0.95)(1.95, 0.95)
\psline{-}(1.95, 0.95)(1.95,1)
\psline{->}(1.95,1)(2.55,1)
\psline{-}(2.55,1)(2.95, 1)
\psline{->}(2.95, 1)(2.95, 1.6)
\psline{-}(2.95, 1.6)(2.95, 2)
\psline{->}(2.95, 2)(2.35, 2)
\psline{-}(2.35, 2)(1.95, 2)
\psline{->}(1.95,2)(1.95, 2.6)
\psline{-}(1.95, 2.6)(1.95, 3)
\psline{->}(1.95, 3)(1.35, 3)
\psline{-}(1.35, 3)(.95, 3)
\psline{-}(.95, 3)(.95, 3.05)
\psline{->}(.95, 3.05)(1.55, 3.05)
\psline{-}(1.55, 3.05)(1.95, 3.05)
\psline{->}(1.95,3.05)(1.95, 3.6)
\psline{-}(1.95, 3.6)(1.95,4)
\psline{-}(1.95,4)(2,4)
\psline{->}(2,4)(2,3.4)
\psline{-}(2,3.4)(2,3)
\psline[linestyle = dotted]{->}(5, 2.5)(7, 2.5)
%\rput(5.5,2){\large $\longrightarrow$}
%%
\psline{->}(8,3)(8.6,3)
\psline{-}(8.6,3)(9,3)
\psline{->}(9,3)(9,3.6)
\psline{-}(9,3.6)(9,4)
\psline{->}(9,4)(9.6,4)
\psline{-}(9.6,4)(10,4)
\psline{->}(10,4)(10,3)
\psline{-}(10,3)(10, 2)
\psline{->}(10,2)(9,2)
\psline{-}(9,2)(8,2)
\psline{->}(8,2)(8, 2.6)
\psline{-}(8, 2.6)(8, 3)
\end{pspicture}
%\end{pdfpic}
\caption{A loop and its nonbactracking core obtained by backtrack erasures.}
\label{corefig}
\end{figure}

If $(l_1,\ldots, l_n)$ and $(l'_1,\ldots, l'_m)$ are two finite sequences of loops, we will say that they are equivalent if one can be obtained from the other by deleting and inserting null loops at various locations. Let $\ms$ denote the set of equivalence classes. Individual loops are also considered as members of $\ms$ by the natural inclusion. An element of $\ms$ will be called a loop sequence. The null loop sequence is the equivalence class of the null loop.  Any non-null loop sequence $s$ has a representative member $(l_1,\ldots, l_n)$ that has no null loops. This will be called the minimal representation of $s$. The length of  a loop sequence $s = (l_1,\ldots, l_n)$  is defined as $|s| := |l_1|+\cdots +|l_n|$. %The null loop sequence is defined to have length zero.

We will now define some operations on loops. These are rules for getting new loops out of old ones. There are five kinds of operations, called merger, deformation, splitting, twisting and inaction. The first four kinds of operations have two subtypes each, called positive and negative.

\textbf{Merger:} Let $l$ and $l'$ be two non-null loops. Let $x$ be a location in $l$ and $y$ be a location in $l'$. If $l$ contains an edge $e$ at location $x$ and $l'$  contains the same edge $e$ at location $y$, then write $l = aeb$  and $l' = ced$ (where $a,b,c,d$ are paths), and define the positive merger of $l$ and $l'$ at locations $x$ and $y$ as 
\[
l\oplus_{x,y} l' := [a e dceb]\,,
\]
where $[aedceb]$ is the nonbacktracking core of the cycle $aedceb$, and define the negative merger of $l$ and $l'$ at locations $x$ and $y$ as 
\[
l\ominus_{x,y} l' := [ac^{-1}d^{-1}b]\, .
\]
If $l$ contains an edge $e$ at location $x$ and $l'$ contains the inverse edge $e^{-1}$ at location $y$, write $l = aeb$ and $l' = ce^{-1} d$ and define the positive merger of $l$ and $l'$ at locations $x$ and $y$ as
\[
l\oplus_{x,y} l' := [ae c^{-1} d^{-1} e b]\, ,
\]
and define the negative merger of $l$ and $l'$ at locations $x$ and $y$ as
\[
l\ominus_{x,y} l' := [adcb]\, ,
\]
Figure \ref{posstfig} illustrates a positive merger and Figure \ref{negstfig} illustrates a negative merger.

%POSITIVE MERGER (e,e)
\begin{figure}[t]
%\begin{pdfpic}
\begin{pspicture}(1,1)(11,4.5)
\psset{xunit=1cm,yunit=1cm}
\psline{->}(1,1)(1,2)
\psline{-}(1,2)(1,3)
\psline{->}(1,3)(2,3)
\psline{-}(2,3)(3,3)
\psline{->}(3,3)(3,2.4)
\psline{-}(3,2.4)(3,2)
\psline{->}(3,2)(2.4,2)
\psline{-}(2.4,2)(2,2)
\psline{->}(2,2)(2,1.4)
\psline{-}(2,1.4)(2,1)
\psline{->}(2,1)(1.4,1)
\psline{-}(1.4,1)(1,1)
\psline{->}(5,2)(5,2.6)
\psline{-}(5,2.6)(5,3)
\psline{->}(5,3)(5.6,3)
\psline{-}(5.6,3)(6,3)
\psline{->}(6,3)(6, 3.6)
\psline{-}(6,3.6)(6,4)
\psline{->}(6,4)(5, 4)
\psline{-}(5,4)(4,4)
\psline{->}(4,4)(4,3.4)
\psline{-}(4,3.4)(4,3)
\psline{->}(4,3)(4,2.4)
\psline{-}(4,2.4)(4,2)
\psline{->}(4,2)(4,1.4)
\psline{-}(4,1.4)(4,1)
\psline{->}(4,1)(5,1)
\psline{-}(5,1)(6,1)
\psline{->}(6,1)(6,1.6)
\psline{-}(6,1.6)(6,2)
\psline{->}(6,2)(5.4,2)
\psline{-}(5.4,2)(5,2)

\psline[linestyle = dotted]{->}(6.75, 2.5)(8.25, 2.5)
\psline{->}(9,1)(9,2)
\psline{-}(9,2)(9,3)
\psline{->}(9,3)(10,3)
\psline{-}(10,3)(11.05,3)
\psline{->}(11.05,3)(11.05,2.4)
\psline{-}(11.05,2.4)(11.05,2)
\psline{->}(11,2)(10.4,2)
\psline{-}(10.4,2)(10,2)
\psline{->}(10,2)(10,1.4)
\psline{-}(10,1.4)(10,1)
\psline{->}(10,1)(9.4,1)
\psline{-}(9.4,1)(9,1)
\psline{->}(12,2)(12,2.6)
\psline{-}(12,2.6)(12,3)
\psline{->}(12,3)(12.6,3)
\psline{-}(12.6,3)(13,3)
\psline{->}(13,3)(13, 3.6)
\psline{-}(13,3.6)(13,4)
\psline{->}(13,4)(12, 4)
\psline{-}(12,4)(11,4)
\psline{->}(11,4)(11,3.4)
\psline{-}(11,3.4)(11,3)
\psline{->}(11,3)(11,2.4)
\psline{-}(11,2.4)(11,2)
\psline{->}(11.05,2)(11.05,1.4)
\psline{-}(11.05,1.4)(11.05,1)
\psline{->}(11.05,1)(12,1)
\psline{-}(12,1)(13,1)
\psline{->}(13,1)(13,1.6)
\psline{-}(13,1.6)(13,2)
\psline{->}(13,2)(12.4,2)
\psline{-}(12.4,2)(12,2)
\rput(3.5,2.5){$\oplus$}
\rput(2.8,2.5){$e$}
\rput(4.2,2.5){$e$}
\rput(10.8,2.5){$e$}
\end{pspicture}
%\end{pdfpic}
\caption{Positive merger.}
\label{posstfig} %ee case
\end{figure}
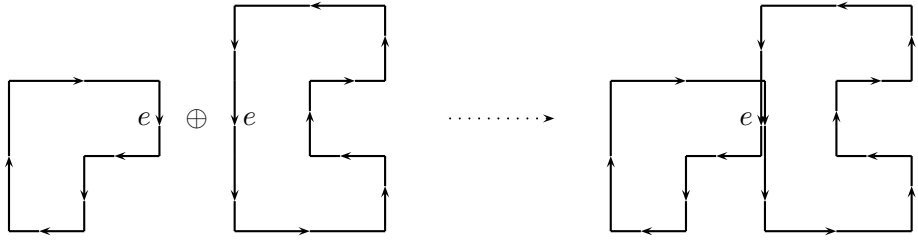

%NEGATIVE MERGER (e,e)
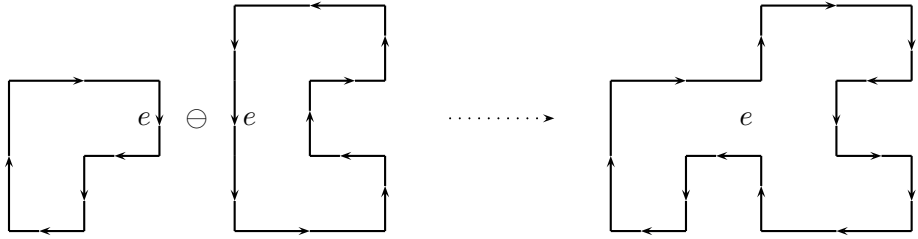
\begin{figure}[t]
%\begin{pdfpic}
\begin{pspicture}(1,1)(11,4.5)
\psset{xunit=1cm,yunit=1cm}
\psline{->}(1,1)(1,2)
\psline{-}(1,2)(1,3)
\psline{->}(1,3)(2,3)
\psline{-}(2,3)(3,3)
\psline{->}(3,3)(3,2.4)
\psline{-}(3,2.4)(3,2)
\psline{->}(3,2)(2.4,2)
\psline{-}(2.4,2)(2,2)
\psline{->}(2,2)(2,1.4)
\psline{-}(2,1.4)(2,1)
\psline{->}(2,1)(1.4,1)
\psline{-}(1.4,1)(1,1)
\psline{->}(5,2)(5,2.6)
\psline{-}(5,2.6)(5,3)
\psline{->}(5,3)(5.6,3)
\psline{-}(5.6,3)(6,3)
\psline{->}(6,3)(6, 3.6)
\psline{-}(6,3.6)(6,4)
\psline{->}(6,4)(5, 4)
\psline{-}(5,4)(4,4)
\psline{->}(4,4)(4,3.4)
\psline{-}(4,3.4)(4,3)
\psline{->}(4,3)(4,2.4)
\psline{-}(4,2.4)(4,2)
\psline{->}(4,2)(4,1.4)
\psline{-}(4,1.4)(4,1)
\psline{->}(4,1)(5,1)
\psline{-}(5,1)(6,1)
\psline{->}(6,1)(6,1.6)
\psline{-}(6,1.6)(6,2)
\psline{->}(6,2)(5.4,2)
\psline{-}(5.4,2)(5,2)

\psline[linestyle = dotted]{->}(6.75, 2.5)(8.25, 2.5)
\psline{->}(9,1)(9,2)
\psline{-}(9,2)(9,3)
\psline{->}(9,3)(10,3)
\psline{-}(10,3)(11,3)
\psline{->}(11,2)(10.4,2)
\psline{-}(10.4,2)(10,2)
\psline{->}(10,2)(10,1.4)
\psline{-}(10,1.4)(10,1)
\psline{->}(10,1)(9.4,1)
\psline{-}(9.4,1)(9,1)
\psline{-}(12,2.4)(12,2)
\psline{->}(12,3)(12,2.4)
\psline{-}(12.4,3)(12,3)
\psline{->}(13,3)(12.4,3)
\psline{-}(13,3.4)(13, 3)
\psline{->}(13,4)(13,3.4)
\psline{-}(12,4)(13, 4)
\psline{->}(11,4)(12,4)
\psline{-}(11,3.6)(11,4)
\psline{->}(11,3)(11,3.6)
\psline{-}(11,1.6)(11,2)
\psline{->}(11,1)(11,1.6)
\psline{-}(12,1)(11,1)
\psline{->}(13,1)(12,1)
\psline{-}(13,1.4)(13,1)
\psline{->}(13,2)(13,1.4)
\psline{-}(12.6,2)(13,2)
\psline{->}(12,2)(12.6,2)
\rput(3.5,2.5){$\ominus$}
\rput(2.8,2.5){$e$}
\rput(4.2,2.5){$e$}
\rput(10.8,2.5){$e$}
\end{pspicture}
%\end{pdfpic}
\caption{Negative merger.}
\label{negstfig} %ee
\end{figure}

\textbf{Deformation:}  
If a loop $l'$ is produced by merging a plaquette with a loop $l$, we will say that $l'$ is a deformation of $l$. If the merger is positive we will say that $l'$ is a positive deformation, whereas if the merger is negative we will say that $l'$ is a negative deformation. Since a plaquette cannot contain an edge $e$ or its inverse at more than one location, we will use the notations $l \oplus_x p$ and $l\ominus_x p$ to denote the loops obtained by merging $l$ and $p$ at locations $x$ and $y$, where $y$ is the unique location in $p$ where $e$ or $e^{-1}$ occurs, where $e$ is the edge occurring at location $x$ in $l$. Also, we will denote the set of positively oriented plaquettes containing an edge  $e$ or its inverse by $\cp^+(e)$. Figure~\ref{posdefig} illustrates a positive deformation. A negative deformation is illustrated in Figure~\ref{negdefig}.

%POSITIVE DEFORMATION
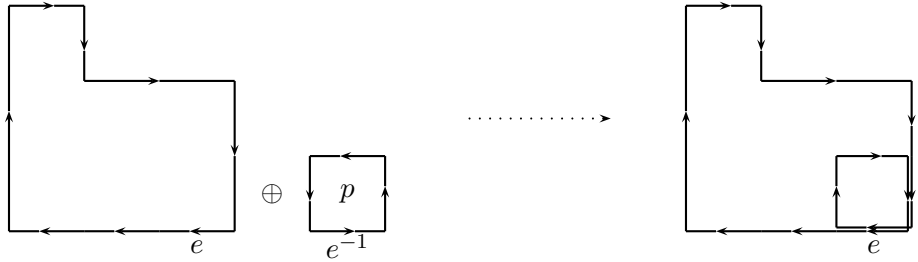
\begin{figure}[t]
%\begin{pdfpic}
\begin{pspicture}(1,1)(11,4.5)
\psset{xunit=1cm,yunit=1cm}
\psline{->}(1,1)(1,2.6)
\psline{-}(1,2.6)(1,4)
\psline{->}(1,4)(1.6,4)
\psline{-}(1.6,4)(2,4)
\psline{->}(2,4)(2, 3.4)
\psline{-}(2, 3.4)(2,3)
\psline{->}(2,3)(3,3)
\psline{-}(3,3)(4,3)
\psline{->}(4,3)(4,2)
\psline{-}(4,2)(4,1)
\psline{->}(4, 1)(3.4,1)
\psline{-}(3.4,1)(3,1)
\psline{->}(3, 1)(2.4,1)
\psline{-}(2.4,1)(2,1)
\psline{->}(2, 1)(1.4,1)
\psline{-}(1.4,1)(1,1)

\psline{->}(5,1)(5.6,1)
\psline{-}(5.6,1)(6,1)
\psline{->}(6,1)(6,1.6)
\psline{-}(6,1.6)(6,2)
\psline{->}(6,2)(5.4,2)
\psline{-}(5.4,2)(5,2)
\psline{->}(5,2)(5,1.4)
\psline{-}(5,1.4)(5,1)

\psline[linestyle = dotted]{->}(7, 2.5)(9, 2.5)

\psline{->}(10,1)(10,2.6)
\psline{-}(10,2.6)(10,4)
\psline{->}(10,4)(10.6,4)
\psline{-}(10.6,4)(11,4)
\psline{->}(11,4)(11, 3.4)
\psline{-}(11, 3.4)(11,3)
\psline{->}(11,3)(12,3)
\psline{-}(12,3)(13,3)
\psline{->}(13,3)(13,2.4)
\psline{-}(13,2.4)(13,2)
\psline{->}(13,2)(13,1.4)
\psline{-}(13,1.4)(13,1.05)
\psline{->}(13,1.05)(12.4,1.05)
\psline{-}(12.4,1.05)(12,1.05)
\psline{->}(12,1.05)(12,1.6)
\psline{-}(12,1.6)(12,2)
\psline{->}(12,2)(12.6,2)
\psline{-}(12.6,2)(12.95,2)
\psline{->}(12.95,2)(12.95,1.4)
\psline{-}(12.95,1.4)(12.95,1)
\psline{->}(12.95, 1)(12.4,1)
\psline{-}(12.4,1)(12,1)
\psline{->}(12, 1)(11.4,1)
\psline{-}(11.4,1)(11,1)
\psline{->}(11, 1)(10.4,1)
\psline{-}(10.4,1)(10,1)

\rput(12.5, 0.8){$e$}
\rput(3.5, 0.8){$e$}
\rput(5.5, 0.8){$e^{-1}$}
\rput(5.5, 1.5){$p$}
\rput(4.5, 1.5){$\oplus$}
\end{pspicture}
%\end{pdfpic}
\caption{Positive deformation.}
\label{posdefig}
\end{figure}

%NEGATIVE DEFORMATION
\begin{figure}[t]
%\begin{pdfpic}
\begin{pspicture}(1,1)(11,4.5)
\psset{xunit=1cm,yunit=1cm}
\psline{->}(1,1)(1,2.6)
\psline{-}(1,2.6)(1,4)
\psline{->}(1,4)(1.6,4)
\psline{-}(1.6,4)(2,4)
\psline{->}(2,4)(2, 3.4)
\psline{-}(2, 3.4)(2,3)
\psline{->}(2,3)(3,3)
\psline{-}(3,3)(4,3)
\psline{->}(4,3)(4,2)
\psline{-}(4,2)(4,1)
\psline{->}(4, 1)(3.4,1)
\psline{-}(3.4,1)(3,1)
\psline{->}(3, 1)(2.4,1)
\psline{-}(2.4,1)(2,1)
\psline{->}(2, 1)(1.4,1)
\psline{-}(1.4,1)(1,1)

\psline{->}(5,1)(5.6,1)
\psline{-}(5.6,1)(6,1)
\psline{->}(6,1)(6,1.6)
\psline{-}(6,1.6)(6,2)
\psline{->}(6,2)(5.4,2)
\psline{-}(5.4,2)(5,2)
\psline{->}(5,2)(5,1.4)
\psline{-}(5,1.4)(5,1)

\psline[linestyle = dotted]{->}(7, 2.5)(9, 2.5)

\psline{->}(10,1)(10,2.6)
\psline{-}(10,2.6)(10,4)
\psline{->}(10,4)(10.6,4)
\psline{-}(10.6,4)(11,4)
\psline{->}(11,4)(11, 3.4)
\psline{-}(11, 3.4)(11,3)
\psline{->}(11,3)(12,3)
\psline{-}(12,3)(13,3)
\psline{->}(13,3)(13,2.4)
\psline{-}(13,2.4)(13,2)
\psline{->}(13,2)(12.4,2)
\psline{-}(12.4,2)(12,2)
\psline{->}(12,2)(12,1.4)
\psline{-}(12,1.4)(12,1)
\psline{->}(12, 1)(11.4,1)
\psline{-}(11.4,1)(11,1)
\psline{->}(11, 1)(10.4,1)
\psline{-}(10.4,1)(10,1)

\rput(3.5, 0.8){$e$}
\rput(5.5, 0.8){$e^{-1}$}
\rput(5.5, 1.5){$p$}
\rput(4.5, 1.5){$\ominus$}
\end{pspicture}
%\end{pdfpic}
\caption{Negative deformation.}
\label{negdefig}
\end{figure}
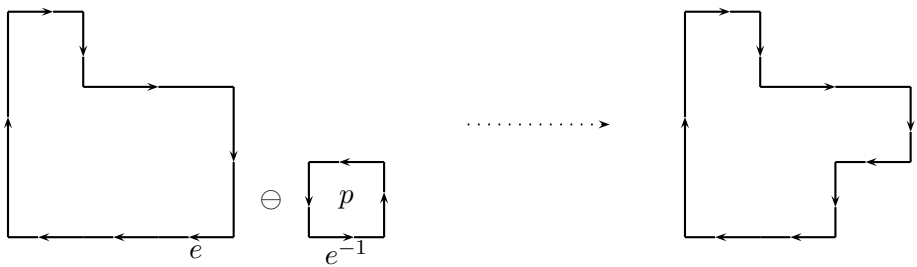

\textbf{Splitting:} Let $l$ be a non-null loop and let $x,y$ be distinct locations in $l$. If $l$ contains the same edge $e$ at $x$ and $y$, write $l = aebec$ and define the positive splitting of $l$ at $x$ and $y$ to be the pair of loops 
\[
\times^1_{x,y} l := [aec]\, , \ \ \times^2_{x,y} l := [be]\, .
\]
If $l$ contains $e$ at location $x$ and $e^{-1}$ at location $y$, write $l = aeb e^{-1} c$ and define the negative splitting of $l$ at $x$ and $y$ to be the pair of loops
\[
\times^1_{x,y} l := [ac]\, , \ \ \times^2_{x,y} l := [b]\, .
\]
%If $x>y$, let $\times^1_{x,y}l := \times^1_{y,x} l$ and $\times^2_{x,y} l := \times^2_{y,x} l$ in both of the above cases. 
Figure \ref{posspfig} illustrates an example of positive splitting, and Figure \ref{negspfig} illustrates negative splitting.

%POSITIVE PLITTING
\begin{figure}[t]
%\begin{pdfpic}
\begin{pspicture}(2,1)(12,4.5)
\psset{xunit=1cm,yunit=1cm}
\psline{->}(1,1)(1,2)
\psline{-}(1,2)(1,3)
\psline{->}(1,3)(2,3)
\psline{-}(2,3)(3.05,3)
\psline{->}(3.05,3)(3.05,2.4)
\psline{-}(3.05,2.4)(3.05,2)
\psline{->}(3,2)(2.4,2)
\psline{-}(2.4,2)(2,2)
\psline{->}(2,2)(2,1.4)
\psline{-}(2,1.4)(2,1)
\psline{->}(2,1)(1.4,1)
\psline{-}(1.4,1)(1,1)
\psline{->}(4,2)(4,2.6)
\psline{-}(4,2.6)(4,3)
\psline{->}(4,3)(4.6,3)
\psline{-}(4.6,3)(5,3)
\psline{->}(5,3)(5, 3.6)
\psline{-}(5,3.6)(5,4)
\psline{->}(5,4)(4, 4)
\psline{-}(4,4)(3,4)
\psline{->}(3,4)(3,3.4)
\psline{-}(3,3.4)(3,3)
\psline{->}(3,3)(3,2.4)
\psline{-}(3,2.4)(3,2)
\psline{->}(3.05,2)(3.05,1.4)
\psline{-}(3.05,1.4)(3.05,1)
\psline{->}(3.05,1)(4,1)
\psline{-}(4,1)(5,1)
\psline{->}(5,1)(5,1.6)
\psline{-}(5,1.6)(5,2)
\psline{->}(5,2)(4.4,2)
\psline{-}(4.4,2)(4,2)

\psline[linestyle = dotted]{->}(6, 2.5)(8, 2.5)
\psline{->}(9,1)(9,2)
\psline{-}(9,2)(9,3)
\psline{->}(9,3)(10,3)
\psline{-}(10,3)(11,3)
\psline{->}(11,3)(11,2.4)
\psline{-}(11,2.4)(11,2)
\psline{->}(11,2)(10.4,2)
\psline{-}(10.4,2)(10,2)
\psline{->}(10,2)(10,1.4)
\psline{-}(10,1.4)(10,1)
\psline{->}(10,1)(9.4,1)
\psline{-}(9.4,1)(9,1)
\psline{->}(13,2)(13,2.6)
\psline{-}(13,2.6)(13,3)
\psline{->}(13,3)(13.6,3)
\psline{-}(13.6,3)(14,3)
\psline{->}(14,3)(14,3.6)
\psline{-}(14,3.6)(14,4)
\psline{->}(14,4)(13, 4)
\psline{-}(13,4)(12,4)
\psline{->}(12,4)(12,3.4)
\psline{-}(12,3.4)(12,3)
\psline{->}(12,3)(12,2.4)
\psline{-}(12,2.4)(12,2)
\psline{->}(12,2)(12,1.4)
\psline{-}(12,1.4)(12,1)
\psline{->}(12,1)(13,1)
\psline{-}(13,1)(14,1)
\psline{->}(14,1)(14,1.6)
\psline{-}(14,1.6)(14,2)
\psline{->}(14,2)(13.4,2)
\psline{-}(13.4,2)(13,2)

\rput(2.8,2.5){$e$}
\rput(10.8,2.5){$e$}
\rput(12.2,2.5){$e$}
\end{pspicture}
%\end{pdfpic}
\caption{Positive splitting.}
\label{posspfig}
\end{figure}
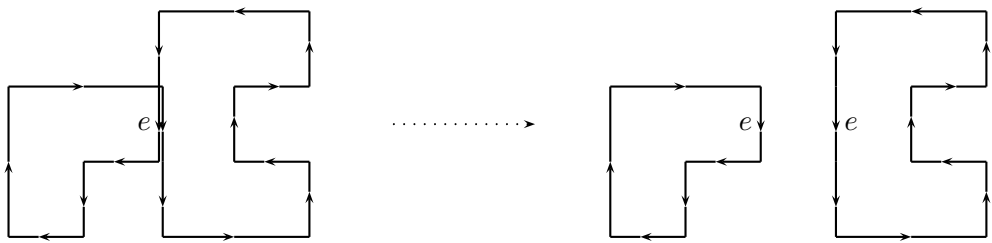

%NEGATIVE PLITTING
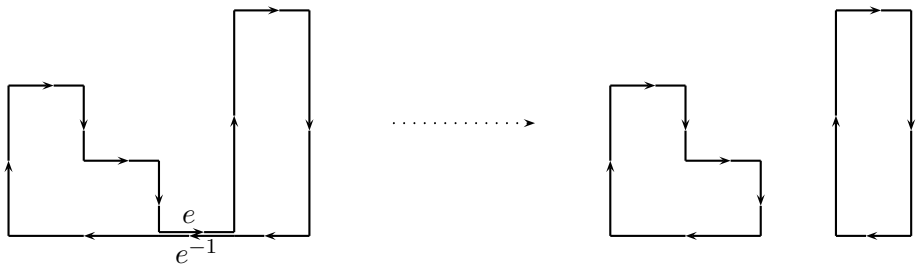
\begin{figure}[t]
%\begin{pdfpic}
\begin{pspicture}(2,1)(12,4.5)
\psset{xunit=1cm,yunit=1cm}
\psline{->}(1,1)(1,2)
\psline{-}(1,2)(1,3)
\psline{->}(1,3)(1.6,3)
\psline{-}(1.6,3)(2,3)
\psline{->}(2,3)(2,2.4)
\psline{-}(2,2.4)(2,2)
\psline{->}(2,2)(2.6,2)
\psline{-}(2.6,2)(3,2)
\psline{->}(3,2)(3,1.4)
\psline{-}(3,1.4)(3,1.05)
\psline{->}(3,1.05)(3.6,1.05)
\psline{-}(3.6,1.05)(4,1.05)
\psline{->}(4,1.05)(4,2.6)
\psline{-}(4,2.6)(4,4)
\psline{->}(4,4)(4.6,4)
\psline{-}(4.6,4)(5,4)
\psline{->}(5,4)(5, 2.4)
\psline{-}(5,2.4)(5,1)
\psline{->}(5,1)(4.4, 1)
\psline{-}(4.4,1)(4,1)
\psline{->}(4,1)(3.4,1)
\psline{-}(3.4,1)(3,1)
\psline{->}(3,1)(2,1)
\psline{-}(2,1)(1,1)
\psline[linestyle = dotted]{->}(6, 2.5)(8, 2.5)
\psline{->}(9,1)(9,2)
\psline{-}(9,2)(9,3)
\psline{->}(9,3)(9.6,3)
\psline{-}(9.6,3)(10,3)
\psline{->}(10,3)(10,2.4)
\psline{-}(10,2.4)(10,2)
\psline{->}(10,2)(10.6,2)
\psline{-}(10.6,2)(11,2)
\psline{->}(11,2)(11,1.4)
\psline{-}(11,1.4)(11,1)
\psline{->}(12,1)(12,2.6)
\psline{-}(12,2.6)(12,4)
\psline{->}(12,4)(12.6,4)
\psline{-}(12.6,4)(13,4)
\psline{->}(13,4)(13, 2.4)
\psline{-}(13,2.4)(13,1)
\psline{->}(13,1)(12.4, 1)
\psline{-}(12.4,1)(12,1)
\psline{->}(11,1)(10,1)
\psline{-}(10,1)(9,1)

\rput(3.4,1.25){$e$}
\rput(3.5,0.8){$e^{-1}$}
\end{pspicture}
%\end{pdfpic}
\caption{Negative splitting.}
\label{negspfig}
\end{figure}

\textbf{Twisting:} Let $l$ be a non-null loop and $x,y$ be two locations in $l$. If $l$ contains an edge $e$ at both $x$ and $y$, write $l= aebec$ and define the negative twisting of $l$ between $x$ and $y$ as the loop
\[
\propto_{x,y} l := [ab^{-1} c]\, .
\]
If $l$ contains an edge $e$ at location $x$ and $e^{-1}$ at location $y$, write $l = aebe^{-1}c$, and define the positive twisting of $l$ between $x$ and $y$ as  the loop
\[
\propto_{x,y} l := [aeb^{-1}e^{-1} c]\, .
\]
It is easy to verify that these are indeed loops. %If $x>y$, let $\propto_{x,y} l :=\, \propto_{y,x} l$ in both of the above cases. 
Positive twisting is illustrated in Figure \ref{postwfig} and negative twisting is illustrated in Figure \ref{negtwfig}.
%POSITIVE TWISTING
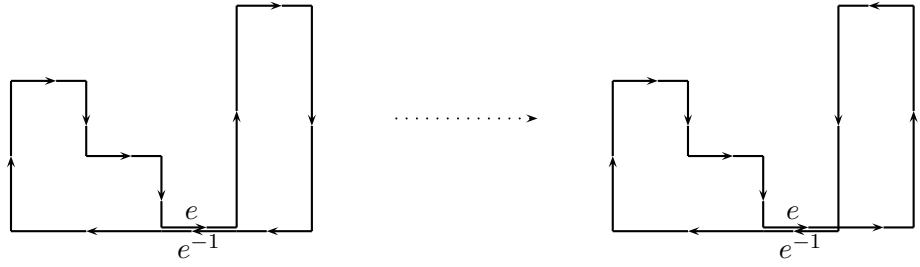
\begin{figure}[t]
%\begin{pdfpic}
\begin{pspicture}(1,1)(11,4.5)
\psset{xunit=1cm,yunit=1cm}
\psline{->}(1,1)(1,2)
\psline{-}(1,2)(1,3)
\psline{->}(1,3)(1.6,3)
\psline{-}(1.6,3)(2,3)
\psline{->}(2,3)(2,2.4)
\psline{-}(2,2.4)(2,2)
\psline{->}(2,2)(2.6,2)
\psline{-}(2.6,2)(3,2)
\psline{->}(3,2)(3,1.4)
\psline{-}(3,1.4)(3,1.05)
\psline{->}(3,1.05)(3.6,1.05)
\psline{-}(3.6,1.05)(4,1.05)
\psline{->}(4,1.05)(4,2.6)
\psline{-}(4,2.6)(4,4)
\psline{->}(4,4)(4.6,4)
\psline{-}(4.6,4)(5,4)
\psline{->}(5,4)(5, 2.4)
\psline{-}(5,2.4)(5,1)
\psline{->}(5,1)(4.4, 1)
\psline{-}(4.4,1)(4,1)
\psline{->}(4,1)(3.4,1)
\psline{-}(3.4,1)(3,1)
\psline{->}(3,1)(2,1)
\psline{-}(2,1)(1,1)
\psline[linestyle = dotted]{->}(6, 2.5)(8, 2.5)
\psline{->}(9,1)(9,2)
\psline{-}(9,2)(9,3)
\psline{->}(9,3)(9.6,3)
\psline{-}(9.6,3)(10,3)
\psline{->}(10,3)(10,2.4)
\psline{-}(10,2.4)(10,2)
\psline{->}(10,2)(10.6,2)
\psline{-}(10.6,2)(11,2)
\psline{->}(11,2)(11,1.4)
\psline{-}(11,1.4)(11,1.05)
\psline{->}(11,1.05)(11.6,1.05)
\psline{-}(11.6,1.05)(12,1.05)
\psline{-}(12,2.4)(12,1)
\psline{->}(12,4)(12,2.4)
\psline{-}(12.4,4) (12,4)
\psline{->}(13,4)(12.4,4)
\psline{-}(13,2.6)(13, 4)
\psline{->}(13,1.05)(13,2.6)
\psline{-}(12.6,1.05)(13,1.05)
\psline{->}(12,1.05)(12.6,1.05)
\psline{->}(12,1)(11.4,1)
\psline{-}(11.4,1)(11,1)
\psline{->}(11,1)(10,1)
\psline{-}(10,1)(9,1)

\rput(3.4,1.25){$e$}
\rput(3.5,0.8){$e^{-1}$}
\rput(11.4,1.25){$e$}
\rput(11.5,0.8){$e^{-1}$}
\end{pspicture}
%\end{pdfpic}
\caption{Positive Twisting.}
\label{postwfig}
\end{figure}

%NEGATIVE TWISTING
\begin{figure}[t]
%\begin{pdfpic}
\begin{pspicture}(1,1)(11,4.5)
\psset{xunit=1cm,yunit=1cm}
\psline{->}(1,1)(1,2)
\psline{-}(1,2)(1,3)
\psline{->}(1,3)(2,3)
\psline{-}(2,3)(3.05,3)
\psline{->}(3.05,3)(3.05,2.4)
\psline{-}(3.05,2.4)(3.05,2)
\psline{->}(3,2)(2.4,2)
\psline{-}(2.4,2)(2,2)
\psline{->}(2,2)(2,1.4)
\psline{-}(2,1.4)(2,1)
\psline{->}(2,1)(1.4,1)
\psline{-}(1.4,1)(1,1)
\psline{->}(4,2)(4,2.6)
\psline{-}(4,2.6)(4,3)
\psline{->}(4,3)(4.6,3)
\psline{-}(4.6,3)(5,3)
\psline{->}(5,3)(5, 3.6)
\psline{-}(5,3.6)(5,4)
\psline{->}(5,4)(4, 4)
\psline{-}(4,4)(3,4)
\psline{->}(3,4)(3,3.4)
\psline{-}(3,3.4)(3,3)
\psline{->}(3,3)(3,2.4)
\psline{-}(3,2.4)(3,2)
\psline{->}(3.05,2)(3.05,1.4)
\psline{-}(3.05,1.4)(3.05,1)
\psline{->}(3.05,1)(4,1)
\psline{-}(4,1)(5,1)
\psline{->}(5,1)(5,1.6)
\psline{-}(5,1.6)(5,2)
\psline{->}(5,2)(4.4,2)
\psline{-}(4.4,2)(4,2)

\psline[linestyle = dotted]{->}(6, 2.5)(8, 2.5)

\psline{->}(9,1)(9,2)
\psline{-}(9,2)(9,3)
\psline{->}(9,3)(10,3)
\psline{-}(10,3)(11,3)
\psline{->}(11,2)(10.4,2)
\psline{-}(10.4,2)(10,2)
\psline{->}(10,2)(10,1.4)
\psline{-}(10,1.4)(10,1)
\psline{->}(10,1)(9.4,1)
\psline{-}(9.4,1)(9,1)
\psline{-}(12,2.4)(12,2)
\psline{->}(12,3)(12,2.4)
\psline{-}(12.4,3)(12,3)
\psline{->}(13,3)(12.4,3)
\psline{-}(13,3.4)(13, 3)
\psline{->}(13,4)(13,3.4)
\psline{-}(12,4)(13, 4)
\psline{->}(11,4)(12,4)
\psline{-}(11,3.6)(11,4)
\psline{->}(11,3)(11,3.6)
\psline{-}(11,1.6)(11,2)
\psline{->}(11,1)(11,1.6)
\psline{-}(12,1)(11,1)
\psline{->}(13,1)(12,1)
\psline{-}(13,1.4)(13,1)
\psline{->}(13,2)(13,1.4)
\psline{-}(12.6,2)(13,2)
\psline{->}(12,2)(12.6,2)
\rput(2.8,2.5){$e$}
\end{pspicture}
%\end{pdfpic}
\caption{Negative twisting.}
\label{negtwfig}
\end{figure}
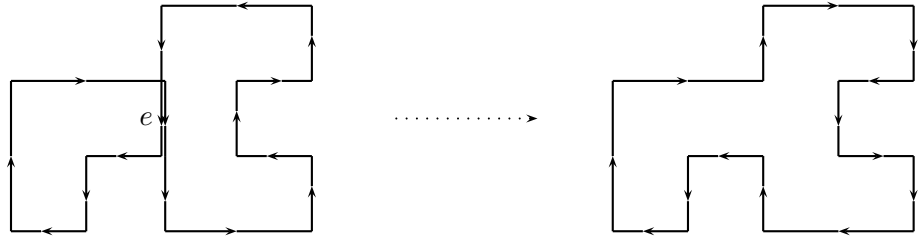

\textbf{Inaction:} As the name suggests, inaction means nothing changes: the loop remains the same after the operation as it was before.

If $s$ and $s'$ are loop sequences, say that $s'$ is a splitting of $s$ if $s'$ is obtained from $s$ by splitting one of the component loops in the minimal representation of $s$. Similarly, say that $s'$ is deformation or twisting of $s$ if $s'$ is obtained by deforming or twisting one of the components of $s$, $s'$ is inaction of $s$ if $s'=s$, and say that $s'$ is a merger of $s$ if $s'$ is obtained by merging two of the component loops of $s$. Let
\begin{align*}
\fd^+(s) &:= \{s': \text{ $s'$ is a positive deformation of $s$}\}\, ,\\
\fd^-(s) &:= \{s': \text{ $s'$ is a negative deformation of $s$}\}\, ,\\
\fs^+(s) &:= \{s': \text{ $s'$ is a positive splitting of $s$}\}\, ,\\
\fs^-(s) &:= \{s': \text{ $s'$ is a negative splitting of $s$}\}\, ,\\
\fst^+(s) &:= \{s': \text{ $s'$ is a positive merger of $s$}\}\, ,\\
\fst^-(s) &:= \{s': \text{ $s'$ is a negative merger of $s$}\}\, ,\\
\ftw^+(s) &:= \{s': \text{ $s'$ is a positive twisting of $s$}\}\, ,\\
\ftw^-(s) &:= \{s': \text{ $s'$ is a negative twisting of $s$}\}\, .
\end{align*}
Let $\fd(s)=\fd^+(s)\cup \fd^-(s)$, and define $\fs(s)$, $\fst(s)$ and $\ftw(s)$ similarly.

If a loop can be split positively at locations $x$ and $y$, then it can also be split positively at $y$ and $x$, producing the same pair of loops but in reverse order. Since the order of loops in a loop sequence is important, these two splittings are identified as distinct elements of $\fs^+(s)$. Similar remarks apply for negative splittings, twistings and mergers. For example, while counting mergers, one should be careful about the following. Let $s = (l_1,\ldots,l_n)$ be a loop sequence. Suppose that $l_1$ and $l_r$ can be positively merged at locations $x$ and $y$. Then $(l_1\oplus_{x,y} l_r,l_2,\ldots, l_{r-1}, l_{r+1},\ldots, l_n)$ is the sequence obtained by performing this merging operation. However, in this situation, $l_r$ and $l_1$ can be positively merged at locations $y$ and $x$, producing the loop sequence $(l_2,\ldots, l_{r-1}, l_r\oplus_{y,x} l_1,l_{r+1},\ldots, l_n)$. Although the loops $l_1\oplus_{x,y} l_r$ and $l_r\oplus_{y,x} l_1$ are the same, the two operations mentioned above are counted as distinct elements of $\fst^+(s)$, since the order of loops in a loop sequence is important. A similar remark applies for negative mergers. 

A trajectory is a sequence of loop sequences $(s_0, s_1, \ldots)$ where each $s_{i+1}$ is obtained from $s_{i}$ by applying one of the five kinds of operations defined above to loops in the minimal representation of $s_i$.  A trajectory can be finite or infinite. A finite trajectory whose terminal loop sequence is null and all non-terminal loop sequences are non-null  will be called a vanishing trajectory. Let $\mx(s)$ be the set of all vanishing trajectories that start at a loop sequence $s$. Let $\mx_{i,a, b, c}(s)$ be the set of all vanishing trajectories that start at a loop sequence $s$ and has $i$ deformations, $a$ twistings, $b$ mergers and $c$ inactions.  We will later prove that for any element of $\mx_{i, a, b, c}(s)$, the number of splitting operations is bounded by a number that depends on $s$, $i$ and $b$.

Define the genus of a trajectory as the sum of the number of mergers and half the number of twistings and inactions. The term ``genus'' is inspired by the topological definition of the genus of a surface. A trajectory traces out a surface over time, and mergers give rise to handles in this surface. However, the genus that we define here may have a somewhat different value than the topological genus, since twistings and inactions are not considered in the topological definition. If a trajectory has no twistings and inactions, then its genus is exactly the same as the topological genus. 

For $i, k\geq 0$ let 
$$\mx_{i, k}(s):=\bigcup\limits_{a+2b+c=k}\mx_{i, a, b, c}(s)\, ,$$
and let
$$\mx_{k}(s):=\bigcup\limits_{i=0}^{\infty}\mx_{i, k}(s)\,.$$ 
In words, $\mx_k(s)$ is the set of vanishing trajectories of genus $k/2$ that start at $s$, and $\mx_{i,k}(s)$ is the subset of $\mx_k(s)$ consisting of all trajectories that have exactly $i$ deformations. 

Figure \ref{trajfig} illustrates a vanishing trajectory with genus $2$. The surface traced out by this trajectory is depicted in Figure \ref{surffig}. Performed operations are negative deformation, positive twisting, negative splitting, negative deformation,  negative deformation, negative merger, negative deformation, inaction, negative deformation and  negative deformation. If $s$ is the loop on the top left corner of Figure \ref{trajfig}, then this trajectory is an element of $\mx_{6,1,1,1}(s)$, $\mx_{6, 4}(s)$ and $\mx_{4}(s)$.

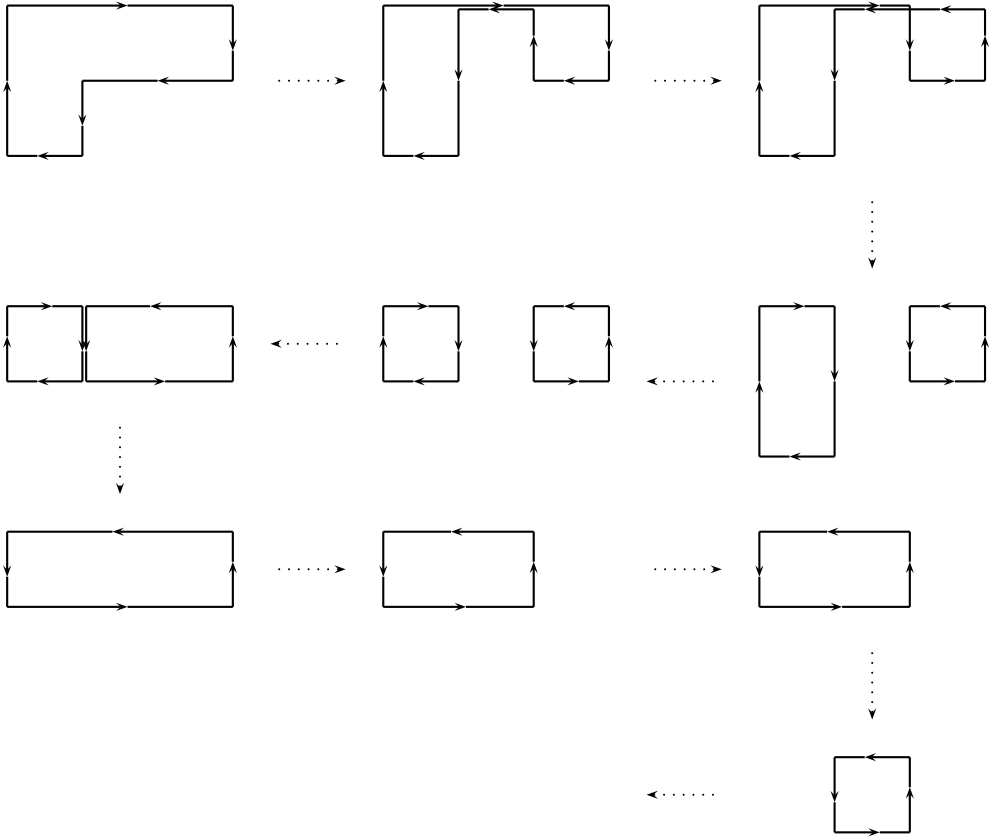
\begin{figure}[t]
\begin{pspicture}(1,-5.5)(14.5,6.5)
\psset{xunit=1cm,yunit=1cm}
\psline{->}(1,6)(2.6,6)
\psline{-}(2.6,6)(4,6)
\psline{->}(4,6)(4,5.4)
\psline{-}(4,5.4)(4,5)
\psline{->}(4,5)(3,5)
\psline{-}(3,5)(2,5)
\psline{->}(2,5)(2,4.4)
\psline{-}(2,4.4)(2,4)
\psline{->}(2,4)(1.4,4)
\psline{-}(1.4,4)(1,4)
\psline{->}(1,4)(1,5)
\psline{-}(1,5)(1,6)
\psline[linestyle = dotted]{->}(4.5,5)(5.5,5)
\psline{->}(6,6)(7.6,6)
\psline{-}(7.6,6)(9,6)
\psline{->}(9,6)(9,5.4)
\psline{-}(9,5.4)(9,5)
\psline{->}(9,5)(8.4,5)
\psline{-}(8.4,5)(8,5)
\psline{->}(8,5)(8,5.6)
\psline{-}(8,5.6)(8,5.95)
\psline{->}(8,5.95)(7.4,5.95)
\psline{-}(7.4, 5.95)(7,5.95)
\psline{->}(7, 5.95)(7,5)
\psline{-}(7,5)(7,4)
\psline{->}(7,4)(6.4,4)
\psline{-}(6.4,4)(6,4)
\psline{->}(6,4)(6,5)
\psline{-}(6,5)(6,6)
\psline[linestyle = dotted]{->}(9.5,5)(10.5,5)
\psline{->}(11,6)(12.6,6)
\psline{-}(12.6,6)(13,6)
\psline{->}(13,6)(13,5.4)
\psline{-}(13,5.4)(13,5)
\psline{->}(13,5)(13.6,5)
\psline{-}(13.6,5)(14,5)
\psline{->}(14,5)(14,5.6)
\psline{-}(14,5.6)(14,5.95)
\psline{->}(14,5.95)(13.4,5.95)
\psline{-}(13.4,5.95)(13,5.95)
\psline{->}(13,5.95)(12.4,5.95)
\psline{-}(12.4, 5.95)(12,5.95)
\psline{->}(12, 5.95)(12,5)
\psline{-}(12,5)(12,4)
\psline{->}(12,4)(11.4,4)
\psline{-}(11.4,4)(11,4)
\psline{->}(11,4)(11,5)
\psline{-}(11,5)(11,6)
\psline[linestyle = dotted]{->}(12.5,3.5)(12.5,2.5)
\psline{->}(11,2)(11.6,2)
\psline{-}(11.6,2)(12,2)
\psline{->}(12,2)(12,1)
\psline{-}(12,1)(12,0)
\psline{->}(12,0)(11.4,0)
\psline{-}(11.4,0)(11,0)
\psline{->}(11,0)(11,1)
\psline{-}(11,1)(11,2)

\psline{->}(14,2)(13.4,2)
\psline{-}(13.4,2)(13,2)
\psline{->}(13,2)(13,1.4)
\psline{-}(13,1.4)(13,1)
\psline{->}(13,1)(13.6,1)
\psline{-}(13.6,1)(14,1)
\psline{->}(14,1)(14,1.6)
\psline{-}(14,1.6)(14,2)
\psline[linestyle = dotted]{->}(10.5, 1)(9.5, 1)
\psline{->}(7,2)(7,1.4)
\psline{-}(7,1.4)(7,1)
\psline{->}(7,1)(6.4,1)
\psline{-}(6.4,1)(6,1)
\psline{->}(6,1)(6,1.6)
\psline{-}(6,1.6)(6,2)
\psline{->}(6,2)(6.6,2)
\psline{-}(6.6,2)(7,2)

\psline{->}(9,2)(8.4,2)
\psline{-}(8.4,2)(8,2)
\psline{->}(8,2)(8,1.4)
\psline{-}(8,1.4)(8,1)
\psline{->}(8,1)(8.6,1)
\psline{-}(8.6,1)(9,1)
\psline{->}(9,1)(9,1.6)
\psline{-}(9,1.6)(9,2)
\psline[linestyle = dotted]{->}(5.5, 1.5)(4.5, 1.5)
\psline{->}(2,2)(2,1.4)
\psline{-}(2,1.4)(2,1)
\psline{->}(2,1)(1.4,1)
\psline{-}(1.4,1)(1,1)
\psline{->}(1,1)(1,1.6)
\psline{-}(1,1.6)(1,2)
\psline{->}(1,2)(1.6,2)
\psline{-}(1.6,2)(2,2)

\psline{->}(4,2)(2.9,2)
\psline{-}(2.9,2)(2.05,2)
\psline{->}(2.05,2)(2.05,1.4)
\psline{-}(2.05,1.4)(2.05,1)
\psline{->}(2.05,1)(3.1,1)
\psline{-}(3.1,1)(4,1)
\psline{->}(4,1)(4,1.6)
\psline{-}(4,1.6)(4,2)
\psline[linestyle = dotted]{->}(2.5, .5)(2.5, -.5)
\psline{->}(4,-1)(2.4,-1)
\psline{-}(2.4,-1)(1,-1)
\psline{->}(1,-1)(1,-1.6)
\psline{-}(1,-1.6)(1,-2)
\psline{->}(1,-2)(2.6,-2)
\psline{-}(2.6,-2)(4,-2)
\psline{->}(4,-2)(4,-1.4)
\psline{-}(4,-1.4)(4,-1)
\psline[linestyle = dotted]{->}(4.5, -1.5)(5.5, -1.5)
\psline{->}(8,-1)(6.9,-1)
\psline{-}(6.9,-1)(6,-1)
\psline{->}(6,-1)(6,-1.6)
\psline{-}(6,-1.6)(6,-2)
\psline{->}(6,-2)(7.1,-2)
\psline{-}(7.1,-2)(8,-2)
\psline{->}(8,-2)(8,-1.4)
\psline{-}(8,-1.4)(8,-1)
\psline[linestyle = dotted]{->}(9.5, -1.5)(10.5, -1.5)
\psline{->}(13,-1)(11.9,-1)
\psline{-}(11.9,-1)(11,-1)
\psline{->}(11,-1)(11,-1.6)
\psline{-}(11,-1.6)(11,-2)
\psline{->}(11,-2)(12.1,-2)
\psline{-}(12.1,-2)(13,-2)
\psline{->}(13,-2)(13,-1.4)
\psline{-}(13,-1.4)(13,-1)
\psline[linestyle = dotted]{->}(12.5, -2.5)(12.5, -3.5)
\psline{->}(13,-4)(12.4,-4)
\psline{-}(12.4,-4)(12,-4)
\psline{->}(12,-4)(12,-4.6)
\psline{-}(12,-4.6)(12,-5)
\psline{->}(12,-5)(12.6,-5)
\psline{-}(12.6,-5)(13,-5)
\psline{->}(13,-5)(13,-4.4)
\psline{-}(13,-4.4)(13,-4)
\psline[linestyle = dotted]{->}(10.5, -4.5)(9.5, -4.5)
\end{pspicture}
\caption{A vanishing trajectory of a loop. Performed operations are deformation, twisting, splitting, deformation, deformation, merger, deformation, inaction, deformation and deformation.}
\label{trajfig}
\end{figure}
Define the weight of the transition from $s$ to $s'$ at inverse coupling strength $\beta$ as
\begin{align}
w_\beta(s,s') &:= 
\begin{cases}
1 &\text{ if $s'=s$,}\\
-1/|s| &\text{ if $s'\in \ftw^+(s) \cup \fs^+(s)\cup \fst^+(s)$}\\
1/|s| &\text{ if $s'\in \ftw^-(s) \cup \fs^-(s)\cup \fst^-(s)$}\\
-\beta/|s| &\text{ if $s' \in \fd^+(s)$,}\\
\beta/|s| &\text{ if $s' \in \fd^-(s)$.}
\end{cases}\label{weight}
\end{align}
If $X = (s_0,s_1,\ldots, s_n)$ is a vanishing   trajectory, define the weight of $X$ at inverse coupling strength $\beta$ as the product
\[
w_\beta(X) := w_\beta(s_0, s_1)w_\beta(s_1, s_2)\cdots w_\beta(s_{n-1}, s_n) \, .
\]
Note that the weight of a trajectory may be positive or negative.  For example, the vanishing trajectory in Figure \ref{trajfig} has weight $-\beta^6/221184000$. This distinguishes these trajectory weights from the weights attached to trajectories of Markov processes in probability theory.

\begin{comment}
%We will also need the weight of transitions which do not depend of $\beta$.
Define the $\beta$-free weight $v(s,s')$ as follows:
\begin{align}
v(s,s') &:= 
\begin{cases}
w_{\beta}(s, s') &\text{ if $s'\in \ftw(s) \cup \fs(s)\cup \fst(s)$,}\\
\beta^{-1}w_{\beta}(s, s') &\text{ if $s' \in \fd(s)$.}
\end{cases}\label{fweight}
\end{align}
Define the $\beta$-free weight of a vanishing trajectory $X=(s_0, s_1, \ldots, s_n)$ as
\[
v(X):=v(s_0, s_1)v(s_1, s_2)\cdots v(s_{n-1}s_n)
\]
Let the length of a vanishing trajectory $X$, $\ell(X)$, be the number of transitions in $X$. That is, if $s_0, \ldots, s_{n-1}$ are non-null loop sequences and $s_n=\emptyset$, then the trajectory $X=(s_0, s_1, \ldots, s_n)$ has length $\ell(X)=n$. In particular, $\ell(\emptyset)=0$. The length of the trajectory in Figure \ref{trajfig} is $9$. Define $k^{\mathrm{th}}$ multiplicity $m_k(X)$ of a vanishing trajectory $X$ as follows. Let $m_0(\emptyset)=1$, $m_{k}(\emptyset)=0$ for all $k\geq 1$. If $X\in \mx_{i, a, b}(s)$ for some $s\neq \emptyset$, then if $k\geq a+2b$, let
$$m_{k}(X)=\dbinom{\lx+k-1-a-2b}{k-a-2b}\,,$$
and let $m_k(X)=0$ otherwise. Notice that $m_0(X)=1$ if $X$ has genus zero and $0$ otherwise. The weights and multiplicities of trajectories will appear in the formulas for the terms in the $1/N$ expansion.
\end{comment}
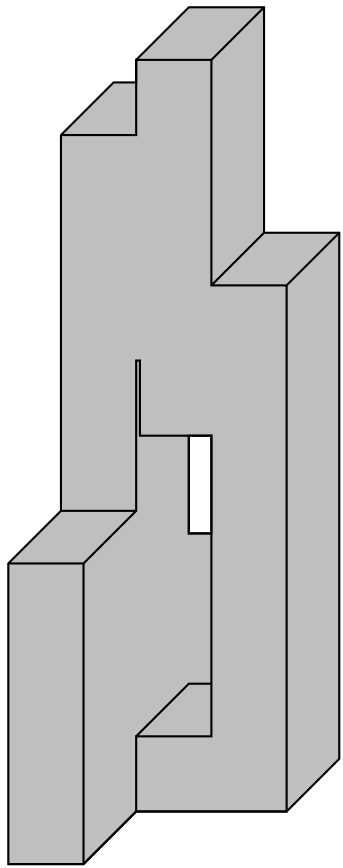
\begin{figure}[h]
%\begin{pdfpic}
\begin{pspicture}(-1,1)(9,13)
\psset{xunit=1cm,yunit=1cm}
\pspolygon[fillstyle=solid, fillcolor=lightgray](1,1)(1,5)(1.7,5.7)(1.7,10.7)(2.4,11.4)(2.7,11.4)(2.7, 11.7)(3.4,12.4)(4.4,12.4)(4.4, 9.4)(5.4,9.4)(5.4,2.4)(4.7,1.7)(2.7,1.7)(2,1)(1,1)
\pspolygon[fillstyle=solid, fillcolor=white](3.4,5.4)(3.4,6.7)(3.7,6.7)(3.7,5.4)(3.4,5.4)
\psline{-}(5.4,9.4)(4.7,8.7)
\psline{-}(4.4,9.4)(3.7, 8.7)
\psline{-}(4.4,12.4)(3.7, 11.7)
\psline{-}(2,1)(2,5)(2.7,5.7)(2.7, 7.7)(2.75, 7.7)(2.75, 6.7)(3.7, 6.7)(3.7, 2.7)(2.7, 2.7)(2.7, 1.7)
\psline{-}(1.7, 10.7)(2.7, 10.7)(2.7, 11.7)(3.7, 11.7)(3.7, 8.7)(4.7, 8.7)(4.7, 1.7)(2.7, 1.7)(2,1)(1,1)
\psline{-}(3.7,5.4)(3.4,5.4)(3.4,6.7)
\psline{-}(2.7, 2.7)(3.4,3.4)(3.7,3.4)
\psline{-}(2,5)(1,5)
\psline{-}(1.7,5.7)(2.7,5.7)
\end{pspicture}
\caption{The surface traced out by the trajectory from Figure \ref{trajfig}.}
\label{surffig}
\end{figure}

%\vskip.5in

\section{Main result: The $1/N$ expansion}\label{result}
%\textbf{- Main results of this paper}\\\\
We will assume that the reader is familiar with the notations and terminologies defined in Sections~\ref{intro} and~\ref{stringsec}. Consider $SO(N)$ lattice gauge theory on a subset $\Lambda$ of $\zz^d$. The main objective of this paper is to present asymptotic series expansion for the loop function
\[
\phi_{\Lambda, N, \beta}(s) := \frac{\smallavg{W_{l_1}W_{l_2}\cdots W_{l_n}}}{N^n}\,,
\]
where $s$ is a loop sequence with minimal representation $(l_1, \ldots, l_n)$. We will omit subscripts $\Lambda$ and $\beta$ and simply write $\phi_N(s)$ whenever $\Lambda$ and $\beta$ are clear from the context.  The following theorem establishes the $1/N$ expansion~\eqref{1nexp} and gives string-theoretic formulas for the terms in the expansion. This is the main result of this paper. 
\begin{thm}\label{mainthmofpaper}
There exists a sequence of positive real numbers $\{\beta_0(d,k)\}_{k\ge 0}$ depending only on $d$ such that for any $k\ge 0$, the following hold when $|\beta|\le \beta_0(d,k)$: %$3^{4k+2}|\beta|d<1$
\begin{enumerate}
\item[\textup{(i)}] For each $s\in \cs$, the sum
\begin{align*}
f_k(s)&:=\sum_{X\in \mx_k(s)}w_{\beta}(X) %\label{trajeq}
\end{align*}
is absolutely convergent. %, where $w_{\beta}(X)$ is the weight of a trajectory $X$ at inverse coupling strength $\beta$ and $\mx_k(s)$ is the set of vanishing trajectories of genus $k/2$ that start at $s$. 
\item[\textup{(ii)}] Let $\Lambda_1,\Lambda_2,\ldots$ be any sequence of finite subsets of $\zz^d$ that are monotonically increasing to $\zz^d$.  Then for each $s\in \cs$,
\[
\lim\limits_{N\to \infty}N^k\biggl(\phi_{\Lambda_N,N,\beta}(s)-f_0(s)-\frac{1}{N}f_1(s)- \frac{1}{N^2}f_2(s)-\cdots-\frac{1}{N^{k}}f_{k}(s)\biggr)=0\,,
\]
where $f_0,\ldots, f_k$ are the functions defined in part \textup{(i)}.
\item[\textup{(iii)}] Lastly, for each $s\in \cs$, $|f_k(s)|\leq (2^{2k+13}d)^{|s|}$. 
\end{enumerate}
\end{thm}
As mentioned before, the case of $k=0$ was proved in \cite{chatterjee15}. 
Theorem~\ref{mainthmofpaper} tells that the series expansion for $f_0$ involves only trajectories of genus zero, whereas any other $f_k$ involves trajectories of  positive genus $k/2$.

Notice that the above result is valid only for small enough $\beta$, that is, at strong coupling. It is a non-perturbative result. The original $1/N$ expansion of 't Hooft~\cite{thooft74}, on the other hand, is a perturbative weak coupling expansion. When $\beta$ is large, a lattice gauge theory can be heuristically treated as a perturbation of a Gaussian theory, which naturally leads to Feynman diagram calculations. However, this argument has not yet been made rigorous, as far as we know.

The rest of the paper is devoted to proving Theorem \ref{mainthmofpaper}. 

%\vskip.5in
% % % % % % % % % % % % % % % % % % % % % % % % % % % % % % % % % % % % % % % % % % % % % % % % % % % % % % %
% % % % % % % % % % % % % % % % % % % % % % % % % % % % % % % % % % % % % % % % % % % % % % % % % % % % % % %
% % % % % % % % % % % % % % % % % % % % % % % % % % % % % % % % % % % % % % % % % % % % % % % % % % % % % % %
\section{Preliminary lemmas}\label{prelim}
Define the length of a loop sequence $s$ with minimal representation $(l_1, \ldots, l_n)$ as
\[
|s|:=|l_1|+\cdots+|l_n|\, ,
\]
the size of $s$ as 
\[
\#s:=n\, ,
\]
and the index of $s$ as
\[
\iota(s):=|s|-\#s\, .
\]
In this section we will prove some lemmas about the behavior of the length $|s|$ and the index $\iota(s)$ of a loop sequence $s$ under the string operations defined in Section \ref{stringsec}.
%[page 36 of Sourav's paper]. 
\begin{lmm}\label{twist}
Let $s$ be a non-null loop sequence. If $s'$ is a twisting of $s$, then $|s'|\leq |s|$ and $\iota(s')\le\iota(s)$. 
\end{lmm}
\begin{proof}
It is easy to see that a twisting operation does not add any extra edges to loops in $s$. Hence $|s'|\leq |s|$. On the other hand either $\#s'=\#s$ or $\#s'=\#s-1$.
In the first case
$$\iota(s')=|s'|-\#s'\leq |s|-\#s=\iota(s)\, .$$
The second case can occur only if one of the loops in $s$ vanishes after the twisting operation. Since any non-null loop has length at least $4$, this implies that  $|s'|\leq |s|-4$ and so
$$\iota(s')=|s'|-\#s'\leq (|s|-4)-(\#s-1)=\iota(s)-3\, ,$$
which completes the proof.
\end{proof}
\begin{lmm}\label{merger}
If  $s$ is a non-null loop sequence and $s'$ is a merger of $s$, then $|s'|\leq |s|$ and $\iota(s')\leq \iota(s)+1$.
\end{lmm}
\begin{proof}
Mergers do not add extra edges. Therefore $|s'|\leq |s|$. On the other hand either $\#s'=\#s-1$ or $\#s'=\#s-2$. The first case happens if two loops in $s$ merge to form one, which gives  
$$\iota(s')=|s'|-\#s'\leq |s|-(\#s-1)=\iota(s)+1\, .$$
The second case happens if two loops merge to form the null loop. Since any non-null loop has length at least $4$, this gives $|s'|\leq |s|-8$ and 
$$\iota(s')=|s'|-\#s'=(|s|-8)-(\#s-2)=\iota(s)-6\,.$$
This completes the proof of the lemma.
\end{proof}
\begin{lmm}\label{deform}
If $s$ is a non-null loop sequence and $s'$ is a deformation of $s$, then $|s'|\leq |s|+4$ and $\iota(s')\leq \iota(s)+4$.
\end{lmm}
\begin{proof}
A plaquette is composed of four edges. Hence $|s'|\leq |s|+4$. Since only one of the components loops of $s$ is deformed to get $s'$, either $\#s'=\#s$ or $\#s'=\#s-1$. In the first case,
$$\iota(s')=|s'|-\#s'\leq (|s|+4)-\#s=\iota(s)+4\, .$$
The second case can happen only if the deformed loop vanishes. Thus in this case, $|s'|\leq |s|-4$ and
$$\iota(s')=|s'|-\#s'\leq (|s|-4)-(\#s-1)=\iota(s)-3\, ,$$
which completes the argument.
\end{proof}
The next three lemmas are quoted without proof from \cite{chatterjee15}.
\begin{lmm}\label{split1}
Let $l$ be a non-null loop and suppose that $x$ and $y$ are two distinct locations in $l$ such that $l$ admits a positive splitting at $x$ and $y$. Let $l_1 := \times^1_{x,y} l$ and $l_2 := \times^2_{x,y} l$. Then $l_1$ and $l_2$ are non-null loops, $|l_1|\le |l| - |y-x|$, and $|l_2|\le |y-x|$. 
\end{lmm}
\begin{lmm}\label{split2}
Let $l$ be a non-null loop and suppose that $x$ and $y$ are two distinct locations in $l$ such that $l$ admits a negative splitting at $x$ and $y$. Let $l_1 := \times^1_{x,y} l$ and $l_2 := \times^2_{x,y} l$. Then $l_1$ and $l_2$ are non-null loops, $|l_1|\le |l| - |y-x|-1$, and $|l_2|\le |y-x|-1$. 
\end{lmm}
\begin{lmm}\label{iota2}
If $s'$ is obtained from $s$ by a splitting operation, then $\iota(s')<\iota(s)$.
\end{lmm}
Recall the definition of the Catalan numbers: $C_0=1$, and for~$i\ge 1$, 
\[
C_i = \frac{1}{i+1}{2i\choose i} = {2i\choose i} - {2i \choose i+1}\,.
\]
We will use a well known recursion relation for Catalan numbers: For each $i\ge 0$, 
\begin{equation}\label{catalan}
C_{i+1} = \sum_{j=0}^i C_j C_{i-j}\,.
\end{equation}
We will also use the facts that $C_i$ is increasing in $i$ and that 
\begin{equation}\label{catalan2}
C_{i+1}\le 4C_i
\end{equation}
for each $i\ge 0$. Lastly, we will need the following lemma about Catalan numbers.
\begin{lmm}\label{app1}
Let $C_k$ be the $k^{\mathrm{th}}$ Catalan number. If $n, m$ are positive integers, then
$$C_{n+m-1}\leq (n+m)^2C_{n-1}C_{m-1}\,.$$
\end{lmm}
\begin{proof}
The proof is by induction on $n+m$. If $n+m\leq 5$, then it is easy to check by direct computation that 
\[
C_{n+m-1}\leq (n+m)^2\,.
\]
If $n+m=6$, then $\max\{n,m\}\geq 3$ and 
\[
C_{n+m-1}\leq (n+m)^2C_2 \leq (n+m)^2C_{n-1}C_{m-1}\,.
\]
Now assume that the claimed inequality is true up to $n+m=p$ for some $p\geq 7$. Without loss of generality assume $n\geq m$ so that $n\geq 4$. Note that
\begin{align*}
C_{k}&=\dfrac{1}{k+1}\dbinom{2k}{k}=\dfrac{4k-2}{k(k+1)}\dbinom{2k-2}{k-1}=\dfrac{4k-2}{k+1}C_{k-1}\,.
\end{align*}
By induction hypothesis,
\begin{align*}
(n+m)^2C_{n-1}C_{m-1}&=\dfrac{4n-6}{n}\biggl(\dfrac{n+m}{n+m-1}\biggr)^2 (n+m-1)^2C_{n-2}C_{m-1}\\
&\geq \dfrac{4n-6}{n}\biggl(\dfrac{n+m}{n+m-1}\biggr)^2 C_{n+m-2}\\
&=\dfrac{(4n-6)(n+m)}{n(4n+4m-6)}\biggl(\dfrac{n+m}{n+m-1}\biggr)^2 C_{n+m-1}\,.
\end{align*}
So it is enough to prove that
\begin{align}
\biggl(\dfrac{n+m}{n+m-1}\biggr)^2 &\geq \dfrac{n(4n+4m-6)}{(4n-6)(n+m)}\, . \label{catalaneq}
\end{align}
%Since
%\begin{align*}
%\biggl(\dfrac{n+m}{n+m-1}\biggr)^2=\biggl(1+\dfrac{1}{n+m-1}\biggr)^2=1+\	dfrac{2}{n+m-1}+\dfrac{1}{(n+m-1)^2}\, ,
%\end{align*}
%we need to prove that
%\begin{equation}\label{catalaneq}
%\dfrac{2(n+m)}{n+m-1}+\dfrac{n+m}{(n+m-1)^2}\geq \dfrac{6m}{4n-6}\,.
%\end{equation}
%Since the left side can be written as 
%$$2+\dfrac{2}{n+m-1}+\biggl(1+\dfrac{1}{n+m-1}\biggr)\dfrac{1}{n+m-1}\,,$$
Observe that the left side is decreasing in $m$ whereas the right side is increasing. So it is enough to prove the inequality \eqref{catalaneq} for $n=m$. In this case the inequality reduces to $4n^2+3\geq 16n$, which is true because $n\ge 4$.
\end{proof}

%\vskip.5in
\section{The unsymmetrized master loop equation}\label{unsymmsec}
Let $s$ be a loop sequence with minimal representation $(l_1,\ldots, l_n)$. For each $1\leq r\leq n$, let $A_r(e)$ be the set of locations in $l_r$ where edge $e$ occurs, and let $B_r(e)$ be the set of locations in $l_r$ where $e^{-1}$ occurs. Let $C_r(e)=A_r(e)\cup B_r(e)$. We will simply write $A_r, B_r, C_r$ whenever $e$ is clear from the context. 

Now let $e$ be the first element of $l_1$ and let  $m$ be the size of $C_1$. For any loop function $h: \cs \to \rr$ let
\begin{align*}
\sum_{\text{twist}\,s}h&=\sum_{\substack{x, y\in A_1\\x\ne y}}h(\propto_{x,y} l_1,\ldots, l_n) + \sum_{\substack{x, y\in B_1\\x\ne y}} h(\propto_{x,y} l_1, \ldots, l_n) \\
&\qquad -\sum_{x\in A_1, \, y\in B_1} h(\propto_{x,y} l_1,\ldots,l_n) -  \sum_{x\in B_1, \, y\in A_1} h(\propto_{x,y} l_1,\ldots,l_n)\, .
\end{align*}
Similarly, let
\begin{align*}
\sum_{\text{merge}\,s}h&:=\sum_{r=2}^n\sum_{x\in C_1, \, y\in C_r} h(l_1 \ominus_{x,y} l_r,\dots, l_n) - \sum_{r=2}^n\sum_{x\in C_1, \, y\in C_r}h(l_1 \oplus_{x,y} l_r, \dots, l_n)\, ,
\end{align*}
\begin{align*}
\sum_{\text{split}\,s}h&:=\sum_{x\in A_1, \, y\in B_1} h(\times_{x,y}^1 l_1,\times_{x,y}^2 l_1,\ldots,l_n)+ \sum_{x\in B_1, \, y\in A_1} h(\times_{x,y}^1 l_1,\times_{x,y}^2 l_1,\ldots,l_n)\\
&\qquad - \sum_{\substack{x,y\in A_1\\ x\ne y}} h(\times^1_{x,y} l_1,\times^2_{x,y} l_1, \ldots,l_n)-\sum_{\substack{x,y\in B_1\\ x\ne y}} h(\times^1_{x,y} l_1,\times^2_{x,y} l_1, \ldots,l_n)\, ,
\end{align*}
and
\begin{align*}
\sum_{\text{deform}\,s}h&:=\sum_{p\in \cp^+(e)}\sum_{x\in C_1}h(l_1 \ominus_{x}p, \ldots,l_n) - \sum_{p\in \cp^+(e)}\sum_{x\in C_1}h(l_1 \oplus_{x}p,\ldots,l_n) \, .
\end{align*}
Let 
\[
\sideset{}{^+}\sum\limits_{\text{twist}\,s}, \ \ \sideset{}{^+}\sum\limits_{\text{merge}\,s}, \ \ \sideset{}{^+}\sum\limits_{\text{split}\,s}, \  \ \sideset{}{^+}\sum\limits_{\text{deform}\,s}
\]
denote sums as above with minus signs between terms replaced by plus signs. The goal of this section is to prove the following result.
\begin{thm}\label{eqfuns}
There exist positive constants $\{\beta_0(d,k)\}_{k\ge 0}$ depending on $d$ and functions $\{f_k\}_{k\ge 0}$ depending on $\beta$ such that the assertions of part \textup{(ii)} of Theorem \ref{mainthmofpaper} are valid when $|\beta|\le \beta_0(d,k)$. Moreover, the sequence of functions $\{f_k\}_{k\ge 0}$ satisfies the recursive relation
\begin{align}
mf_k(s) &= mf_{k-1}(s)+\sum_{\textup{twist }s} f_{k-1}+ \sum_{\textup{merge }s}f_{k-2} +\sum_{\textup{split }s}f_{k} + \beta \sum_{\textup{deform }s}f_k \label{limequns}
\end{align}
for each $s\in \cs$, where $m = m(s)$ is the size of the set $C_1$ defined above, and $f_{j}(s)$ is understood to be zero when~$j<0$. 
\end{thm}
Note that the above theorem does not make the claim that $f_k$ has the form given in part (i) of Theorem \ref{mainthmofpaper}. It merely asserts the existence of $f_k$'s such that part (ii) is valid. The starting point of the proof of Theorem \ref{eqfuns} is the following ``master loop equation'' proved in \cite{chatterjee15}. 
\begin{thm}\label{mastern}
Let $\phi_N = \phi_{\Lambda_N, N,\beta}$ be as in Theorem \ref{mainthmofpaper}. 
Take a loop sequence $s$  such that all vertices of $\zz^d$ that are at distance $\le 1$ from any of the component loops of $s$ are contained in $\Lambda_N$. Then %Let $e$ be the first edge of $l_1$ and let $m$, $A_1$, $B_1$ and $C_1$ be as defined above. Then 
\begin{align}
m\phi_N(s)&=\dfrac{1}{N-1}\sum_{\textup{twist}\,s}\phi_N+\dfrac{1}{N(N-1)}\sum_{\textup{merge}\,s}\phi_N+\dfrac{N}{N-1}\sum_{\textup{split}\,s}\phi_N+\dfrac{\beta N}{N-1}\sum_{\textup{deform}\,s}\phi_N\,, \label{mlequns}
\end{align}
where $m$ is the size of the set $C_1$ defined above.
\end{thm}
Using the master loop equation, it was proved in \cite{chatterjee15} that part (ii) of Theorem \ref{mainthmofpaper} holds for $k=0$. To be precise, the following result was proved. 
\begin{thm}\label{masterf0uns}
Let $\phi_N = \phi_{\Lambda_N, N,\beta}$ be as in Theorem \ref{mainthmofpaper}. 
There exists a number $\beta_0(d,0)>0$ such that if $|\beta|\le \beta_0(d,0)$, then $f_0(s) = \lim_{N\ra\infty} \phi_N(s)$ exists for each $s\in \cs$, and satisfies the recursive relation
\begin{align}
mf_0(s) &= \sum_{\textup{split}\,s} f_0+ \beta \sum_{\textup{deform}\,s} f_0\, , \label{mastereqf0uns}
\end{align}
where $m$ is the size of the set $C_1$ defined above.
\end{thm}
%\textcolor{red}{Equations \eqref{mlequns} and \eqref{mastereqf0uns} are symmetrized as follows: it is easy to see that these equations are still valid even if we replace $e$ with any edge of any loop in minimal representation of $s$. So divide all edges in loop components of $s$ into sets that contain only edge $e$ and $e^{-1}$. Take $e$ from each set and apply equations \eqref{mlequns} and \eqref{mastereqf0uns} with respect to this edge. By adding these equations we get symmetrized versions.}

%\textcolor{red}{Also Theorem~\ref{uniquef0} follows by uniqueness of solution of the equation \eqref{mastereqf0uns}.

%\textcolor{red}{We will follow the same approach. Instead of proving Theorem \ref{eqf} directly we will first prove unsymmetrized version of the equation \eqref{recurf}.}

We are now ready to start the proof of Theorem \ref{eqfuns}. The proof is by induction on $k$. As mentioned above, the case $k=0$ is already known (Theorem \ref{masterf0uns}). Take some $k>0$ and assume that the claim is true for all $k'<k$. In particular, assume that $f_{k'}$ exists for any $k'<k$. Take some $s\in \cs$ and let all notation be as in the beginning of this section. Observe that by the induction hypothesis, the master loop equation \eqref{limequns} can be rewritten as
\begin{align}
mf_{k'}(s) &= \sum_{\text{twist }s} (f_{0}+\cdots+f_{k'-1})+ \sum_{\text{merge }s}(f_{0}+\cdots+f_{k'-2})\notag\\
&\qquad +\sum_{\text{split }s}(f_{0}+\cdots+f_{k'}) + \beta \sum_{\text{deform }s}(f_{0}+\cdots+f_{k'}) \label{lform}
\end{align}
for any $k'< k$. Define sequence of increasing numbers $\{L_q: q\geq 0\}$ as $L_0=1$ and 
\begin{align}
L_q= (36L_{q-1})^{4/3} \notag %\label{lk}
\end{align}
for $q\geq 1$. Fixing $k$ and working under the hypothesis that \eqref{lform} holds for all $k'<k$, we continue our proof with the following lemma. Throughout, $\Lambda_N$ and  $\phi_N = \phi_{\Lambda_N, N,\beta}$ are as in Theorem \ref{mainthmofpaper}.
% % % % % % % % % % % % % % % % % % % % % % % % % % % % % % % % % % % % % % % % % % % % % % % % % % % % % % % % % % % % % %
% % % % % % % % % % % % % % % % % % % % % % % % % % % % % % % % % % % % % % % % % % % % % % % % % % % % % % % % % % % % % %
\begin{lmm}\label{boundHq}
Take any $N$. Let $s$ be a non-null loop sequence such that all vertices of $\zz^d$ that are at distance $\le 1$ from any of the component loops of $s$ are contained in $\Lambda_N$. For any $0\leq q\leq k$, there exists $\beta_0'(d, q)>0$, depending only on $q$ and $d$, such that for any $|\beta|\leq \beta_0'(d,q)$ and any $s$ as above,
\begin{align}
\biggl|N^q\biggl(\phi_N(s)-f_0(s)-\frac{1}{N}f_1(s)-\cdots-\frac{1}{N^{q-1}}f_{q-1}(s)\biggr)\biggr|\leq L_q^{|s|}\,. \label{ineqHq}
\end{align}
\end{lmm} 
\begin{proof}
The proof is by induction on $q$. Let $H_{0, N}(s):=\phi_N(s)$ and let 
\begin{align*}
H_{q, N}(s):=N^q\biggl(\phi_N(s)-f_0(s)-\frac{1}{N}f_1(s)-\cdots-\frac{1}{N^{q-1}}f_{q-1}(s)\biggr)% \label{hj}
\end{align*}
for $1\leq q \leq k$. Since $|W_{l}|\leq N$ for any loop $l$, $|H_{0, N}(s)|=|\phi_N(s)|\leq 1$. So the inequality \eqref{ineqHq} holds for $q=0$. Also note that since $\phi_N(\emptyset)=1$ for all $N$, $f_0(\emptyset)=1$ and $f_{j}(\emptyset)=0$ for all $j\geq 1$. For the rest of the proof we will assume that $s\neq \emptyset$. Let $q\geq 1$ and suppose that \eqref{ineqHq} holds for all $0\leq q'<q$. By using the identity
\begin{align}
\dfrac{N^r}{N-1}&=\dfrac{1}{N-1}+\sum_{i=0}^{r-1}N^{i} \label{triv}
\end{align}
that holds for any $r\geq 0$, we can write the equation \eqref{mlequns} as
\begin{align}
mN^q\phi_N(s)&=\biggl(\dfrac{1}{N-1}+\sum_{i=0}^{q-1}N^{i} \biggr)\sum_{\text{twist}\,s}\phi_N+\biggl(\dfrac{1}{N-1}+\sum_{i=0}^{q-2}N^{i}\biggr)\sum_{\text{merge}\,s}\phi_N\notag\\
&\quad+\biggl(\dfrac{1}{N-1}+\sum_{i=0}^{q}N^{i}\biggr)\sum_{\text{split}\,s}\phi_N+\beta\biggl(\dfrac{1}{N-1}+\sum_{i=0}^{q}N^{i}\biggr)\sum_{\text{deform}\,s}\phi_N\,. \label{expphi}
\end{align}
On the other hand, by equation \eqref{lform},
\begin{align}
mN^q\sum\limits_{i=0}^{q-1}\dfrac{1}{N^{i}}f_{i}(s)&=\sum\limits_{i=0}^{q-1}N^{q-i}\biggl(\sum_{\text{twist}\,s}\sum_{j=0}^{i-1}f_j+\sum_{\text{merge}\,s}\sum_{j=0}^{i-2}f_j+\sum_{\text{split}\,s}\sum_{j=0}^{i}f_j+\beta\sum_{\text{deform}\,s}\sum_{j=0}^{i}f_j\biggr)\,.\notag
\end{align}
By interchanging sums we get
\begin{align*}
\sum\limits_{i=0}^{q-1}N^{q-i}\sum_{\text{twist}\,s}\sum_{j=0}^{i-1}f_j&=\sum_{\text{twist}\,s}\sum_{j=0}^{q-2}f_j\sum_{i=j+1}^{q-1} N^{q-i}\\
&=\sum_{\text{twist}\,s}\sum_{j=0}^{q-2}f_j\sum_{i=j+1}^{q-1} N^{i-j}=\sum_{\text{twist}\,s}\sum_{i=1}^{q-1}N^i\sum_{j=0}^{i-1}\frac{1}{N^j}f_j\,.
\end{align*}
After doing this change of summation also for merger, splitting and deformation terms, we get
\begin{align}
mN^q\sum\limits_{i=0}^{q-1}\dfrac{1}{N^{i}}f_{i}(s)&=\sum_{\text{twist}\,s}\sum_{i=1}^{q-1}N^i\sum_{j=0}^{i-1}\frac{1}{N^j}f_j+\sum_{\text{merge}\,s}\sum_{i=1}^{q-2}N^i\sum_{j=0}^{i-2} \frac{1}{N^j}f_j\notag\\
&\quad +\sum_{\text{split}\,s}\sum_{i=1}^{q}N^i\sum_{j=0}^{i} \frac{1}{N^j}f_j+\beta\sum_{\text{deform}\,s}\sum_{i=1}^{q}N^i\sum_{j=0}^{i}\frac{1}{N^j}f_j\,.\label{expf}
\end{align}
The left side of \eqref{expf} subtracted from the left side of  \eqref{expphi} gives
\[
mN^q\biggl(\phi_N(s)-\sum\limits_{i=0}^{q-1}\dfrac{1}{N^{i}}f_{i}(s)\biggr)=mH_{q, N}(s)\,.
\]
On the other hand, we get four terms when the right side of~\eqref{expf} is subtracted from the right side of~\eqref{expphi}, corresponding to twisting, merger, splitting and deformation. The twisting term simplifies~as:
\begin{align*}
\sum_{\text{twist}\,s}\biggl(\dfrac{1}{N-1}\phi_N+\sum_{i=0}^{q-1}N^{i} \biggl(\phi_N-\sum_{j=0}^{i-1}\frac{1}{N^j}f_j\biggr)\biggr)&=\sum_{\text{twist}\,s}\biggl(\sum_{i=0}^{q-1}H_{i,N}+\dfrac{1}{N-1}\phi_N\biggr)
\end{align*}
Similarly simplifying the other terms, we finally get 
\begin{align}
mH_{q,N}(s)&=\sum_{\text{twist}\,s}\biggl(\sum_{i=0}^{q-1}H_{i,N}+\dfrac{1}{N-1}\phi_N\biggr)+\sum_{\text{merge}\,s}\biggl(\sum_{i=0}^{q-2}H_{i,N}+\dfrac{1}{N-1}\phi_N\biggr)\notag\\
&\qquad+\sum_{\text{split}\,s}\biggl(\sum_{i=0}^{q}H_{i,N}+\dfrac{1}{N-1}\phi_N\biggr)+\beta\sum_{\text{deform}\,s}\biggl(\sum_{i=0}^{q}H_{i,N}+\dfrac{1}{N-1}\phi_N\biggr)\,. \notag
\end{align}
By writing the above equation also for $q-1$ we get
\begin{align}
mH_{q,N}(s)&=mH_{q-1,N}(s)+\sum_{\text{twist}\,s}H_{q-1,N}+\sum_{\text{merge}\,s}H_{q-2,N}+\sum_{\text{split}\,s}H_{q,N}+\beta\sum_{\text{deform}\,s}H_{q,N}\, ,\label{equationhk}
\end{align}
where $H_{-1}(s)$ is assumed to be $\phi_N(s)/N$ when $q=1$. By induction hypothesis, $|H_{q',N}(s)|\leq L_{q'}^{|s|} \leq L_{q-1}^{|s|}$ for all $0\leq q'<q$ and all $s$. Note that $H_{-1}(s)$ is also bounded by $L_{q-1}^{|s|}$. Therefore, if we let $\ell(s):=L_{q-1}^{|s|}$, then by the triangle inequality and equation \eqref{equationhk},
\begin{align}
m|H_{q,N}(s)|&\leq m\ell(s)+\sideset{}{^+}\sum_{\text{twist}\,s}\ell+\sideset{}{^+}\sum_{\text{merge}\,s}\ell+\sideset{}{^+}\sum_{\text{split}\,s}|H_{q,N}|+|\beta|\sideset{}{^+}\sum_{\text{deform}\,s}|H_{q,N}|\,. \label{ineq1}
\end{align}
We will first compute upper bounds for the second and third terms on the right. By Lemma \ref{twist}, if $s' \in \ftw(s)$, then $|s'|\leq |s|$ and therefore $\ell(s')\leq \ell(s)$. Since there can be at most $m^2$ twistings of $l_1$ using edges $e$ and $e^{-1}$, we get
\begin{align*}
&\sideset{}{^+}\sum_{\text{twist}\,s}\ell \leq  2m^2\ell(s) \leq 2m|s|L_{q-1}^{|s|}\,.
\end{align*}
By Lemma~\ref{merger}, if $s'\in \fst(s)$, then $|s'|\leq |s|$ and so $\ell(s')\leq \ell(s)$. If $s=(l_1, \ldots, l_n)$ is the minimal representation of $s$, then for any $2\leq i\leq n$, the loop $l_i$ can be merged to $l_1$ at location $e$ or $e^{-1}$ in at most $m|l_i|$ ways. So total number of mergers of $l_1$ with other loops in $s$ at location $e$ or $e^{-1}$ is no more than $m|s|$. Hence
\begin{align*}
\sideset{}{^+}\sum_{\text{merge}\,s}\ell \leq m|s|\ell(s)=m|s|L_{q-1}^{|s|}\,.
\end{align*}
%By Lemma \ref{split1} and Lemma \ref{split2}, if $s' \in \fs(s)$, then $|s'|\leq |s|$ and $\ell(s')\leq \ell(s)$. Since there can be at most $m^2$ splittings of $l_1$ at locations $e$ and $e^{-1}$, 
%\begin{align*}
%\sideset{}{^+}\sum_{\text{split}\,s}\ell \leq 2m^2\ell(s)\leq 2m|s|L^{|s|}\,.
%\end{align*}
%Finally, by Lemma \ref{deform}, if $s' \in \fd(s)$, then $|s'|\leq |s|+4$ and $\ell(s')\leq L^4\ell(s)$. Since $|\cp^{+}(e)|\leq 2(d-1)$, there are less than $4md$ deformations of $l_1$ at locations $e$ and $e^{-1}$. Thus
%\begin{align*}
%\sideset{}{^+}\sum_{\text{deform}\,s}\ell \leq 4mdL^4\ell(s)\leq md|s|L^{|s|+4}\,.
%\end{align*}
By combining the last two inequalities with inequality \eqref{ineq1}, we get
\begin{align}
|H_{q,N}(s)|&\leq 4|s|L_{q-1}^{|s|}+\dfrac{1}{m}\sideset{}{^+}\sum_{\text{split}\,s}|H_{q,N}|+\dfrac{|\beta|}{m}\sideset{}{^+}\sum_{\text{deform}\,s}|H_{q,N}|\,. \label{ineqhq}
\end{align}
Let $\Delta$ be the set of all finite sequences of integers, including the null sequence, and $\Delta^+$ be its subset consisting of all sequences whose components are all $\geq 4$. Given two non-null elements $\delta = (\delta_1,\ldots, \delta_n)\in \Delta$ and $\delta' = (\delta_1',\ldots, \delta_m')\in \Delta$, we will say that $\delta \le \delta'$ if $m=n$ and $\delta_i\le\delta_i'$ for each~$i$. If $(l_1, \ldots, l_n)$ is the minimal representation of $s$, define $\delta(s):=(|l_1|, \ldots, |l_n|)$ to be the degree sequence of $s$.   Given a vector $\delta = (\delta_1,\ldots, \delta_n)\in \Delta$, define
\[
|\delta| := \sum_{i=1}^n \delta_i\, , \ \ \#\delta := n\, , \ \ \text{and} \ \ \iota(\delta) := |\delta|-\#\delta\, .
\]
All of the above quantities are defined to be zero for the empty sequence. Note that $\iota(s)=\iota(\delta(s))$. Define
$$D_N(\delta):=\sup\limits_{s \in \cs\, : \, \delta(s)\leq \delta}|H_{q,N}(s)|$$
for $\delta \in \Delta^+$ and let $D_N(\delta)=0$ for $\delta\in \Delta\setminus \Delta^+$. For $\lambda\in (0, 1)$, let
\[
F_N(\lambda) := \sum_{\delta\in \Delta^+} \lambda^{\iota(\delta)} D_N(\delta) =  \sum_{\delta\in \Delta} \lambda^{\iota(\delta)} D_N(\delta)\, .
\]
We claim that $F_N(\lambda)<\infty$ for all sufficiently small $\lambda$. To prove this, observe that by the induction hypothesis stated at the beginning of the proof of the lemma,
\begin{align*}
|H_{q,N}(s)|&=|NH_{q-1,N}(s)-Nf_{q-1}(s)|\leq 2NL_{q-1}^{|s|}\,,
\end{align*}
which implies that $D_N(\delta)\leq 2NL_{q-1}^{|\delta|}$. Note that the number of $\delta \in \Delta^{+}$ with $|\delta|=r$ and $\#\delta=n$ is less than or equal to $\binom{r}{n-1}$, and if $\delta$ is a such sequence, then $r\geq 4n$. Therefore
\begin{align}
F_N(\lambda)&\leq 2N\sum_{\delta\in \Delta^+} \lambda^{\iota(\delta)}L_{q-1}^{|\delta|}\nonumber\\
&=2N\sum_{r=1}^\infty \sum_{n=1}^r\sum_{\substack{\delta\in \Delta^+\, : \, |\delta|= r,\nonumber\\
\#\delta = n}} \lambda^{r-n} L_{q-1}^{r}\nonumber\\
&\leq 2N\sum_{r=1}^\infty \sum_{n=1}^r \lambda^{3r/4} L_{q-1}^{r}\binom{r}{n-1}\leq 2N\sum_{r=1}^\infty \lambda^{3r/4} L_{q-1}^{r}2^r\,.\label{fnlambda}
\end{align}
This shows that $F_N(\lambda)<\infty$ for all $\lambda<(2L_{q-1})^{-4/3}$.

%Now we are going to bound the last two terms on the right hand side of inequality \eqref{ineqhq}. 
Take some $\delta=(\delta_1, \ldots, \delta_n)\in \Delta^+$ and let $s$ be a non-null loop sequence with minimal representation $s=(l_1, \ldots, l_n)$, such that $\delta(s)\leq \delta$. Recall that
\begin{align}
\dfrac{1}{m}\sideset{}{^+}\sum_{\text{split}\,s}|H_{q,N}|&=\dfrac{1}{m}\sum_{x\in A_1, \, y\in B_1} |H_{q,N}(\times_{x,y}^1 l_1,\times_{x,y}^2 l_1, l_2, \ldots,  l_n)|\nonumber\\
&\quad + \dfrac{1}{m}\sum_{x\in B_1, \, y\in A_1} |H_{q,N}(\times_{x,y}^1 l_1,\times_{x,y}^2 l_1, l_2, \ldots,  l_n)|\nonumber\\
&\quad  +\dfrac{1}{m}\sum_{\substack{x,y\in A_1\\ x\ne y}} |H_{q,N}(\times^1_{x,y} l_1,\times^2_{x,y} l_1, l_2, \ldots, l_n)|\nonumber\\
&\quad  +\dfrac{1}{m}\sum_{\substack{x,y\in B_1\\ x\ne y}} |H_{q,N}(\times^1_{x,y} l_1,\times^2_{x,y} l_1, l_2, \ldots, l_n)|\, .\label{tt0}
\end{align}
By Lemma \ref{split1},
\begin{align}
&\frac{1}{m}\sum_{\substack{x,y\in A_1\\x\ne y}} |H_{q,N}(\times_{x,y}^1 l_1, \times^2_{x,y} l_1, l_2,\ldots, l_n)|\nonumber\\
&\le \frac{1}{m}\sum_{x\in A_1} \sum_{y\in A_1\backslash\{x\}} D_N(\delta_1-|y-x|, |y-x|, \delta_2,\ldots,\delta_n)\nonumber\\
&\le \frac{2}{m}\sum_{x\in A_1} \sum_{k=1}^{\infty} D_N(\delta_1-k, k, \delta_2,\ldots,\delta_n)\nonumber\\
&\le 2\sum_{k=1}^{\infty} D_N(\delta_1-k, k, \delta_2,\ldots,\delta_n)\, .\label{tt1}
\end{align}
The same inequality holds when $A_1$ is replaced with $B_1$. By Lemma  \ref{split2},
\begin{align}
&\frac{1}{m}\sum_{\substack{x\in A_1, \, y\in B_1}} |H_{q,N}(\times_{x,y}^1 l_1, \times^2_{x,y} l_1, l_2,\ldots, l_n)|\nonumber\\
&\le \frac{1}{m}\sum_{x\in A_1} \sum_{y\in B_1} D_N(\delta_1-|x-y|-1, |x-y|-1, \delta_2,\ldots,\delta_n)\nonumber\\
&\le \frac{2}{m}\sum_{x\in A_1} \sum_{k=1}^{\infty} D_N(\delta_1-k-1, k-1, \delta_2,\ldots,\delta_n)\nonumber\\
&\le 2\sum_{k=1}^{\infty} D_N(\delta_1-k-1, k-1, \delta_2,\ldots,\delta_n)\, .\label{tt2}
\end{align}
Again, the same estimate is valid if $A_1$ and $B_1$ are swapped. By using the estimates from \eqref{tt1} and \eqref{tt2} in \eqref{tt0}, we get
\begin{align}
\dfrac{1}{m}\sideset{}{^+}\sum_{\text{split}\,s}|H_{q,N}|&\leq 4\sum_{k=1}^{\infty} D_N(\delta_1-k, k, \delta_2,\ldots,\delta_n)+ 4\sum_{k=1}^{\infty} D_N(\delta_1-k-1, k-1, \delta_2,\ldots,\delta_n)\,. \label{tt3}
\end{align}
Next, recall that
\begin{align*}
\dfrac{1}{m}\sideset{}{^+}\sum_{\text{deform}\,s}|H_{q,N}(s)|&=\dfrac{1}{m}\sum_{p\in \cp^+(e)}\sum_{x\in C_1}|H_q(l_1 \ominus_{x}p, l_2, \ldots,l_n)|\\
&\quad   + \dfrac{1}{m}\sum_{p\in \cp^+(e)}\sum_{x\in C_1}|H_{q,N}(l_1 \oplus_{x}p, l_2, \ldots, l_n)| \, .
\end{align*}
Note that $|l_1 \ominus_{x}p|\leq |l_1|+4$ for any $x\in C_1$ and $p\in  \cp^+(e)$. So, if $l_1 \ominus_{x}p$ is non-null, then
\[
|H_{q,N}(l_1 \ominus_{x}p, l_2, \ldots,l_n)| \leq D_N(\delta_1+4, \delta_2, \ldots, \delta_n)
\]
and if $l_1 \ominus_{x}p$ is a null loop and $n\neq 1$, then
\[
|H_{q,N}(l_1 \ominus_{x}p, l_2, \ldots,l_n)| \leq D_N(\delta_2, \ldots, \delta_n)\,.
\]
Note that even if $n=1$ and $l_1 \ominus_{x}p$ is a null loop then both sides of above inequality are 0. Hence
\[
|H_{q,N}(l_1 \ominus_{x}p, l_2, \ldots,l_n)| \leq D_N(\delta_1+4, \delta_2, \ldots, \delta_n)+ D_N(\delta_2, \ldots, \delta_n)\,.
\]
The same bound also holds for $|H_{q,N}(l_1 \oplus_{x}p,\ldots,l_n)|$. Since $|\cp^{+}(e)|\leq 2(d-1)$, we get
\begin{align}
\dfrac{1}{m}\sideset{}{^+}\sum_{\text{deform}\,s}|H_{q,N}|&\leq 4dD_N(\delta_1+4, \delta_2, \ldots, \delta_n)+ 4dD_N(\delta_2, \ldots, \delta_n)\,.\label{tt4}
\end{align}
For $k\geq 1$, define
\begin{align*}
&\theta_k(\delta_1, \ldots, \delta_n):=(\delta_1-k, k, \delta_2, \ldots, \delta_n)\, ,\\
&\eta_k (\delta_1, \ldots, \delta_n):=(\delta_1-k-1, k-1, \delta_2, \ldots, \delta_n)\,.
\end{align*}
Moreover, let
\begin{align*}
&\alpha(\delta_1, \ldots, \delta_n):=(\delta_1+4,\delta_2, \ldots, \delta_n)\, ,\\
&\gamma(\delta_1, \ldots, \delta_n):=(\delta_2, \ldots, \delta_n)\,,
\end{align*}
with $\gamma(\delta_1)=\emptyset$. Then by \eqref{ineqhq}, \eqref{tt3} and \eqref{tt4},
\begin{align*}
|H_{q,N}(s)|&\leq 4|s|L_{q-1}^{|s|}+4\sum_{k=1}^{\infty}D_N(\theta_k(\delta))+4\sum_{k=1}^{\infty}D_N(\eta_k(\delta))\\
&\qquad  +4d|\beta|D_N(\alpha(\delta))+ 4d|\beta|D_N(\gamma(\delta))\,.
\end{align*}
Since this holds for all $s$ such that $\delta(s)\le \delta$, we can take supremum over all such $s$ and get
\begin{align*}
D_N(\delta)&\leq 4|\delta|L_{q-1}^{|\delta|}+4\sum_{k=1}^{\infty}D_N(\theta_k(\delta))+4\sum_{k=1}^{\infty}D(\eta_k(\delta))\\
&\qquad  +4d|\beta|D_N(\alpha(\delta))+ 4d|\beta|D_N(\gamma(\delta))\,.
\end{align*}
Therefore
\begin{align}
F_N(\lambda)&\leq 4\sum_{\delta\in \Delta^+}\lambda^{\iota(\delta)} |\delta|L_{q-1}^{|\delta|}\notag+\sum_{\delta\in \Delta^+}\lambda^{\iota(\delta)}\biggl(4\sum_{k=1}^{\infty}D_N(\theta_k(\delta))+4\sum_{k=1}^{\infty}D_N(\eta_k(\delta))\nonumber\\
&\qquad \qquad \qquad +4d|\beta|D_N(\alpha(\delta))+ 4d|\beta|D_N(\gamma(\delta))\biggr)\,. \label{mm0}
\end{align}
Proceeding as in the derivation of \eqref{fnlambda}, we get, for $\lambda<(4L_{q-1})^{-4/3}$,
\begin{align}
\sum_{\delta\in \Delta^+} \lambda^{\iota(\delta)}|\delta|L_{q-1}^{|\delta|}&=\sum_{r=1}^\infty \sum_{n=1}^r\sum_{\substack{\delta\in \Delta^+\, : \, |\delta|= r,\\
\#\delta = n}} \lambda^{r-n} r L_{q-1}^{r} \notag \\
&\leq \sum_{r=1}^\infty \sum_{n=1}^r \lambda^{3r/4} r L_{q-1}^{r} \binom{r}{n-1}\notag\\
&\leq \sum_{r=1}^\infty  \lambda^{3r/4} L_{q-1}^{r}2^{2r}=\dfrac{4\lambda^{3/4}L_{q-1}}{1-4\lambda^{3/4}L_{q-1}}\,. \label{mm1}
\end{align}
Next, observe that maps $\theta_1, \theta_2, \ldots$ are injective and their ranges are disjoint since the second component of $\theta_k$ is equal to $k$. Since $\iota(\theta_k(\delta))=\iota(\delta)-1$, this shows that
\begin{align}
\sum_{\delta\in \Delta^+} \sum_{k=1}^{\infty}\lambda^{\iota(\delta)}D_N(\theta_k(\delta))&= \sum_{\delta\in \Delta^+} \sum_{k=1}^{\infty}\lambda^{\iota(\theta_k(\delta))+1}D_N(\theta_k(\delta)) \notag \\
&\leq \sum_{\delta\in \Delta} \lambda^{\iota(\delta)+1}D_N(\delta)=\lambda F_N(\lambda)\,. \label{mm2}
\end{align}
Similarly, maps $\eta_1, \eta_2, \ldots$ are also injective and their ranges are disjoint. Since $\iota(\eta_k(\delta))=\iota(\delta)-3$, we get
\begin{align}
\sum_{\delta\in \Delta^+} \sum_{k=1}^{\infty}\lambda^{\iota(\delta)}D_N(\eta_k(\delta))&= \sum_{\delta\in \Delta^+} \sum_{k=1}^{\infty}\lambda^{\iota(\eta_k(\delta))+3}D_N(\eta_k(\delta)) \notag \\
&\leq \sum_{\delta\in \Delta} \lambda^{\iota(\delta)+3}D_N(\delta)=\lambda^3 F_N(\lambda)\,. \label{mm3}
\end{align}
The map $\alpha$ is injective and $\iota(\alpha(\delta))=\iota(\delta)+4$. Therefore
\begin{align}
\sum_{\delta\in \Delta^+} \lambda^{\iota(\delta)}D_N(\alpha(\delta))&= \sum_{\delta\in \Delta^+} \lambda^{\iota(\alpha(\delta))-4}D_N(\alpha(\delta)) \notag \\
&\leq \sum_{\delta\in \Delta} \lambda^{\iota(\delta)-4}D_N(\delta)=\lambda^{-4} F_N(\lambda) \label{mm4}
\end{align}
Finally, observe that $\gamma(\delta_1, \ldots, \delta_n) \in \Delta^+ $ for any $(\delta_1, \ldots, \delta_n)\in \Delta^+$ and
$$\gamma^{-1}(\delta_1, \ldots, \delta_n) \subseteq \{(k, \delta_1, \ldots, \delta_n): k\geq 4\}\,.$$
Thus,
\begin{align}
\sum_{\delta\in \Delta^+} \lambda^{\iota(\delta)}D_N(\gamma(\delta))&=\sum_{\delta\in \Delta^+} \lambda^{\iota(\gamma(\delta))+\delta_1-1}D_N(\gamma(\delta)) \notag \\
&\leq \sum_{\delta'\in \Delta^{+}}\sum_{k=1}^{\infty} \lambda^{\iota(\delta')+k-1}D_N(\delta')=\dfrac{F_N(\lambda)}{1-\lambda}\label{mm5}
\end{align}
Combining the inequalities \eqref{mm1}, \eqref{mm2}, \eqref{mm3}, \eqref{mm4}, \eqref{mm5} with \eqref{mm0}, we get
\begin{align*}
F_N(\lambda) &\leq \dfrac{16\lambda^{3/4}L_{q-1}}{1-4\lambda^{3/4}L_{q-1}}+\biggl(4\lambda^3+4\lambda+\dfrac{4|\beta|d}{\lambda^4}+\dfrac{4|\beta|d}{1-\lambda}\biggr)F_N(\lambda)\,.
\end{align*}
If we let $\lambda=(36L_{q-1})^{-4/3}$, then $4\lambda^3+4\lambda<1/4$ and 
\[
\dfrac{16\lambda^{3/4}L_{q-1}}{1-4\lambda^{3/4}L_{q-1}}=\frac{1}{2}\, .
\]
By choosing $|\beta|$ so that
\[
\dfrac{4|\beta|d}{\lambda^4}+\dfrac{4|\beta|d}{1-\lambda}<\dfrac{1}{4}\, ,
\]
we get $F_N(\lambda)\leq 1$. Therefore
\begin{align*}
|H_{q,N}(s)|\leq |D_N(\delta(s))|&\leq \lambda^{-\iota(\delta(s))}\leq (36L_{q-1})^{4|s|/3}=L_q^{|s|}\,,
\end{align*}
which finishes the proof of the lemma.
\end{proof}
% % % % % % % % % % % % % % % % % % % % % % % % % % % % % % % % % % % % % % % % % % % % % % % % % % % % %
The main consequence of Lemma \ref{boundHq} is that there is a subsequence of $N$'s through which the limit of $H_{k,N}(s)$ 
exists for all $s$. This is easy to prove using the bound given in Lemma \ref{boundHq} and a standard diagonal argument. Let $f_k(s)$ denote such a subsequential limit. By the induction hypothesis, 
\[
f_{k'}(s) = \lim_{N\ra \infty} H_{k',N}(s)
\]
for each $k'<k$ and $s\in \cs$. So by equation \eqref{equationhk}, we get 
\begin{align}
mf_{k}(s) &= mf_{k-1}(s)+\sum_{\text{twist }s} f_{k-1}+ \sum_{\text{merge }s}f_{k-2}+\sum_{\text{split }s}f_{k}+ \beta \sum_{\text{deform }s}f_{k}\, ,\notag
\end{align}
which is exactly equation \eqref{limequns}. 
Thus, any subsequential limit $f_k$ satisfies equation \eqref{limequns}. The convergence of $H_{k,N}$ would be proved, therefore, if we can show that there can be at most one function that satisfies \eqref{limequns}. For $k=0$, this  uniqueness was proved for sufficiently small $\beta$ in the following result from \cite{chatterjee15}. 
\begin{thm}\label{uniquef0uns}
Given any $L\ge 1$, there exists $\beta_1(L,d)>0$ such that if $|\beta| \le\beta_1(L,d)$, then there is a unique function $f_0:\cs \ra \rr$ such that \textup{(a)} $f_0(\emptyset)=1$, \textup{(b)} $|f_0(s)|\le L^{|s|}$ for all $s$,  and \textup{(c)} $f_0$ satisfies the master loop equation \eqref{mastereqf0uns}. 
\end{thm}
The following theorem proves the uniqueness result for aribtrary positive $k$. The proof makes use of Theorem \ref{uniquef0uns}.
\begin{thm}\label{uniquefk}
Given any $k\ge 1$ and $L\ge 1$, there exists $\beta_1(L,d,k)>0$ such that if $|\beta| \le\beta_1(L,d,k)$, then there is a unique function $f_k:\cs \ra \rr$ such that \textup{(a)} $f_k(\emptyset)=0$, \textup{(b)} $|f_k(s)|\le L^{|s|}$ for all $s$,  and \textup{(c)} $f_k$ satisfies the loop equation \eqref{limequns}.
\end{thm}
\begin{proof}
We have already proved the existence of $f_k$ when $|\beta|$ is sufficiently small. Suppose that $f_k$ and $\tf_k$ both satisfy conditions (a), (b) and (c). Consider the function 
$$g(s):=f_k(s)-\tf_k(s)+f_0(s)\,.$$
By taking the difference of equation \eqref{limequns} for $f_k$ and $\tf_k$ and adding the master loop equation \eqref{mastereqf0uns} for $f_0$, we see that $g$ also satisfies \eqref{mastereqf0uns}. Moreover $g(\emptyset)=1$ and if $s\neq \emptyset$, then by the triangle inequality (observing that $|f_0(s)|=\lim_{N\ra\infty} |\phi_N(s)|\le 1$),  we get
$$g(s)\leq L^{|s|}+L^{|s|}+1\leq (2L+1)^{|s|}\,.$$ 
So by Theorem~\ref{uniquef0uns} we get that $g(s)=f_0(s)$ for all sufficiently small $\beta$. This is the same as saying that $f_k(s)=\tf_k(s)$.
\end{proof}
Theorem \ref{uniquefk}, together with the existence of $f_k$ that was proved earlier, completes the proof of Theorem \ref{eqfuns}. 
% % % % % % % % % % % % % % % % % % % % % % % % % % % % % % % % % % % % % % % % % % % % % % % % % % % % % %
%\vskip.5in
\section{The symmetrized master loop equation}
Theorem \ref{eqfuns} tells us that for each $k\ge 0$, there exists a function  $f_{k}: \cs \to  \rr$ and a positive constant $\beta_0(d,k)$ such  that when $|\beta|\le \beta_0(d,k)$,
\begin{align*}
f_k(s)=\lim\limits_{N\to \infty}N^k\biggl(\phi_N(s)-f_0(s)-\frac{1}{N}f_1(s)-\cdots-\frac{1}{N^{k-1}}f_{k-1}(s)\biggr) 
\end{align*}
for every $s\in \cs$. Moreover, it says that $f_k$ satisfies the master loop equation \eqref{limequns}. Observe that the equation \eqref{limequns} depends on the quantity $m$, which is related to the first edge $e$ of the first loop $l_1$ in the minimal representation of a loop sequence $s$. Since the choice of the first edge is made by an arbitrary rule, there is an unpleasant asymmetry in \eqref{limequns}. This asymmetry is removed by the following result. Recall the definition of $\ftw^+$, $\ftw^-$, $\fs^+$, $\fs^-$, $\fst^+$, $\fst^-$, $\fd^+$ and $\fd^-$ from Section \ref{stringsec}.
\begin{thm}\label{eqf}
The function $f_k$ of Theorem \ref{eqfuns} satisfies the symmetrized master loop equation
\begin{align}
|s|f_k(s) &= |s|f_{k-1}(s)+\sum_{s'\in \ftw^-(s)} f_{k-1}(s') - \sum_{s'\in \ftw^+(s)} f_{k-1}(s')\notag\\
&\qquad + \sum_{s'\in \fst^-(s)}f_{k-2}(s') - \sum_{s'\in \fst^+(s)}f_{k-2}(s') +\sum_{s'\in \fs^-(s)}f_k(s') - \sum_{s'\in \fs^+(s)}f_{k}(s')\notag\\
&\qquad + \beta \sum_{s'\in \fd^-(s)}f_k(s') - \beta \sum_{s'\in \fd^+(s)}f_k(s')\,, \label{limiteq}
\end{align}
where $f_{j}(s)=0$ for all $s$ if $j<0$. 
\end{thm} 
For $k=0$, the assertion of Theorem \ref{eqf} was proved in \cite{chatterjee15}. To be precise, the following result was proved. This will be needed in the proof of Theorem \ref{eqf}.
\begin{thm}\label{masterf0}
Then for any loop sequence $s$,
\begin{align}\label{mastereqf0}
|s|f_0(s) &=  \sum_{s'\in \fs^-(s)} f_0(s')- \sum_{s'\in \fs^+(s)} f_0(s') + \beta \sum_{s'\in \fd^-(s)} f_0(s')- \beta \sum_{s'\in \fd^+(s)} f_0(s')\, .
\end{align}
\end{thm} 
The following symmetrized version of Theorem \ref{mastern} was also proved in \cite{chatterjee15}. This too, will be needed in the proof of Theorem \ref{eqf}.
\begin{thm}\label{mastersymm}
Let $\phi_N = \phi_{\Lambda_N, N, \beta}$ be as in Theorem \ref{mainthmofpaper}. Take a loop sequence $s$  such that all vertices of $\zz^d$ that are at distance $\le 1$ from any of the component loops of $s$ are contained in $\Lambda_N$. Then $\phi_N$ satisfies
\begin{align}
(N-1)|s|\phi_N(s) &= \sum_{s'\in \ftw^-(s)} \phi_N(s') - \sum_{s'\in \ftw^+(s)} \phi_N(s') \notag\\
&\qquad + N\sum_{s'\in \fs^-(s)}\phi_N(s') - N\sum_{s'\in \fs^+(s)}\phi_N(s') \notag \\
&\qquad+ \frac{1}{N} \sum_{s'\in \fst^-(s)}\phi_N(s') - \frac{1}{N} \sum_{s'\in \fst^+(s)}\phi_N(s')\notag \\
&\qquad + N\beta \sum_{s'\in \fd^-(s)}\phi_N(s') - N\beta \sum_{s'\in \fd^+(s)}\phi_N(s')\, . \label{finitemaster}
\end{align}
\end{thm}

We are now ready to prove Theorem \ref{eqf}. 
\begin{proof}[Proof of Theorem \ref{eqf}]
We will use induction on $k$. It was already proven for $f_0$ in Theorem~\ref{masterf0}. Take $k\geq 1$ and suppose that the claim is true for all $k'<k$. By the identity \eqref{triv}, the master loop equation from  Theorem~\ref{mastersymm} can be rewritten as
\begin{align}
&|s|\phi_N(s)=\biggl(\dfrac{1}{N}+\cdots+\dfrac{1}{N^k}+\dfrac{1}{N^k(N-1)}\biggr)\biggl(\sum_{s'\in \ftw^-(s)}\phi_N(s')-\sum_{s'\in \ftw^+(s)}\phi_N(s')\biggr)\notag\\
&\qquad\qquad+\biggl(\dfrac{1}{N^2}+\cdots+\dfrac{1}{N^k}+\dfrac{1}{N^k(N-1)}\biggr)\biggl(\sum_{s'\in \fst^-(s)}\phi_N(s')-\sum_{s'\in \fst^+(s)}\phi_N(s')\biggr)\notag\\
&\qquad\qquad+\biggl(1+\dfrac{1}{N}+\cdots+\dfrac{1}{N^k}+\dfrac{1}{N^k(N-1)}\biggr)\biggl(\sum_{s'\in \fs^-(s)}\phi_N(s')-\sum_{s'\in \fs^+(s)}\phi_N(s')\biggr)\notag\\
&\qquad\qquad+\beta\biggl(1+\dfrac{1}{N}+\cdots+\dfrac{1}{N^k}+\dfrac{1}{N^k(N-1)}\biggr)\biggl(\sum_{s'\in \fd^-(s)}\phi_N(s')-\sum_{s'\in \fd^+(s)}\phi_N(s')\biggr)\,. \label{ss1}
\end{align}
Take any $k'< k$. By the induction hypothesis, equation \eqref{limiteq} can be written as
\begin{align}
|s|f_{k'}(s) &= \sum_{s'\in \ftw^-(s)} \biggl(f_{0}(s')+\cdots+f_{k'-1}(s')\biggr)-\sum_{s'\in \ftw^+(s)} \biggl(f_{0}(s')+\cdots+f_{k'-1}(s')\biggr)\\
&\qquad + \sum_{s'\in \fst^-(s)} \biggl(f_{0}(s')+\cdots+f_{k'-2}(s')\biggr)-\sum_{s'\in \fst^+(s)} \biggl(f_{0}(s')+\cdots+f_{k'-2}(s')\biggr)\notag\\
&\qquad + \sum_{s'\in \fs^-(s)} \biggl(f_{0}(s')+\cdots+f_{k'}(s')\biggr)-\sum_{s'\in \fs^+(s)} \biggl(f_{0}(s')+\cdots+f_{k'}(s')\biggr)\notag\\
&\qquad + \beta\sum_{s'\in \fd^-(s)} \biggl(f_{0}(s')+\cdots+f_{k'}(s')\biggr)-\beta \sum_{s'\in \fd^+(s)} \biggl(f_{0}(s')+\cdots+f_{k'}(s')\biggr)\,.\label{ss3}
\end{align}
Therefore, 
\begin{align}
&|s|\biggl(f_0(s)+\dfrac{1}{N}f_1(s)+\cdots+\dfrac{1}{N^{k-1}}f_{k-1}(s)\biggr)\notag\\
&\quad=\sum_{s'\in \ftw^-(s)} \sum_{j=0}^{k-2}\biggl(\dfrac{1}{N^{j+1}}+\cdots+\dfrac{1}{N^{k-1}}\biggr)f_j(s')-\sum_{s'\in \ftw^+(s)} \sum_{j=0}^{k-2}\biggl(\dfrac{1}{N^{j+1}}+\cdots+\dfrac{1}{N^{k-1}}\biggr)f_j(s')\notag\\
&\quad+\sum_{s'\in \fst^-(s)} \sum_{j=0}^{k-3}\biggl(\dfrac{1}{N^{j+2}}+\cdots+\dfrac{1}{N^{k-1}}\biggr)f_j(s')-\sum_{s'\in \fst^+(s)} \sum_{j=0}^{k-3}\biggl(\dfrac{1}{N^{j+2}}+\cdots+\dfrac{1}{N^{k-1}}\biggr)f_j(s')\notag\\
&\quad+\sum_{s'\in \fs^-(s)} \sum_{j=0}^{k-1}\biggl(\dfrac{1}{N^{j}}+\cdots+\dfrac{1}{N^{k-1}}\biggr)f_j(s')-\sum_{s'\in \fs^+(s)} \sum_{j=0}^{k-1}\biggl(\dfrac{1}{N^{j}}+\cdots+\dfrac{1}{N^{k-1}}\biggr)f_j(s')\notag\\
&\quad+\beta\sum_{s'\in \fd^-(s)} \sum_{j=0}^{k-1}\biggl(\dfrac{1}{N^{j}}+\cdots+\dfrac{1}{N^{k-1}}\biggr)f_j(s')-\beta\sum_{s'\in \fd^+(s)} \sum_{j=0}^{k-1}\biggl(\dfrac{1}{N^{j}}+\cdots+\dfrac{1}{N^{k-1}}\biggr)f_j(s')\,. \label{ss2}
\end{align}
Subtracting equation \eqref{ss1} from \eqref{ss2}, combining the corresponding terms and then multiplying both sides by $N^k$, we get
\begin{align}
&|s|H_{k,N}(s)=\sum_{s'\in \ftw^-(s)} \biggl(\sum_{j=0}^{k-1}H_{j,N}(s')+\dfrac{1}{N-1}\phi_N(s')\biggr)-\sum_{s'\in \ftw^+(s)}\biggl(\sum_{j=0}^{k-1}H_{j,N}(s')+\dfrac{1}{N-1}\phi_N(s')\biggr)\notag\\
&\quad+\sum_{s'\in \fst^-(s)} \biggl(\sum_{j=0}^{k-2}H_{j,N}(s')+\dfrac{1}{N-1}\phi_N(s')\biggr)-\sum_{s'\in \fst^+(s)}\biggl(\sum_{j=0}^{k-2}H_{j,N}(s')+\dfrac{1}{N-1}\phi_N(s')\biggr)\notag\\
&\quad +\sum_{s'\in \fs^-(s)} \biggl(\sum_{j=0}^{k}H_{j,N}(s')+\dfrac{1}{N-1}\phi_N(s')\biggr)-\sum_{s'\in \fs^+(s)}\biggl(\sum_{j=0}^{k}H_{j,N}(s')+\dfrac{1}{N-1}\phi_N(s')\biggr)\notag\\
&\quad +\beta\sum_{s'\in \fd^-(s)} \biggl(\sum_{j=0}^{k}H_{j,N}(s')+\dfrac{1}{N-1}\phi_N(s')\biggr)-\beta\sum_{s'\in \fd^+(s)}\biggl(\sum_{j=0}^{k}H_{j,N}(s')+\dfrac{1}{N-1}\phi_N(s')\biggr)\,.\notag
\end{align}
Now send $N\to \infty$ to get
\begin{align}
&|s|f_k(s)=\sum_{s'\in \ftw^-(s)}\sum_{j=0}^{k-1}f_{j}(s')-\sum_{s'\in \ftw^+(s)}\sum_{j=0}^{k-1}f_{j}(s')+\sum_{s'\in \fst^-(s)} \sum_{j=0}^{k-2}f_{j}(s')-\sum_{s'\in \fst^+(s)}\sum_{j=0}^{k-2}f_{j}(s')\notag\\
&\quad +\sum_{s'\in \fs^-(s)} \sum_{j=0}^{k}f_{j}(s')-\sum_{s'\in \fs^+(s)}\sum_{j=0}^{k}f_{j}(s')+\beta\sum_{s'\in \fd^-(s)} \sum_{j=0}^{k}f_{j}(s')-\beta\sum_{s'\in \fd^+(s)}\sum_{j=0}^{k}f_{j}(s')\,.\notag
\end{align}
By using equation \eqref{ss3} with $k'=k-1$ we see that above identity is equivalent to \eqref{limiteq}.
\end{proof}
% % % % % % % % % % % % % % % % % % % % % % % % % % % % % % % % % % % % % % % % % % % % % % % % % % % % %
% % % % % % % % % % % % % % % % % % % % % % % % % % % % % % % % % % % % % % % % % % % % % % % % % % % % %
% % % % % % % % % % % % % % % % % % % % % % % % % % % % % % % % % % % % % % % % % % % % % % % % % % % % %
%\vskip.5in
\section{Series Expansion}\label{secseries}
Recall from Section~\ref{stringsec} the definition of a trajectory $X$ and the weight $w_\beta(X)$ of a trajectory at inverse coupling strength $\beta$. Recall also the definition of $\mx_{i,k}(s)$ from Section~\ref{stringsec} and the definition of the number $m$ associated with a loop sequence $s$ from the beginning of Section~\ref{unsymmsec}. Define a collection of numbers $\{a_{i,k}(s): i,k\ge 0,\ s\in \cs\}$ inductively as follows. Let $a_{0,0}(\emptyset):=1$, $a_{i, 0}(\emptyset):=0$ for all $i\geq 1$ and $a_{0,0}(s):=0$ for every non-null $s$. If $k\ne 0$ let $a_{i,k}(\emptyset):=0$ for all $i\geq 0$ and $a_{0,k}(s):=0$ for all $s\in \cs$. Fix a triple $(s, i, k)$ and suppose we have defined $a_{i', k'}(s')$ for all triples $(s', i', k')$ such that either $k'<k$, or $k'=k$ and $i'<i$, or $k'=k$, $i'=i$ and $\iota(s')<\iota(s)$. Then let
\begin{align}
a_{i,k}(s):=a_{i, k-1}(s)+\dfrac{1}{m}\sum_{\text{twist}\, s} a_{i,k-1}+\dfrac{1}{m}\sum_{\text{merge}\, s} a_{i,k-2}+\dfrac{1}{m}\sum_{\text{split}\, s} a_{i,k}+\dfrac{1}{m}\sum_{\text{deform}\, s} a_{i-1,k}\label{recurcoeff}
\end{align}
where we use convention that $a_{i,k}(s)=0$ if $i<0$ or $k<0$. Note that by Lemma \ref{iota2}, the fourth term on the right is well-defined at the time of defining $a_{i,k}(s)$.  The goal of this section is to prove the following result.
\begin{thm}\label{series}
Let $f_k$ be as in Theorem \ref{eqfuns}. If  $|\beta|$ is sufficiently small (depending only on $d$ and $k$), then for any $s$,
\[
f_k(s)=\sum_{i=0}^{\infty}a_{i,k}(s)\beta^{i}\,.
\]
Moreover, the infinite series is absolutely convergent.
\end{thm}
The case $k=0$ of Theorem \ref{series} was established in \cite{chatterjee15}, where the following lemma about $a_{i,0}$ was also proved. Recall the set $\Delta$ of degree sequences and the associated notations defined inside the proof of Lemma \ref{boundHq}.
\begin{lmm}\label{coefff0}
Let $a_{i,0}$ be defined as in Theorem \ref{series}. There is a constant $K(d)\geq 4$ such that if $s$ is a loop sequence with degree sequence $\delta = (\delta_1,\ldots,\delta_n)$, then for any $i\ge 0$,
\begin{align*}
|a_{i,0}(s)|\leq K(d)^{5i+\iota(\delta)}C_{\delta_1-1}C_{\delta_2-1}\cdots C_{\delta_n-1}\,, %\label{bcoefff0}
\end{align*}
where $C_i$ is the $i^{\textup{th}}$ Catalan number. The product of Catalan numbers is interpreted as $1$ when~$s=\emptyset$.
\end{lmm}
The first step in the proof of Theorem \ref{series} is the following generalization of Lemma \ref{coefff0}.
\begin{lmm}\label{coeff}
There is a constant $K(d)\geq 4$ such that if $s$ is a loop sequence with degree sequence $\delta = (\delta_1,\ldots,\delta_n)$, then for any $i,k\ge 0$,
\begin{align*}
|a_{i,k}(s)|\leq K(d)^{(5+k)i+\iota(\delta)}|\delta|^{2k}C_{\delta_1-1}C_{\delta_2-1}\cdots C_{\delta_n-1}\,, %\label{boundaij}
\end{align*}
where $C_i$ is the $i^{\textup{th}}$ Catalan number. The product of Catalan numbers is interpreted as $1$ when~$s=\emptyset$.
\end{lmm}
\begin{proof}
We will use three fold induction: first over $k$, then over $i$ and finally over $\iota(s)$. The number $K=K(d)\geq 4$ will be chosen at the end of the proof. The case $k=0$ has already been established in Lemma~\ref{coefff0}. Let $k\geq 1$ and suppose that the claimed inequality holds for all triples $(s', i', k')$  with $k'<k$. Since $a_{0,k}(s)=0$ for all $s\in \cs$, the inequality is trivially true for triples $(s, 0, k)$. 

Take $i\geq 1$ and assume that the inequality holds for all triples $(s', i',k')$ with either $k'< k$, or $k'=k$ and $i'<i$. We will use induction on $\iota(s)$ to prove that it holds for $(s, i, k)$ for any loop sequence $s$. If $s=\emptyset$ then it is obviously true. Let $s\neq \emptyset$ and suppose that the bound on $|a_{i,k}(s')|$ holds for all $s'$ with $\iota(s')<\iota(s)$. Let $l=(l_1, \dots, l_n)$ be the minimal representation of $s$ and let $m, A_1, B_1$ and $C_1$ be as in Section \ref{unsymmsec}. By induction hypothesis,
\begin{align}
|a_{i, k-1}(s)|&\leq  K^{(5+k-1)i+\iota(\delta)}|\delta|^{2(k-1)}C_{\delta_1-1}C_{\delta_2-1}\cdots C_{\delta_n-1}\notag\\
&\le K^{(5+k)i-1+\iota(\delta)}|\delta|^{2k}C_{\delta_1-1}C_{\delta_2-1}\cdots C_{\delta_n-1}\,. \label{dd0}
\end{align}
By Lemma \ref{twist}, if $s' \in \ftw(s)$ then $|\delta(s')|\leq |\delta|$ and $\iota(\delta(s'))\leq \iota(\delta)$. So by the induction hypothesis,
\begin{align}
&\dfrac{1}{m}\sum_{\substack{x, y\in A_1\\x\ne y}} |a_{i,k-1}(\propto_{x,y} l_1, l_2, \ldots, l_n)|\notag\\
&\leq \dfrac{1}{m}\sum_{\substack{x, y\in A_1\\x\ne y}} K^{(5+k-1)i+\iota(\delta)}|\delta|^{2(k-1)}C_{\delta_1-1}C_{\delta_2-1}\cdots C_{\delta_n-1}\notag\\
&\leq m K^{(5+k)i-1+\iota(\delta)}|\delta|^{2(k-1)}C_{\delta_1-1}C_{\delta_2-1}\cdots C_{\delta_n-1}\notag\\
& \leq K^{(5+k)i-1+\iota(\delta)}|\delta|^{2k}C_{\delta_1-1}C_{\delta_2-1}\cdots C_{\delta_n-1}\,.\notag
\end{align}
The same bound also holds for the other terms in twisting sum. So by the triangle inequality,
\begin{align}
\dfrac{1}{m}\biggl|\sum_{\text{twist}\,s}a_{i,k-1}\biggr|\leq 4K^{(5+k)i-1+\iota(\delta)}|\delta|^{2k}C_{\delta_1-1}C_{\delta_2-1}\cdots C_{\delta_n-1}\,.\label{dd1}
\end{align}
By Lemma \ref{merger}, if $s'\in \fst(s)$ then $|\delta(s')|\leq |\delta|$ and $\iota(\delta(s')) \leq \iota(\delta)+1$. So by the induction hypothesis and Lemma \ref{app1},
\begin{align*}
&\dfrac{1}{m}\sum_{r=2}^n\sum_{x\in C_1, \, y\in C_r} |a_{i,k-2}(l_1 \ominus_{x,y} l_r, l_2, \ldots, l_{r-1}, l_{r+1}, \ldots, l_n)| \\
&\leq \dfrac{1}{m}\sum_{r=2}^n\sum_{x\in C_1, \, y\in C_r}  K^{(5+k-2)i+\iota(\delta)+1}|\delta|^{2(k-2)}C_{\delta_1+\delta_r-1}C_{\delta_2-1}\cdots C_{\delta_{r-1}-1}C_{\delta_{r+1}-1}\cdots C_{\delta_n-1}\notag\\
&\leq \dfrac{1}{m}\sum_{r=2}^n |C_1||C_r|K^{(5+k-2)i+\iota(\delta)+1}|\delta|^{2(k-2)}(\delta_1+\delta_r)^2C_{\delta_1-1}\cdots  C_{\delta_n-1}\notag\\
&\leq K^{(5+k-2)i+\iota(\delta)+1}|\delta|^{2(k-2)}|\delta|^3C_{\delta_1-1}\cdots  C_{\delta_n-1}\notag\\
&\leq K^{(5+k)i-1+\iota(\delta)}|\delta|^{2k}C_{\delta_1-1}\cdots  C_{\delta_n-1}\,.\notag\\
\end{align*}
The same bound also holds for the second term in the merger sum. So by the triangle inequality,
\begin{align}
\dfrac{1}{m}\biggl|\sum_{\text{merge}\,s}a_{i,k-2} \biggr|\leq 2K^{(5+k)i-1+\iota(\delta)}|\delta|^{2k}C_{\delta_1-1}C_{\delta_2-1}\cdots C_{\delta_n-1}\,.\label{dd2}
\end{align}
By Lemma \ref{split2}, the induction hypothesis and the identity \eqref{catalan},
\begin{align}
&\frac{1}{m}\sum_{x\in A_1,\, y\in B_1} |a_{i, k}(\times_{x,y}^1 l_1,\times_{x,y}^2l_1,l_2,\ldots, l_n) |\nonumber\\
&\le \frac{1}{m}\sum_{x\in A_1,\, y\in B_1}K^{(5+k)i+\iota(\delta)-3} |\delta|^{2k} C_{\delta_1 - |x-y|-2} C_{|x-y|-2}C_{\delta_2-1}\cdots C_{\delta_n-1}\nonumber\\
&\le \frac{2}{m}\sum_{x\in A_1} \sum_{r=2}^{\delta_1-2} K^{(5+k)i+\iota(\delta)-3}|\delta|^{2k} C_{\delta_1 - r-2} C_{r-2}C_{\delta_2-1}\cdots C_{\delta_n-1}\nonumber\\
&\le 2K^{(5+k)i+\iota(\delta)-3} |\delta|^{2k} C_{\delta_1-1}\cdots C_{\delta_n-1}\, .\label{cc1}
\end{align}
The same bound holds if $A_1$ and $B_1$ are swapped. 
Similarly, Lemma \ref{split1}, the induction hypothesis and the identity \eqref{catalan},
\begin{align}
&\frac{1}{m}\sum_{\substack{x,y\in A_1\\ x\ne y}} |a_{i, k}(\times_{x,y}^1 l_1,\times_{x,y}^2l_1,l_2,\ldots, l_n) |\nonumber\\
&\le \frac{1}{m}\sum_{\substack{x,y\in A_1\\ x\ne y}}K^{(5+k)i+\iota(\delta)-1} |\delta|^{2k} C_{\delta_1 - |x-y|-1} C_{|x-y|-1}C_{\delta_2-1}\cdots C_{\delta_n-1}\nonumber\\
&\le \frac{2}{m}\sum_{x\in A_1} \sum_{r=1}^{\delta_1-1} K^{(5+k)i+\iota(\delta)-1}|\delta|^{2k} C_{\delta_1 - r-1} C_{r-1}C_{\delta_2-1}\cdots C_{\delta_n-1}\nonumber\\
&\le 2K^{(5+k)i+\iota(\delta)-1} |\delta|^{2k} C_{\delta_1-1}\cdots C_{\delta_n-1}\, .\label{cc2}
\end{align}
The same bound holds if $A_1$ is replaced by $B_1$. So by combining \eqref{cc1} and \eqref{cc2} we get
\begin{align}
\dfrac{1}{m}\biggl|\sum_{\text{split}\,s} a_{i,k}(s)\biggr|\leq (4K^{-3}+4K^{-1}) K^{(5+k)i+\iota(\delta)} |\delta|^{2k} C_{\delta_1-1}\cdots C_{\delta_n-1}\, .\label{dd3}
\end{align}
By Lemma \ref{deform}, if $s'\in \fd(s)$ then $|\delta(s')|\leq |\delta|+4$ and $\iota(\delta(s'))\leq \iota(\delta)+4$. So for $K\geq 4$, the inequality~\eqref{catalan2} gives
\begin{align*}
|a_{i-1,k}(l_1\oplus_{x} p , l_2,\ldots, l_n)| &\le K^{(5+k)(i-1)+\iota(\delta)+4} (|\delta|+4)^{2k} C_{\delta_1+3} C_{\delta_2-1}\cdots C_{\delta_n-1}\nonumber\\
&\le K^{(5+k)i-1-k+\iota(\delta)}(2|\delta|)^{2k}4^{4}C_{\delta_1-1}\cdots C_{\delta_n-1}\\
&\leq 256 K^{(5+k)i-1+\iota(\delta)}|\delta|^{2k}C_{\delta_1-1}\cdots C_{\delta_n-1}\, .
\end{align*}
The same bound holds if $\oplus$ is replaced by $\ominus$. Since each edge can be deformed by $2(d-1)$ plaquettes, 
\begin{align}
\dfrac{1}{m}\biggl|\sum_{\text{deform}\,s} a_{i-1,k}(s)\biggr|\leq 1024d K^{(5+k)i-1+\iota(\delta)} |\delta|^{2k} C_{\delta_1-1}\cdots C_{\delta_n-1}\, .\label{dd4}
\end{align}
By combining equation \eqref{recurcoeff} with inequalities \eqref{dd0}, \eqref{dd1}, \eqref{dd2}, \eqref{dd3} and \eqref{dd4}, we get
\begin{align*}
|a_{i,k}(s)|\leq (4K^{-3}+11K^{-1}+1024dK^{-1})K^{(5+k)i+\iota(\delta)}|\delta|^{2k}C_{\delta_1-1}\cdots C_{\delta_n-1}\, .
\end{align*}
Choosing $K$ so large that 
\begin{align}
4K^{-3}+11K^{-1}+1024dK^{-1}\leq 1 \label{boundK}
\end{align}
finishes the proof. 
\end{proof}
% % % % % % % % % % % % % % % % % % % % % % % % % % % % % % % % % % % % % % % % % % % % % % % % % %
% % % % % % % % % % % % % % % % % % % % % % % % % % % % % % % % % % % % % % % % % % % % % % % % % % 
We are now ready to prove Theorem \ref{series}.
\begin{proof}[Proof of Theorem \ref{series}]
By Lemma \ref{coeff}, the given series is absolutely convergent if $|\beta|<K^{-(5+k)}$. Define
\begin{align*}
\psi_k(s):=\sum_{i=0}^{\infty} a_{i,k}(s)\beta^i \, .
\end{align*}
We will use induction on $k$ to prove that $\psi_k(s)=f_k(s)$ for all $s$. As mentioned before, the case $k=0$ has already been established in \cite{chatterjee15}. Take some $k\geq 1$ and assume $\psi_j(s)=f_j(s)$ for all $s\in \cs$ and $j<k$. By using the recursive equation \eqref{recurcoeff} and the convention that $a_{i,k}(s)=0$ for all $s$ when $i<0$, the above equation can be written as
\begin{align*}
\psi_k(s)&=\sum_{i=0}^{\infty} a_{i, k-1}(s)\beta^{i}+\dfrac{1}{m}\sum_{\text{twist}\, s}\sum_{i=0}^{\infty}a_{i,k-1}\beta^i+\dfrac{1}{m}\sum_{\text{merge}\, s}\sum_{i=0}^{\infty} a_{i,k-2}\beta^i\\
&\qquad+\dfrac{1}{m}\sum_{\text{split}\, s} \sum_{i=0}^{\infty} a_{i,k}\beta^i+\dfrac{1}{m}\sum_{\text{deform}\, s} \sum_{i=1}^{\infty} a_{i-1,k}\beta^{i}\,.
\end{align*}
Therefore,
\begin{align*}
\psi_k(s)&=\psi_{k-1}(s)+\dfrac{1}{m}\sum_{\text{twist}\, s}\psi_{k-1}+\dfrac{1}{m}\sum_{\text{merge}\, s}\psi_{k-2}+\dfrac{1}{m}\sum_{\text{split}\, s} \psi_{k}+\dfrac{\beta}{m}\sum_{\text{deform}\, s} \psi_{k}\,.
\end{align*}
By the induction hypothesis, this equation can be written as
\begin{align*}
\psi_k(s)&=f_{k-1}(s)+\dfrac{1}{m}\sum_{\text{twist}\, s}f_{k-1}+\dfrac{1}{m}\sum_{\text{merge}\, s}f_{k-2}+\dfrac{1}{m}\sum_{\text{split}\, s} \psi_{k}+\dfrac{\beta}{m}\sum_{\text{deform}\, s} \psi_{k}
\end{align*}
So $\psi_k$ satisfies equation \eqref{limequns}.  If $s$ is a non-null loop sequence with degree sequence $\delta=(\delta_1, \ldots, \delta_n)$ and $2|\beta|<K^{-(5+k)}$, then by Lemma \ref{coeff} and the fact that $a_{0,k}(s)=0$, 
\begin{align}
|\psi_k(s)|&\leq \sum_{i=0}^{\infty}|a_{i,k}(s)|\beta^i\leq \sum_{i=1}^{\infty}K^{(5+k)i+\iota(\delta)}|\delta|^{2k} C_{\delta_1-1}\cdots C_{\delta_n-1} \beta^i \notag\\
&\leq \sum_{i=1}^{\infty}(K^{(5+k)}\beta)^i K^{|\delta|}4^{|\delta|k}4^{|\delta|}\leq (4^{(1+k)}K)^{|\delta|}=(4^{(1+k)}K)^{|s|}\,. \label{boundfk}
\end{align}
Lastly, note that $\psi_k(\emptyset)=0$. 
Thus $\psi_k$ satisfies conditions (a), (b) and (c) of Theorem \ref{uniquefk}. Therefore $\psi_k=f_k$ for all sufficiently small $\beta$.
\end{proof}
The following result, which is a byproduct of the proof of Theorem \ref{series}, establishes an upper bound on the growth rate of $f_k(s)$ and hence completes the proof of part (iii) of Theorem~\ref{mainthmofpaper}.
% % % % % % % % % % % % % % % % % % % % % % % % % % % % % % % % % % % % % % % % % % % % % % % % % % % % %
\begin{prop}\label{growth}
Let $f_k$ be as in Theorem \ref{eqfuns}. Then for all $s$,
\[
|f_k(s)|\leq (2^{2k+13}d)^{|s|}\,.
\]
\end{prop}
\begin{proof}
Since $K=2048d$ satisfies the inequality \eqref{boundK}, the result follows from inequality \eqref{boundfk}.
\end{proof}
Theorem \ref{series} has also the following important corollary, which gives an alternative recursion relation for $a_{i,k}(s)$.
\begin{cor}\label{symmcor}
For any $i,k,s$,
\begin{align*}
a_{i,k}(s) &= a_{i, k-1}(s)+\dfrac{1}{|s|}\sum_{s' \in \ftw^-(s)}a_{i, k-1}(s') - \dfrac{1}{|s|}\sum_{s' \in \ftw^+(s)}a_{i, k-1}(s')\\
&\qquad + \dfrac{1}{|s|}\sum_{s' \in \fst^-(s)}a_{i, k-2}(s') - \dfrac{1}{|s|}\sum_{s' \in \fst^+(s)}a_{i, k-2}(s')\\
&\qquad+\dfrac{1}{|s|}\sum_{s' \in \fs^-(s)} a_{i, k}(s')- \dfrac{1}{|s|}\sum_{s' \in \fs^+(s)} a_{i, k}(s')\\
&\qquad+\dfrac{1}{|s|}\sum_{s' \in \fd^-(s)} a_{i-1, k}(s')-\dfrac{1}{|s|}\sum_{s' \in \fd^+(s)} a_{i-1, k}(s')\,,
\end{align*}
where, as usual, $a_{i,k}(s)$ is interpreted as zero if $i<0$ or $k<0$. 
\end{cor}
\begin{proof}
Simply take the power series expansion given in Theorem \ref{series}, apply it to both sides of the identity \eqref{limiteq} from Theorem \ref{eqf},  and equate the coefficients of $\beta^i$. 
\end{proof}

% % % % % % % % % % % % % % % % % % % % % % % % % % % % % % % % % % % % % % % % % % % % % % % % % % % % % % % % % % %
% % % % % % % % % % % % % % % % % % % % % % % % % % % % % % % % % % % % % % % % % % % % % % % % % % % % % % % % % % % 
% % % % % % % % % % % % % % % % % % % % % % % % % % % % % % % % % % % % % % % % % % % % % % % % % % % % % % % % % % %

%\vskip.5in
\section{Absolute convergence}
The goal of this section is to prove that infinite series in Theorem \ref{mainthmofpaper} is absolutely convergent. Recall the definitions of $w_\beta(X)$ and $\mx_{k}(s)$ from Section \ref{stringsec}.
\begin{thm}\label{absconv}
Take any $k\ge 0$. There exists $\beta_2(d,k)>0$ such that for any $|\beta|\leq \beta_2(d,k)$ and non-null loop sequence~$s$,
\begin{align}
\sum_{X\in \mx_k(s)}|w_{\beta}(X)|<\infty \,. \label{absconvineq}
\end{align}
\end{thm}
Inductively define a collection of numbers $\{b_{i,k}(s):i,k\ge 0, s\in \cs\}$ as follows: let $b_{0,0}(\emptyset)=1$, $b_{i, 0}(\emptyset)=0$ for $i\geq 1$ and $b_{0,0}(s)=0$ for any non-null loop sequence $s$. If $k\geq 1$, then let $b_{i,k}(\emptyset)=0$ for all $i\geq 0$ and $b_{0,k}(s)=0$ for all $s\in \cs$. Suppose that $b_{i',k'}(s')$ has been defined for all triples $(s', i', k')$ such that either $k'<k$, or $k'=k$ and $i'<i$, or  $k'=k$, $i'=i$ and $\iota(s')<\iota(s)$. Then define
\begin{align}
b_{i,k}(s)=b_{i, k-1}(s)&+\dfrac{1}{|s|}\sum_{s' \in \ftw(s)} b_{i,k-1}(s')+\dfrac{1}{|s|}\sum_{s' \in \fst(s)} b_{i,k-2}(s')\notag \\
&+\dfrac{1}{|s|}\sum_{s' \in \fs(s)} b_{i,k}(s')+\dfrac{1}{|s|}\sum_{s'\in \fd(s)} b_{i-1,k}(s')\label{recurb}
\end{align}
with the convention that $b_{i,k}(s)=0$ for all $s$ if $i<0$ or $k<0$. Observe that by Lemma \ref{iota2} the fourth term on the right has already been defined at the time of defining $b_{i,k}(s)$.
% % % % % % % % % % % % % % % % % % % % % % % % % % % % % % % % % % % % % % % % % % % % % % % % % % % % % % % % % % % %
% % % % % % % % % % % % % % % % % % % % % % % % % % % % % % % % % % % % % % % % % % % % % % % % % % % % % % % % % % % %
\begin{lmm}\label{coeffb}
There is a constant $K(d)\geq 4$ such that if $s$ is a loop sequence with degree sequence $\delta = (\delta_1,\ldots,\delta_n)$, then
\begin{align}
0\leq b_{i,k}(s)\leq K(d)^{(5+k)i+\iota(\delta)}|\delta|^{2k}C_{\delta_1-1}C_{\delta_2-1}\cdots C_{\delta_n-1} \label{boundbij}
\end{align}
where $C_i$ is the $i^{\textup{th}}$ Catalan number. The product of Catalan numbers is interpreted as $1$ when~$s=\emptyset$.
\end{lmm}
\begin{proof}
The proof is similar to the proof of Lemma \ref{coeff}. The number $K=K(d)$ will be chosen at the end of the proof. Observe that $b_{i,k} \geq 0$ is obvious from the recursive definition. To prove the upper bound we will again use three-fold induction: first on $k$, then on $i$ and then on $\iota(s)$. 

The inequality \eqref{boundbij} is trivially true if $k=0$. Take $k\geq 1$ and suppose that the inequality holds for all triples $(s', i', k')$ such that $k'<k$. Since $b_{0,k}(s)=0$ for all $s\in \cs$, \eqref{boundbij} holds for $(s,0,k)$ for all $s$. 

Fix $i\geq 1$ and assume that \eqref{boundbij} holds for $(s, i', k)$ for all $i'<i$ and all $s\in \cs$. We will prove that it holds for $(s, i, k)$ by induction on $\iota(s)$. If $s=\emptyset$, then $b_{i,k}(s)=0$ and inequality is obviously true. Let $s\neq \emptyset$ and suppose that bound on $b_{i,k}(s')$ holds for all $s'$ with $\iota(s')<\iota(s)$. If $\delta=(\delta_1, \ldots, \delta_n)$ is the degree sequence of $s$, then by the induction hypothesis,
\begin{align}
b_{i, k-1}(s)&\leq  K^{(5+k-1)i+\iota(\delta)}|\delta|^{2(k-1)}C_{\delta_1-1}C_{\delta_2-1}\cdots C_{\delta_n-1}\notag\\
&\le K^{(5+k)i-1+\iota(\delta)}|\delta|^{2k}C_{\delta_1-1}C_{\delta_2-1}\cdots C_{\delta_n-1}\,. \label{bb0}
\end{align}
Let $\ftw_r(s)$ be the set of all loop sequences obtained by applying a twisting operation to the $r^{\textup{th}}$ component loop of $s$. Similarly define $\fs^-_r(s)$, $\fs^+_r(s)$, $\fd^-_r(s)$ and $\fd^+_r(s)$. By Lemma \ref{twist}, if $s' \in \ftw(s)$ then $|\delta(s')|\leq |\delta|$ and $|\iota(s')|\leq |\iota(\delta)|$. Thus by the induction hypothesis,
\begin{align}
\dfrac{1}{|s|}\sum_{s' \in \ftw(s)} b_{i,k-1}(s') &= \dfrac{1}{|s|}\sum_{r=1}^{n}\sum_{s'\in \ftw_r(s)} b_{i,k-1}(s')\notag\\
&\leq \dfrac{1}{|s|}\sum_{r=1}^{n}|\ftw_r(s)| K^{(5+k-1)i+\iota(\delta)}|\delta|^{2(k-1)}C_{\delta_1-1}C_{\delta_2-1}\cdots C_{\delta_n-1}\notag\\
&\leq \dfrac{1}{|s|}\sum_{r=1}^{n}\delta_r^2 K^{(5+k-1)i+\iota(\delta)}|\delta|^{2(k-1)}C_{\delta_1-1}C_{\delta_2-1}\cdots C_{\delta_n-1}\notag\\
&\leq K^{(5+k)i-1+\iota(\delta)}|\delta|^{2k}C_{\delta_1-1}C_{\delta_2-1}\cdots C_{\delta_n-1}\,.\label{bb1}
\end{align}
By Lemma \ref{merger}, if $s'\in \fst(s)$ then $|\delta(s')|\leq |\delta|$ and $\iota(s')<\iota(\delta)+1$. So by the induction hypothesis,
\begin{align}
&\dfrac{1}{|s|}\sum_{s'\in\fst^{-}(s)}b_{i,k-2}(s')\notag\\
&=\dfrac{1}{|s|}\sum_{1\leq u<v\leq n}\sum_{x\in C_{u}, \, y\in C_{v}}b_{i,k-2}(l_1, \dots, l_{u-1}, l_{u} \ominus_{x,y} l_{v},l_{u+1},\dots, l_{v-1}, l_{v+1}, \dots,l_n)\notag \\
&\qquad +\dfrac{1}{|s|}\sum_{1\leq v<u\leq n}\sum_{x\in C_{u}, \, y\in C_{v}}b_{i,k-2}(l_1, \dots, l_{v-1},l_{v+1},\dots, l_{u-1}, l_{u} \ominus_{x,y} l_{v}, l_{u+1}, \dots,l_n) \notag\\
&\leq \dfrac{2}{|s|}\sum_{1\leq u<v\leq n}\sum_{x\in C_{u}, \, y\in C_{v}} K^{(5+k-2)i+\iota(\delta)+1}|\delta|^{2(k-2)}C_{\delta_u+\delta_v-1} \prod_{\substack{1\leq i\leq n\\
i \not\in \{u, v\}}} C_{\delta_{i}-1}\notag
\end{align}
Applying Lemma \ref{app1}, this gives
\begin{align}
&\dfrac{1}{|s|}\sum_{s'\in\fst^{-}(s)}b_{i,k-2}(s')\notag\\
&\leq \dfrac{2}{|s|}\sum_{1\leq u<v\leq n}\sum_{x\in C_{u}, \, y\in C_{v}} K^{(5+k)i-1+\iota(\delta)}|\delta|^{2k-4}(\delta_u+\delta_v)^2C_{\delta_u-1}C_{\delta_v-1} \prod_{\substack{1\leq i\leq n\\
i \not\in \{u, v\}}} C_{\delta_{i}-1}\notag\\
&\leq \dfrac{2}{|s|}\sum_{1\leq u<v\leq n} \delta_u\delta_v(\delta_u+\delta_v)^2 K^{(5+k)i-1+\iota(\delta)}|\delta|^{2k-4}\prod_{1\leq i\leq n} C_{\delta_{i}-1}\notag\\
&\leq K^{(5+k)i-1+\iota(\delta)}|\delta|^{2k}C_{\delta_1-1}\cdots  C_{\delta_n-1}\,.\notag
\end{align}
The same bound also holds for the sum over positive mergers. Therefore,
\begin{align}
\sum_{s' \in \fst(s)}b_{i,k-2}(s')\leq 2K^{(5+k)i-1+\iota(\delta)}|\delta|^{2k}C_{\delta_1-1}C_{\delta_2-1}\cdots C_{\delta_n-1}\,.\label{bb2}
\end{align}
By the induction hypothesis and Lemma \ref{split2}, 
\begin{align}
\dfrac{1}{|s|}\sum_{s'\in \fs^{-}(s)}b_{i,k}(s')&=\frac{1}{|s|}\sum_{r=1}^{n}\sum_{s'\in \fs^{-}_r(s)} b_{i, k}(s') \nonumber\\
&\leq \frac{1}{|s|}\sum_{r=1}^{n}\sum_{1\leq x\ne y\leq \delta_r} b_{i, k}(l_1,\dots,l_{r-1},\times_{x,y}^1 l_r,\times_{x,y}^2l_r,l_{r+1},\ldots, l_n) \nonumber\\
&\leq \frac{1}{|s|}\sum_{r=1}^{n}\sum_{1\leq x\ne y\leq \delta_r} K^{(5+k)i+\iota(\delta)-3} |\delta|^{2k} C_{\delta_r - |x-y|-2} C_{|x-y|-2}\prod_{\substack{1\leq i\leq n\\i\ne r}}C_{\delta_i-1}\nonumber\\
&\leq \frac{2}{|s|}\sum_{r=1}^{n}\sum_{x=1}^{\delta_r}\sum_{u=2}^{\delta_r-2} K^{(5+k)i+\iota(\delta)-3} |\delta|^{2k} C_{\delta_r -u-2} C_{u-2}\prod_{\substack{1\leq i\leq n\\i\ne r}}C_{\delta_i-1}\,.\nonumber
\end{align}
Using \eqref{catalan}, we get
\begin{align}
\dfrac{1}{|s|}\sum_{s'\in \fs^{-}(s)}b_{i,k}(s') &\leq \frac{1}{|s|}\sum_{r=1}^{n}\sum_{x=1}^{\delta_r}K^{(5+k)i+\iota(\delta)-3} |\delta|^{2k} C_{\delta_r -3}\prod_{\substack{1\leq i\leq n\\i\ne r}}C_{\delta_i-1}\nonumber\\
&\leq \frac{2}{|s|}\sum_{r=1}^{n}\delta_r K^{(5+k)i+\iota(\delta)-3} |\delta|^{2k} \prod_{1\leq i\leq n}C_{\delta_i-1}\nonumber\\
&=2K^{(5+k)i+\iota(\delta)-3} |\delta|^{2k} C_{\delta_1-1}\cdots C_{\delta_n-1}\, .\label{bb3a}
\end{align}
Similarly, by the induction hypothesis and Lemma \ref{split1}, 
\begin{align}
\dfrac{1}{|s|}\sum_{s'\in \fs^{+}(s)}b_{i,k}(s')&=\frac{1}{|s|}\sum_{r=1}^{n}\sum_{s'\in \fs^{+}_r(s)} b_{i, k}(s') \nonumber\\
&\leq \frac{1}{|s|}\sum_{r=1}^{n}\sum_{1\leq x\ne y\leq \delta_r} b_{i, k}(l_1,\dots,l_{r-1},\times_{x,y}^1 l_r,\times_{x,y}^2l_r,l_{r+1},\ldots, l_n) \nonumber\\
&\leq \frac{1}{|s|}\sum_{r=1}^{n}\sum_{1\leq x\ne y\leq \delta_r} K^{(5+k)i+\iota(\delta)-1} |\delta|^{2k} C_{\delta_r - |x-y|-1} C_{|x-y|-1}\prod_{\substack{1\leq i\leq n\\i\ne r}}C_{\delta_i-1}\nonumber\\
&\leq \frac{2}{|s|}\sum_{r=1}^{n}\sum_{x=1}^{\delta_r}\sum_{u=1}^{\delta_r-1} K^{(5+k)i+\iota(\delta)-1} |\delta|^{2k} C_{\delta_r-u-1} C_{u-1}\prod_{\substack{1\leq i\leq n\\i\ne r}}C_{\delta_i-1}\,.\nonumber
\end{align}
Again, applying \eqref{catalan}, we get
\begin{align}
\dfrac{1}{|s|}\sum_{s'\in \fs^{+}(s)}b_{i,k}(s') &\leq \frac{2}{|s|}\sum_{r=1}^{n}\sum_{x=1}^{\delta_r}K^{(5+k)i+\iota(\delta)-1} |\delta|^{2k} C_{\delta_r -1}\prod_{\substack{1\leq i\leq n\\i\ne r}}C_{\delta_i-1}\nonumber\\
&\leq \frac{2}{|s|}\sum_{r=1}^{n}\delta_r K^{(5+k)i+\iota(\delta)-1} |\delta|^{2k} \prod_{1\leq i\leq n}C_{\delta_i-1}\nonumber\\
&=2K^{(5+k)i+\iota(\delta)-1} |\delta|^{2k} C_{\delta_1-1}\cdots C_{\delta_n-1}\, .\label{bb3b}
\end{align}
By combining \eqref{bb3a} and \eqref{bb3b} we get
\begin{align}
\sum_{s'\in \fs(s)} b_{i,k}(s') \leq (2K^{-3}+2K^{-1}) K^{(5+k)i+\iota(\delta)} |\delta|^{2k} C_{\delta_1-1}\cdots C_{\delta_n-1}\, .\label{bb3}
\end{align}
Take any $1\le r\le n$, any location $x$ in $l_r$, and any plaquette $p$ that can be merged with $l_r$ at location $x$. Then by Lemma \ref{deform} and the induction hypothesis,
\begin{align*}
&b_{i-1,k}(l_1,\dots, l_{r-1}, l_r\oplus_{x} p , l_{r+1},\ldots, l_n) \\
&\le K^{(5+k)(i-1)+\iota(\delta)+4} (|\delta|+4)^{2k} C_{\delta_r+3}\prod_{\substack{1\leq i\leq n\\i\neq r}} C_{\delta_i-1}\,.
\end{align*}
 Since each edge can be deformed by $2(d-1)$ plaquettes, this gives
\begin{align*}
\dfrac{1}{|s|}\sum_{s'\in \fd^+(s)}b_{i-1,k}(s')&=\dfrac{1}{|s|}\sum_{r=1}^{n}\sum_{s'\in \fd^+_r(s)}b_{i-1,k}(s')\\
%&=\dfrac{1}{|s|}\sum_{r=1}^{n}\sum_{x=1}^{\delta_r}b_{i-1,k}(l_1,\dots, l_{r-1}, l_r\oplus_{x} p , l_{r+1},\ldots, l_n)|\\
&\le \dfrac{2d}{|s|}\sum_{r=1}^{n}\delta_r K^{(5+k)(i-1)+\iota(\delta)+4} (|\delta|+4)^{2k} C_{\delta_r+3}\prod_{\substack{1\leq i\leq n\\i\neq r}} C_{\delta_i-1}\,.
\end{align*}
Applying the inequality \eqref{catalan2} and the facts that $|\delta|\ge 4$ and $K\ge 4$, we get
\begin{align*}
\dfrac{1}{|s|}\sum_{s'\in \fd^+(s)}b_{i-1,k}(s') &\le \dfrac{2d}{|s|}\sum_{r=1}^{n}\delta_r K^{(5+k)i-1-k+\iota(\delta)}(2|\delta|)^{2k}4^{4}C_{\delta_r-1}\prod_{\substack{1\leq i\leq n\\i\neq r}} C_{\delta_i-1} \nonumber\\
&\le \dfrac{512 d}{|s|}\sum_{r=1}^{n}\delta_r K^{(5+k)i-1+\iota(\delta)}|\delta|^{2k}\prod_{1\leq i\leq n} C_{\delta_i-1} \nonumber\\
&= 512d K^{(5+k)i-1+\iota(\delta)}|\delta|^{2k}C_{\delta_1-1}\cdots C_{\delta_n-1}\, .
\end{align*}
The same bound holds if $\oplus$ is replaced by $\ominus$. Therefore,
\begin{align}
\dfrac{1}{|s|}\sum_{s'\in \fd(s)} b_{i-1,k}(s')\leq 1024d K^{(5+k)i-1+\iota(\delta)} |\delta|^{2k} C_{\delta_1-1}\cdots C_{\delta_n-1}\, .\label{bb4}
\end{align}
By using the inequalities \eqref{bb0}, \eqref{bb1}, \eqref{bb2}, \eqref{bb3} and \eqref{bb4} in the definition \eqref{coeffb} of $b_{i,k}(s)$, we get
\begin{align*}
b_{i,k}(s)\leq (2K^{-3}+6K^{-1}+1024dK^{-1})K^{(5+k)i+\iota(\delta)}|\delta|^{2k}C_{\delta_1-1}\cdots C_{\delta_n-1}\, .
\end{align*}
The induction step is completed by choosing $K$ so large that the number in the bracket is $\leq 1$.
\end{proof}
% % % % % % % % % % % % % % % % % % % % % % % % % % % % % % % % % % % % % % % % % % % % % % % % % % % % % % % % % % % % % % % % % % % % % % % % % % % % % % % % % % % % % % % % % % % % % % % % % % % % % % % % % % % % % % % % % % % % % % % % % % % % % % % % % % % % % % % % % % % % % % % % % % % % % % % % % %
\begin{lmm}\label{finite}
For any non-null loop sequence $s$ and nonnegative integers $i, a, b, c$, the set of trajectories $\mx_{i,a, b, c}(s)$ is a finite set. In particular, $\mx_{i, k}(s)$ is finite for any nonnegative $i$ and $k$.
\end{lmm}
\begin{proof}
Note that $i$ deformations, $a$ twistings, $b$ mergers and $c$ inactions can increase $\iota(s)$ by at most $4i+b$ (by Lemmas \ref{twist},  \ref{merger} and \ref{deform}). Since splitting reduces the index by at least one (Lemma \ref{iota2}), each $X\in\mx_{i, a, b, c}(s)$ can have at most $\iota(s)+4i+b$ splitting operations. Therefore $\mx_{i, a, b, c}(s)$ is finite. Since there are only finitely many nonnegative integers $a$, $b$, $c$ such that $a+2b+c=k$ we deduce that $\mx_{i,k}(s)$ is also finite.
\end{proof}
% % % % % % % % % % % % % % % % % % % % % % % % % % % % % % % % % % % % % % % % % % % % % % % % % %
% % % % % % % % % % % % % % % % % % % % % % % % % % % % % % % % % % % % % % % % % % % % % % % % % % 
For each $\beta$, $i$, $k$ and $s$, define
\begin{align*}
S_{\beta, i, k}(s)&:=\sum_{X\in \mx_{i, k}(s)}|w_{\beta}(X)|
\end{align*}
Note that $S_{\beta, i, k}$ is well-defined since $\mx_{i, k}(s)$ is  a finite set by the above lemma. 
\begin{lmm}\label{ansconvb}
If $|\beta|$ is sufficiently small (depending only on $d$ and $k$), then 
\begin{align}
S_{\beta, i, k}(s)&=b_{i,k}(s)|\beta|^i\,. \label{sbik}
\end{align}
\end{lmm}
\begin{proof}
For $k=0$, this result was proved in \cite{chatterjee15}. Take $k\geq 1$ and a triple $(s,i,k)$. Suppose that \eqref{sbik} holds for all $(s', i', k')$ with $k'<k$. We will show that it holds for $(s,i,k)$. 

First, suppose that $i=0$. We will use induction on $\iota(s)$ to prove the claim for the triple $(s,0,k)$. If $s=\emptyset$ then both sides are zero. Take some $s\neq \emptyset$ and suppose that \eqref{sbik} holds for all $(s', 0, k)$ with $\iota(s')<\iota(s)$. Note that any $X\in \mx_{0, k}(s)$ can be written as $X=(s, X')$ where either $X' \in \mx_{0, k-1}(s)$, or $X' \in \mx_{0, k-1}(s')$ for some $s'\in \ftw(s)$, or $X' \in \mx_{0, k-2}(s')$ for some $s'\in \fst(s)$, or $X' \in \mx_{0, k}(s')$ for some $s'\in \fs(s)$. Therefore 
\begin{align*}
S_{\beta, 0, k}(s)
&=\sum_{\substack{X=(s,X')\\X'\in \mx_{0, k-1}(s)}}|w_{\beta}(X)|+\sum_{s' \in \ftw(s)}\sum_{\substack{X=(s,X')\\X'\in \mx_{0, k-1}(s')}}|w_{\beta}(X)|\\
&\qquad+\sum_{s' \in \fst(s)}\sum_{\substack{X=(s,X')\\X'\in \mx_{0, k-2}(s')}}|w_{\beta}(X)|+\sum_{s' \in \fs(s)}\sum_{\substack{X=(s,X')\\X'\in \mx_{0, k}(s')}}|w_{\beta}(X)|\,.
\end{align*}
By the definition of $w_\beta(X)$,
\begin{align*}
\sum_{\substack{X=(s,X')\\X'\in \mx_{0, k-1}(s)}}|w_{\beta}(X)|&=\sum_{X'\in \mx_{0, k-1}(s)}|w_{\beta}(X')|\,.
\end{align*}
Similarly, 
\begin{align*}
\sum_{s' \in \ftw(s)}\sum_{\substack{X=(s,X')\\X'\in \mx_{0, k-1}(s')}}|w_{\beta}(X)|&=\dfrac{1}{|s|}\sum_{s' \in \ftw(s)}\sum_{X'\in \mx_{0, k-1}(s')}|w_{\beta}(X')|\, ,
\end{align*}
\begin{align*}
\sum_{s' \in \fst(s)}\sum_{\substack{X=(s,X')\\X'\in \mx_{0, k-2}(s')}}|w_{\beta}(X)|&=\dfrac{1}{|s|}\sum_{s' \in \fst(s)}\sum_{X'\in \mx_{0, k-2}(s')}|w_{\beta}(X')|\, ,
\end{align*}
and
\begin{align*}
\sum_{s' \in \fs(s)}\sum_{\substack{X=(s,X')\\X'\in \mx_{0, k}(s')}}|w_{\beta}(X)|&=\dfrac{1}{|s|}\sum_{s' \in \fs(s)}\sum_{X'\in \mx_{0, k}(s')}|w_{\beta}(X')|\,.
\end{align*}
Putting together the last five displays and applying the induction hypothesis, we get
\begin{align*}
S_{\beta, 0, k}(s)&=S_{\beta, 0, k-1}(s)+\dfrac{1}{|s|}\sum_{s' \in \ftw(s)}S_{\beta, 0, k-1}(s')+\dfrac{1}{|s|}\sum_{s' \in \fst(s)}S_{\beta, 0, k-2}(s')+\dfrac{1}{|s|}\sum_{s' \in \fs(s)}S_{\beta, 0, k}(s')\\
&=b_{0, k-1}(s)+\dfrac{1}{|s|}\sum_{s' \in \ftw(s)}b_{0, k-1}(s')+\dfrac{1}{|s|}\sum_{s' \in \fst(s)}b_{0, k-2}(s')+\dfrac{1}{|s|}\sum_{s' \in \fs(s)}b_{0, k}(s')\\
&=b_{0,k}(s)\,.
\end{align*}
This completes the proof for triples of the form $(s, 0,k)$. Fixing $k$ as before, take some $i\geq 1$ and suppose that \eqref{sbik} holds for all $(s', i', k')$ where either $k'<k$, or $k'=k$ and $i'<i$. We will use induction on $\iota(s)$ to show that it also holds for $(s, i,k)$ for all $s$. 

If $s=\emptyset$ then both sides are zero. Take some $s\neq \emptyset$ and suppose that \eqref{sbik} holds for all $(s', i, k)$ with $\iota(s')<\iota(s)$. Note that any $X \in \mx_{i, k}(s)$ can be written as $(s, X')$ where either $X' \in \mx_{i, k-1}(s)$, or $X' \in  \mx_{i, k-1}(s')$ for some $s'\in \ftw(s)$, or $X' \in  \mx_{i, k-2}(s')$ for some $s'\in \fst(s)$, or $X' \in  \mx_{i, k}(s')$ for some $s'\in \fs(s)$, or $X' \in  \mx_{i-1, k}(s')$ for some $s'\in \fd(s)$. Therefore
\begin{align*}
S_{\beta, i, k}(s) &=\sum_{\substack{X=(s,X')\\X'\in \mx_{i, k-1}(s)}}|w_{\beta}(X)|+\sum_{s' \in \ftw(s)}\sum_{\substack{X=(s,X')\\X'\in \mx_{i, k-1}(s')}}|w_{\beta}(X)|+\sum_{s' \in \fst(s)}\sum_{\substack{X=(s,X')\\X'\in \mx_{i, k-2}(s')}}|w_{\beta}(X)| \notag \\
&\qquad+\sum_{s' \in \fs(s)}\sum_{\substack{X=(s,X')\\X'\in \mx_{i, k}(s')}}|w_{\beta}(X)|+\sum_{s' \in \fd(s)}\sum_{\substack{X=(s,X')\\X'\in \mx_{i-1, k}(s')}}|w_{\beta}(X)|\,. \notag 
\end{align*}
By the definition of $w_\beta(X)$,
\begin{align*}
\sum_{\substack{X=(s,X')\\X'\in \mx_{i, k-1}(s)}}|w_{\beta}(X)|&=\sum_{X'\in \mx_{i, k-1}(s)}|w_{\beta}(X')|\,.
\end{align*}
Similarly,
\begin{align*}
\sum_{s' \in \ftw(s)}\sum_{\substack{X=(s,X')\\X'\in \mx_{i, k-1}(s')}}|w_{\beta}(X)|&=\dfrac{1}{|s|}\sum_{s' \in \ftw(s)}\sum_{X'\in \mx_{i, k-1}(s')}|w_{\beta}(X')|\,,
\end{align*}
\begin{align*}
\sum_{s' \in \fst(s)}\sum_{\substack{X=(s,X')\\X'\in \mx_{i, k-2}(s')}}|w_{\beta}(X)|&=\dfrac{1}{|s|}\sum_{s' \in \fst(s)}\sum_{X'\in \mx_{i, k-2}(s')}|w_{\beta}(X')|\,,
\end{align*}
\begin{align*}
\sum_{s' \in \fs(s)}\sum_{\substack{X=(s,X')\\X'\in \mx_{i, k}(s')}}|w_{\beta}(X)|&=\dfrac{1}{|s|}\sum_{s' \in \fs(s)}\sum_{X'\in \mx_{i, k}(s')}|w_{\beta}(X')|\,,
\end{align*}
and
\begin{align*}
\sum_{s' \in \fd(s)}\sum_{\substack{X=(s,X')\\X'\in \mx_{i-1, k}(s')}}|w_{\beta}(X)|&=\dfrac{|\beta|}{|s|}\sum_{s' \in \fd(s)}\sum_{X'\in \mx_{i-1, k}(s')}|w_{\beta}(X')|\,.
\end{align*}
Putting together the last six displays, we get
\begin{align*}
S_{\beta, i, k}(s) &= S_{\beta, i, k-1}(s)+\dfrac{1}{|s|}\sum_{s' \in \ftw(s)}S_{\beta, i, k-1}(s')+\dfrac{1}{|s|}\sum_{s' \in \fst(s)}S_{\beta, i, k-2}(s')\\
&\qquad +\dfrac{1}{|s|}\sum_{s' \in \fs(s)}S_{\beta, i, k}(s')+\dfrac{|\beta|}{|s|}\sum_{s' \in \fd(s)}S_{\beta, i-1, k}(s')\,.
\end{align*}
Applying the induction hypothesis to the terms on the right gives
\begin{align*}
S_{\beta, i, k}(s)
&=b_{i, k-1}(s)|\beta|^i+\dfrac{1}{|s|}\sum_{s' \in \ftw(s)}b_{i, k-1}(s')|\beta|^i+\dfrac{1}{|s|}\sum_{s' \in \fst(s)}b_{i, k-2}(s')|\beta|^{i}\\
&\qquad +\dfrac{1}{|s|}\sum_{s' \in \fs(s)}b_{i, k}(s')|\beta|^i+\dfrac{|\beta|}{|s|}\sum_{s' \in \fd(s)}b_{i-1, k}(s')|\beta|^{i-1}\\
&=b_{i,k}(s)|\beta|^{i}\,.
\end{align*}
This completes induction.
\end{proof}
Lemma \ref{ansconvb} easily implies Theorem \ref{absconv}.
\begin{proof}[Proof of Theorem \ref{absconv}]
Since any $X\in \mx_{k}(s)$ is in exactly one of $\mx_{i,k}(s)$ for some $i\ge 0$,
\begin{align*}
\sum_{X\in \mx_k(s)}|w_{\beta}(X)|=\sum_{i=0}^{\infty}\sum_{X\in \mx_{i, k}(s)}|w_{\beta}(X)|=\sum_{i=0}^{\infty} S_{\beta, i, k}=\sum_{i=0}^{\infty}b_{i, k}|\beta|^{i}\,.
\end{align*}
By Lemma \ref{coeffb} this sum is convergent for all sufficiently small $|\beta|$ (depending only on $d$ and $k$).
\end{proof}
% % % % % % % % % % % % % % % % % % % % % % % % % % % % % % % % % % % % % % % % % % % % % % % % % % % % % % % % % % % %
% % % % % % % % % % % % % % % % % % % % % % % % % % % % % % % % % % % % % % % % % % % % % % % % % % % % % % % % % % % % 
% % % % % % % % % % % % % % % % % % % % % % % % % % % % % % % % % % % % % % % % % % % % % % % % % % % % % % % % % % % %
%\vskip.5in
\section{Gauge-string duality}
We are  now ready to complete the proof of part (i) of Theorem \ref{mainthmofpaper}. The absolute convergence claim has already been proved in Theorem \ref{absconv}. Let
\begin{align*}
T_{\beta, i, k}(s):=\sum_{X\in \mx_{i, k}(s)}w_{\beta}(X)\,.
\end{align*}
By Lemma \ref{finite} the above sum has finite number of terms and therefore $T_{\beta, i, k}$ is well-defined. By Theorem \ref{series}, it suffices to prove that for sufficiently small $|\beta|$ (depending only on $d$ and $k$),
\begin{align}
T_{\beta, i, k}(s)&=a_{i,k}(s)\beta^i\,. \label{enough}
\end{align}
The base case $k=0$ was established in \cite{chatterjee15}. Take some $k\geq 1$ and suppose that~\eqref{enough} holds for all triples $(s', i', k')$ with $k'<k$. First we will use induction on $\iota(s)$ to prove the claim for triples of the form $(s, 0,k)$. If $s=\emptyset$, then both sides of \eqref{enough} are zero. Take some $s\neq \emptyset$ and suppose the claim holds for all triples $(s', 0, k)$ with $\iota(s')<\iota(s)$. Recall that any $X\in \mx_{0, k}(s)$ can be written as  $(s, X')$ where either $X'\in  \mx_{0, k-1}(s)$, or $X' \in  \mx_{0, k-1}(s')$ for some $s'\in \ftw(s)$, or $X' \in  \mx_{0, k-2}(s')$ for some $s'\in \fst(s)$, or $X' \in  \mx_{0, k}(s')$ for some $s'\in \fs(s)$.  Therefore
\begin{align*}
T_{\beta, 0, k}(s)
&=\sum_{\substack{X=(s,X')\\X'\in \mx_{0, k-1}(s)}}w_{\beta}(X)+\sum_{s' \in \ftw(s)}\sum_{\substack{X=(s,X')\\X'\in \mx_{0, k-1}(s')}}w_{\beta}(X)\\
&\quad+\sum_{s' \in \fst(s)}\sum_{\substack{X=(s,X')\\X'\in \mx_{0, k-2}(s')}}w_{\beta}(X)+\sum_{s' \in \fs(s)}\sum_{\substack{X=(s,X')\\X'\in \mx_{0, k}(s')}}w_{\beta}(X)\,. \notag
\end{align*}
Applying the definition of $w_\beta(X)$, we get
\begin{align*}
\sum_{\substack{X=(s,X')\\X'\in \mx_{0, k-1}(s)}}w_{\beta}(X)&=\sum_{X'\in \mx_{0, k-1}(s)}w_{\beta}(X')\, .
\end{align*}
Similarly,
\begin{align*}
\sum_{s' \in \ftw(s)}\sum_{\substack{X=(s,X')\\X'\in \mx_{0, k-1}(s')}}w_{\beta}(X)&=\dfrac{1}{|s|}\sum_{s' \in \ftw^-(s)}\sum_{X'\in \mx_{0, k-1}(s')}w_{\beta}(X')- \dfrac{1}{|s|}\sum_{s' \in \ftw^+(s)}\sum_{X'\in \mx_{0, k-1}(s')}w_{\beta}(X')\, ,
\end{align*}
\begin{align*}
\sum_{s' \in \fst(s)}\sum_{\substack{X=(s,X')\\X'\in \mx_{0, k-2}(s')}}w_{\beta}(X)&=\dfrac{1}{|s|}\sum_{s' \in \fst^-(s)}\sum_{X'\in \mx_{0, k-2}(s')}w_{\beta}(X') - \dfrac{1}{|s|}\sum_{s' \in \fst^+(s)}\sum_{X'\in \mx_{0, k-2}(s')}w_{\beta}(X')
\end{align*}
and
\begin{align*}
\sum_{s' \in \fs(s)}\sum_{\substack{X=(s,X')\\X'\in \mx_{0, k}(s')}}w_{\beta}(X)&=\dfrac{1}{|s|}\sum_{s' \in \fs^-(s)}\sum_{X'\in \mx_{0, k}(s')}w_{\beta}(X') - \dfrac{1}{|s|}\sum_{s' \in \fs^+(s)}\sum_{X'\in \mx_{0, k}(s')}w_{\beta}(X')\,.
\end{align*}
Putting together the last five displays gives 
\begin{align*}
T_{\beta, 0, k}(s) &=T_{\beta, 0, k-1}(s)+\dfrac{1}{|s|}\sum_{s' \in \ftw^-(s)}T_{\beta, 0, k-1}(s')- \dfrac{1}{|s|}\sum_{s' \in \ftw^+(s)}T_{\beta, 0, k-1}(s')\\
&\qquad + \dfrac{1}{|s|}\sum_{s' \in \fst^-(s)}T_{\beta, 0, k-2}(s')- \dfrac{1}{|s|}\sum_{s' \in \fst^+(s)}T_{\beta, 0, k-2}(s')\\
&\qquad +\dfrac{1}{|s|}\sum_{s' \in \fs^-(s)}T_{\beta, 0, k}(s')-\dfrac{1}{|s|}\sum_{s' \in \fs^+(s)}T_{\beta, 0, k}(s')\,.
\end{align*}
Applying the induction hypothesis to the right side and using Corollary \ref{symmcor}, we get
\begin{align*}
T_{\beta, 0, k}(s)&=a_{0, k-1}(s)+\dfrac{2}{|s|}\sum_{s' \in \ftw^-(s)}a_{0, k-1}(s')- \dfrac{2}{|s|}\sum_{s' \in \ftw^+(s)}a_{0, k-1}(s')\\
&\qquad +\dfrac{1}{|s|}\sum_{s' \in \fst^-(s)}a_{0, k-2}(s')- \dfrac{1}{|s|}\sum_{s' \in \fst^+(s)}a_{0, k-2}(s')\\
&\qquad +\dfrac{2}{|s|}\sum_{s' \in \fs^-(s)}a_{0, k}(s') - \dfrac{2}{|s|}\sum_{s' \in \fs^+(s)}a_{0, k}(s')\\
&=a_{0,k}(s)\,.
\end{align*}
This completes the proof for triples of the form $(s,0,k)$. Fixing $k$ as before, take some $i\geq 1$ and suppose that \eqref{enough} holds for all triples $(s', i',k')$ where either $k'<k$, or $k'=k$ and $i'<i$. We will use induction on $\iota(s)$ to prove that \eqref{enough} holds for $(s, i, k)$ for all $s$. If $s=\emptyset$, then both sides are zero. Take some $s\neq \emptyset$ and suppose that \eqref{enough} also holds for all $(s', i ,k)$ such that $\iota(s')<\iota(s)$. Note that any $X \in \mx_{i, k}(s)$ can be written as $(s, X')$ where either $X' \in  \mx_{i, k-1}(s)$, or $X' \in  \mx_{i, k-1}(s')$ for some $s'\in \ftw(s)$, or $X' \in  \mx_{i, k-2}(s')$ for some $s'\in \fst(s)$, or $X' \in  \mx_{i, k}(s')$ for some $s'\in \fs(s)$, or $X' \in  \mx_{i-1, k}(s')$ for some $s'\in \fd(s)$. Therefore
\begin{align*}
T_{\beta, i, k}(s)&=\sum_{\substack{X=(s,X')\\X'\in \mx_{i, k-1}(s)}}w_{\beta}(X)+\sum_{s' \in \ftw(s)}\sum_{\substack{X=(s,X')\\X'\in \mx_{i, k-1}(s')}}w_{\beta}(X)+\sum_{s' \in \fst(s)}\sum_{\substack{X=(s,X')\\X'\in \mx_{i, k-2}(s')}}w_{\beta}(X) \notag \\
&\qquad+\sum_{s' \in \fs(s)}\sum_{\substack{X=(s,X')\\X'\in \mx_{i, k}(s')}}w_{\beta}(X)+\sum_{s' \in \fd(s)}\sum_{\substack{X=(s,X')\\X'\in \mx_{i-1, k}(s')}}w_{\beta}(X)\,.
\end{align*}
By the definition of $w_\beta(X)$,
\begin{align*}
\sum_{\substack{X=(s,X')\\X'\in \mx_{i, k-1}(s)}}w_{\beta}(X)&=\sum_{X'\in \mx_{i, k-1}(s)}w_{\beta}(X')\, .
\end{align*} 
Similarly, 
\begin{align*}
\sum_{s' \in \ftw(s)}\sum_{\substack{X=(s,X')\\X'\in \mx_{i, k-1}(s')}}w_{\beta}(X)&=\dfrac{1}{|s|}\sum_{s' \in \ftw^-(s)}\sum_{X'\in \mx_{i, k-1}(s')}w_{\beta}(X') - \dfrac{1}{|s|}\sum_{s' \in \ftw^+(s)}\sum_{X'\in \mx_{i, k-1}(s')}w_{\beta}(X')\, ,
\end{align*}
\begin{align*}
\sum_{s' \in \fst(s)}\sum_{\substack{X=(s,X')\\X'\in \mx_{i, k-2}(s')}}w_{\beta}(X)&=\dfrac{1}{|s|}\sum_{s' \in \fst^-(s)}\sum_{X'\in \mx_{i, k-2}(s')}w_{\beta}(X') - \dfrac{1}{|s|}\sum_{s' \in \fst^+(s)}\sum_{X'\in \mx_{i, k-2}(s')}w_{\beta}(X')\,,
\end{align*}
\begin{align*}
\sum_{s' \in \fs(s)}\sum_{\substack{X=(s,X')\\X'\in \mx_{i, k}(s')}}w_{\beta}(X)&=\dfrac{1}{|s|}\sum_{s' \in \fs^-(s)}\sum_{X'\in \mx_{i, k}(s')}w_{\beta}(X') -\dfrac{1}{|s|}\sum_{s' \in \fs^+(s)}\sum_{X'\in \mx_{i, k}(s')}w_{\beta}(X')\,,
\end{align*}
and
\begin{align*}
\sum_{s' \in \fd(s)}\sum_{\substack{X=(s,X')\\X'\in \mx_{i-1, k}(s')}}w_{\beta}(X)&=\dfrac{\beta}{|s|}\sum_{s' \in \fd^-(s)}\sum_{X'\in \mx_{i-1, k}(s')}w_{\beta}(X') - \dfrac{\beta}{|s|}\sum_{s' \in \fd^+(s)}\sum_{X'\in \mx_{i-1, k}(s')}w_{\beta}(X')\,.
\end{align*}
Combining the last six displays  gives
\begin{align*}
T_{\beta, i, k}(s)&=T_{\beta, i, k-1}(s)+\dfrac{1}{|s|}\sum_{s' \in \ftw^-(s)}T_{\beta, i, k-1}(s') - \dfrac{1}{|s|}\sum_{s' \in \ftw^+(s)}T_{\beta, i, k-1}(s')\\
&\qquad +\dfrac{1}{|s|}\sum_{s' \in \fst^-(s)}T_{\beta, i, k-2}(s') - \dfrac{1}{|s|}\sum_{s' \in \fst^+(s)}T_{\beta, i, k-2}(s')\\
&\qquad +\dfrac{1}{|s|}\sum_{s' \in \fs^-(s)} T_{\beta, i, k}(s') - \dfrac{1}{|s|}\sum_{s' \in \fs^+(s)} T_{\beta, i, k}(s')\\
&\qquad +\dfrac{\beta}{|s|}\sum_{s' \in \fd^-(s)} T_{\beta, i-1, k}(s')- \dfrac{\beta}{|s|}\sum_{s' \in \fd^+(s)} T_{\beta, i-1, k}(s')\,.
\end{align*}
Applying the induction hypothesis to the right side, and using Corollary \ref{symmcor}, we get
\begin{align*}
T_{\beta, i, k}(s)&=a_{i, k-1}(s)\beta^i+\dfrac{1}{|s|}\sum_{s' \in \ftw^-(s)}a_{i, k-1}(s')\beta^i- \dfrac{1}{|s|}\sum_{s' \in \ftw^+(s)}a_{i, k-1}(s')\beta^i\\
&\qquad +\dfrac{1}{|s|}\sum_{s' \in \fst^-(s)}a_{i, k-2}(s')\beta^{i}-\dfrac{1}{|s|}\sum_{s' \in \fst^+(s)}a_{i, k-2}(s')\beta^{i}\\
&\qquad+\dfrac{1}{|s|}\sum_{s' \in \fs^-(s)} a_{i, k}(s')\beta^i-\dfrac{1}{|s|}\sum_{s' \in \fs^+(s)} a_{i, k}(s')\beta^i\\
&\qquad+\dfrac{\beta}{|s|}\sum_{s' \in \fd^-(s)} a_{i-1, k}(s')\beta^{i-1}- \dfrac{\beta}{|s|}\sum_{s' \in \fd^+(s)} a_{i-1, k}(s')\beta^{i-1}\\
&=a_{i,k}(s)\beta^{i}\,.
\end{align*}
This completes the proof of part (i) of Theorem \ref{mainthmofpaper}.
% % % % % % % % % % % % % % % % % % % % % % % % % % % % % % % % % % % % % % % % % % % % % % % % % % % % % % % % % % % % % % % % % % % % % % % % % % % % % % % % % % % % % % % % % % % % % % % % % % % % % % % % % % % % % % % % % % % % % % % % % % % % % % % % % % % % % % % % % % % % % % % % % % % % % % % % % % 

% % % % % % % % % % % % % % % % % % % % % % % % % % % % % % % % % % % % % % % % % % % % % % % % % % % % % % % % % % %
% % % % % % % % % % % % % % % % % % % % % % % % % % % % % % % % % % % % % % % % % % % % % % % % % % % % % % % % % % % 
% % % % % % % % % % % % % % % % % % % % % % % % % % % % % % % % % % % % % % % % % % % % % % % % % % % % % % % % % % %
%\appendix
%\setcounter{section}{1}
%\setcounter{thm}{0}
%\setcounter{equation}{0}
%\section*{Appendix}

% % % % % % % % % % % % % % % % % % % % % % % % % % % % % % % % % % % % % % % % % % % % % % % % % % % % % % % % % % % % % % % % % % % % % % % % % % % % % % % % % % % % % % % % % % % % % % % % % % % % % % % % % % % % % % % % % % % % % % % % % % % % % % % % % % % % % % % % % % % % % % % % % % % % % % % % % % % % % % % % % % % % % % % % % % % % % % % % % % % % % % % % % % % % % % % % % % % % % % % % % % % % % % % % % % % % % % % % % % % % % % % % % % % % % % % 

%\vskip.5in
%\noindent{\bf Acknowledgments.} 

\section*{Acknowledgments}
The authors thank David Brydges and Steve Shenker for helpful comments. 
%\vskip.5in
%\bibliographystyle{plainnat}

\bibliographystyle{amsplain}

\end{document}